\documentclass[11pt]{article}

\usepackage{authblk}

\usepackage{amsmath}
\usepackage{amsthm}
\usepackage{amssymb}
\usepackage{xcolor}
\usepackage{url}
\usepackage[hidelinks]{hyperref}
\newcommand{\doilink}[1]{\href{https://doi.org/#1}{\nolinkurl{#1}}}
\newcommand{\weblink}[1]{\href{#1}{\nolinkurl{#1}}}
\usepackage{array}
\usepackage{tikz}
\usepackage[a4paper,margin=3cm]{geometry}
\theoremstyle{plain}
\newtheorem{theorem}{Theorem}[section]
\newtheorem{prop}[theorem]{Proposition}
\newtheorem{lemma}[theorem]{Lemma}
\newtheorem{cor}[theorem]{Corollary}
\newtheorem{problem}{Problem}[section]
\theoremstyle{remark}
\newtheorem{example}[theorem]{Example}
\theoremstyle{definition}
\newtheorem{definition}[theorem]{Definition}
\newtheorem*{quoteddefinition}{Definition}
\theoremstyle{plain}
\newtheorem*{quotedtheorem}{Theorem}

\newcommand{\byeqn}[1]{\overset{\eqref{#1}}{=}}

\newcommand{\idmap}{\mathrm{id}}
\newcommand{\inv}{^{-1}}

\newcommand\gd{\mathcal{D}}
\newcommand\gl{\mathcal{L}}
\newcommand\gr{\mathcal{R}}
\newcommand\gh{\mathcal{H}}
\newcommand\gj{\mathcal{J}}
\newcommand\gll{\mathrel{\gl}}
\newcommand\grr{\mathrel{\gr}}
\newcommand\gjj{\mathrel{\gj}}

\newcommand{\End}{\mathop{\mathrm{End}}}
\newcommand{\Sym}{\mathop{\mathrm{Sym}}}
\newcommand{\Aut}{\mathop{\mathrm{Aut}}}

\DeclareMathOperator{\im}{im}
\DeclareMathOperator{\rank}{rank}
\DeclareMathOperator{\sgn}{sgn}
\DeclareMathOperator{\Tr}{Tr}

\title{Complete Mappings of Semigroups}

\author[1,2]{Jo\~ao Ara\'ujo\thanks{
\href{mailto:jj.araujo@fct.unl.pt}{\texttt{jj.araujo@fct.unl.pt}}}}
\author[1,3]{Wolfram Bentz\thanks{
\href{mailto:wbentz@uaberta.pt}{\texttt{wbentz@uaberta.pt}}}}
\author[4]{Peter J. Cameron\thanks{
\href{mailto:pjc20@st-andrews.ac.uk}{\texttt{pjc20@st-andrews.ac.uk}}}}
\author[5]{Kevin Hendrey\thanks{
\href{mailto:kevin.hendrey1@monash.edu}{\texttt{kevin.hendrey1@monash.edu}}}}
\author[6]{Michael Kinyon\thanks{
\href{mailto:michael.kinyon@du.edu}{\texttt{michael.kinyon@du.edu}}}}
\affil[1]{\small Centro de Matem\'{a}tica e Aplica\c{c}\~{o}es,
        Faculdade de Ci\^{e}ncias e Tecnologia,
        Universidade Nova de Lisboa,
        Campus da Caparica,
        2829-516 Caparica, Portugal}
\affil[2]{\small Departamento de Matem\'{a}tica,
        Faculdade de Ci\^{e}ncias e Tecnologia,
        Universidade Nova de Lisboa,
        Campus da Caparica,
        2829-516 Caparica, Portugal}
\affil[3]{\small Departamento de Ci\^{e}ncias e Tecnologia,
        Universidade Aberta,
        Rua Escola Polit\'{e}cnica, 147,
        1269-001 Lisboa, Portugal}
\affil[4]{\small School of Mathematics and Statistics,
        University of St Andrews,
        St Andrews KY16 9SS, UK}
\affil[5]{\small School of Mathematical Sciences,
        Monash University, VIC 3800, Australia}
\affil[6]{\small Department of Mathematics,
        University of Denver,
        Denver CO 80208, USA}

\date{}

\begin{document}

\maketitle

\begin{abstract} 
A complete mapping of a semigroup $S$ is a bijection $\alpha\colon S\to S$ such that the map
$\theta\colon S\to S$ defined by $x\theta=x\cdot x\alpha$ is also a bijection. Equivalently,
it determines a transversal of the multiplication table of $S$. Complete mappings connect group theory, Latin squares, and cryptography,
and their existence for finite groups was characterized by the resolution
of the Hall--Paige conjecture. In this paper, we develop the corresponding theory for
finite semigroups.

We prove that every finite semigroup admitting a complete mapping is regular and that
the problem reduces to principal factors. We classify the existence of a complete mapping in Rees matrix
semigroups without zero, give a Hall-type criterion for Rees $0$-matrix semigroups
over groups with complete mappings, and prove sufficient conditions
for Rees $0$-matrix semigroups whose maximal subgroups do not have complete mappings.
As the main application of the Rees $0$-matrix analysis, we show  that
        $T_n$ has a complete mapping if and only if
        $n=1$ or $n\ge4$. Equivalently, $T_n$ has a complete mapping if and only if the same holds for  $S_n$.
We prove that the full linear monoid of a finite-dimensional vector space has a complete mapping except in dimension $1$ over a field of odd order and in dimension $2$ over $\mathbb F_2$. We also prove that the partition monoid  $\mathcal P_n$ has a complete mapping if and only if $n=1$ or $n\ge4$, and that every finite aperiodic regular $*$-semigroup has a complete mapping. As a consequence, the planar partition, Motzkin and Jones monoids have complete mappings.
The proofs combine Green--Rees structure theory with the Hall--Paige
theorem and Burnside transfer, the Hall--Gale--Edmonds framework for
matchings, flows, matroids and polyhedra,
Hoffman--Kruskal total unimodularity and integral polyhedra, K\H{o}nig edge-colouring, Bevis--Hall--Katz incidence theory over finite abelian
groups, and Bregman--Egorychev--Falikman
permanent estimates.        
The paper concludes with open problems ranging from structural and
classification questions to potential applications in cryptography.

\end{abstract}

\section{Introduction}

Let $(S,\cdot)$ be a magma, that is, a set with a binary operation. A \emph{complete mapping} of $S$ is a bijection $\alpha\colon S\to S$ such that the mapping $\theta\colon S\to S$ defined by $x\theta = x\cdot x\alpha$ ($x\in S$), called an \emph{orthomorphism}, is also a bijection.
Here and throughout, our maps act on the right: $x\alpha$ means $\alpha(x)$.

\begingroup
Our aim is to investigate the existence of complete mappings in finite
semigroups. This roadmap is deliberately terse. Semigroup theory terms are used but not presupposed: each is defined in the main text before its use.
All semigroups, groups and index sets in the following statements are finite.
The main results are the following.
\begin{enumerate}\itemsep0pt
\item Every semigroup admitting a complete mapping is regular
(Theorem~\ref{Thm:regular}).

\item The existence problem reduces exactly to the principal factors: a
semigroup has a complete mapping if and only if every semigroup obtained from
one of its $\mathcal J$-classes by replacing products that leave the class by
zero has a complete mapping (Theorem~\ref{t:J-reduct}).
 Combining this new reduction with the standard principal-factor
form of Rees' theorem (Theorem~\ref{t:principal-factor-Rees}), we first note that each principal
factor is either null or a Rees $0$-matrix semigroup. A null principal factor
with at least two elements has no complete mapping, so the remaining
classification problem is precisely to determine which Rees $0$-matrix
semigroups have complete mappings.

\item Without loss of generality, let $P$ be a sandwich matrix whose first row and first column consist of
the identity. The Rees matrix semigroup $\mathcal M(G,I,\Lambda,P)$ has a
complete mapping if and only if at least one of the following holds: $|G|$ is
odd; the Sylow $2$-subgroups of $G$ are non-cyclic; $|I||\Lambda|$ is even; or
$P$ has an entry of even order (Theorem~\ref{t:ReesNon0}).

\item For a Rees $0$-matrix semigroup
$\mathcal M^0(G,I,\Lambda,P)$, let $Q$ record the non-zero positions of $P$.
Suppose that either $G$ has a complete mapping or $|I||\Lambda|$ is even.
Then $\mathcal M^0(G,I,\Lambda,P)$ has a complete mapping if and only if $Q$
admits a non-negative real matrix $H=(h_{\lambda i})$ such that
$h_{\lambda i}=0$ whenever $Q_{\lambda i}=0$, every row sum is $|I|$ and
every column sum is $|\Lambda|$. Equivalently, every $s$ rows of $Q$ meet at
least $s|I|/|\Lambda|$ columns, or dually every $r$ columns meet at least
$r|\Lambda|/|I|$ rows (Theorems~\ref{c:Rees0} and
\ref{t:Rees0-noncomp-even}). {In the square case
$|I|=|\Lambda|=n$, our condition implies that the matrix $H/n$ is doubly stochastic and vanishes at every
zero-position of $Q$; conversely, every such doubly stochastic matrix yields a
balanced support weighting by multiplication with $n$.}

\item If $V$ has finite dimension $d$ vector space over $\mathbb F_q$, then every proper
principal factor of $\End_{\mathbb F_q}(V)$ has a complete mapping, and
$\End_{\mathbb F_q}(V)$ has a complete mapping if and only if
$\operatorname{GL}_d(q)$ does. Equivalently, the only exceptions are $d=1$
with $q$ odd and $(d,q)=(2,2)$ (Theorem~\ref{t:linear-monoid}).

\item Every proper principal factor of the partition monoid $\mathcal P_n$
has a complete mapping, and $\mathcal P_n$ has a complete mapping if and only
if $n=1$ or $n\ge4$ (Theorem~\ref{t:partition-monoid}).

\item Every finite aperiodic regular $*$-semigroup has a complete mapping
(Proposition~\ref{p:aperiodic-regular-star}). Consequently, the planar
partition, Motzkin and Jones (also called Temperley--Lieb) monoids have complete mappings
(Corollary~\ref{c:planar-diagram-monoids}).

\item The full transformation monoid $T_n$ has a complete mapping if and only
if $n=1$ or $n\ge4$ (Theorem~\ref{Thm:full-transformation-semigroup}).

\item A finite inverse semigroup has a complete mapping if and only if, in
each $\mathcal J$-class, either its maximal subgroups have complete mappings
or the number of its $\mathcal L$-classes, equivalently its
$\mathcal R$-classes, is even (Theorem~\ref{t:inverse}).
\end{enumerate}
\endgroup

{Complete mappings for groups were introduced by H. B. Mann in his study of orthogonal Latin squares based on groups \cite{Mann1,Mann2}. Interest turned to the finite case after P. T. Bateman showed that every infinite group has a complete mapping \cite[Theorem]{Bateman}. Further work was done by L. J. Paige \cite{Paige} and then M. Hall and Paige \cite{hp}. It was in the latter paper that the characterization known as the \emph{Hall--Paige conjecture} was first formulated.}

{The completed Hall--Paige theorem is recorded in
\cite[Theorems~5 and~6]{evans4}.}

\begin{theorem}\label{Thm:HP}
A finite group $G$ has a complete mapping if and only if the Sylow $2$-subgroups of $G$ are either trivial or non-cyclic.
\end{theorem}

{\noindent Hall and Paige proved that a finite group with a non-trivial cyclic Sylow $2$-subgroup has no complete mapping; this is recorded as Theorem~5 in Evans's survey \cite[Theorem~5]{evans4}. The sufficiency was open until 2009, when the reduction and remaining cases were completed by S. Wilcox \cite[Theorem~24]{wilcox}, A. B. Evans \cite[Theorems~1--8]{evans}, and J. N. Bray. These proofs used the Classification of Finite Simple Groups: Wilcox reduced the problem to the case of simple groups and settled the groups of Lie type except for the Tits group (the alternating groups had already been resolved by Hall and Paige); Evans handled the Tits group and all the sporadic groups except for the fourth Janko group, which was done by Bray. Bray's proof was published later in \cite[Section~2]{bccsz}. Subsequently, very different proofs not using the Classification were found, using methods of probabilistic and extremal combinatorics and establishing more general results, by Eberhard, Manners and Mrazovi\'c~\cite[Theorem~1.2]{emm} and by M\"uyesser and Pokrovskiy~\cite[Theorem~1.1]{mp}. These later works give CFSG-free proofs for all sufficiently large finite groups and establish stronger asymptotic or random versions.}

{Complete mappings have grown into a rich and still active line of research in cryptography, with real applications in industry. Mittenthal, seeking substitutions resistant to cryptanalysis, studied the use of orthomorphisms of elementary abelian $2$-groups as bijective block substitutions (S-boxes) \cite[pp.~59--60]{Mittenthal}, gave systematic constructions from shift-register sequences \cite[Proposition~3]{Mittenthal}, and characterised these orthomorphisms by a balance condition on the maximal subgroups \cite[unnumbered theorem, p.~66]{Mittenthal}.}

\begingroup
Before trying to extend this story to other magmas, and in particular to
semigroups, we separate three related combinatorial notions.

Let $S$ be a finite magma of order $n$, and let $T$ be its Cayley table. We
record an entry of $T$ as a triple $(r,c,s)$, where $r$ is its row, $c$ its
column and $s=r\cdot c$ its {\em symbol}. A \emph{transversal} of $T$ is a set of
$n$ entries meeting every row, every column and every symbol exactly once.
If $\alpha\colon S\to S$ is a bijection, the entries selected from the rows
of $T$ by $\alpha$ are
\[
T_\alpha=\{(x,x\alpha,x\cdot x\alpha):x\in S\}.
\]
The map $\alpha$ is a complete mapping precisely when $T_\alpha$ is a
transversal. Thus, for every finite magma,  the existence of a complete mapping is equivalent to the
existence of a transversal.

A Cayley table is a Latin square when each symbol occurs exactly once in
every row and every column. Two Latin squares on the same set of cells are
\emph{orthogonal} if every ordered pair of  symbols occurs in exactly one cell;
either square is then called an \emph{orthogonal mate} of the other. For
example, the following two Latin squares are orthogonal, as the superposition
on the right displays:
\[
\begin{array}{c@{\hspace{1.2cm}}c@{\hspace{1.2cm}}c}
\begin{array}{c|ccc}
 &0&1&2\\ \hline
0&0&1&2\\
1&1&2&0\\
2&2&0&1
\end{array}
&
\begin{array}{c|ccc}
 &0&1&2\\ \hline
0&0&2&1\\
1&1&0&2\\
2&2&1&0
\end{array}
&
\begin{array}{c|ccc}
 &0&1&2\\ \hline
0&(0,0)&(1,2)&(2,1)\\
1&(1,1)&(2,0)&(0,2)\\
2&(2,2)&(0,1)&(1,0)
\end{array}
\end{array}
\]

A \emph{resolution into transversals} of a Latin square of order $n$ is a set
of $n$ pairwise disjoint transversals whose union is the set of all $n^2$
entries. For the first square above, let
\[
\begin{aligned}
T_0&=\{(0,0,0),(1,1,2),(2,2,1)\},\\
T_1&=\{(0,2,2),(1,0,1),(2,1,0)\},\\
T_2&=\{(0,1,1),(1,2,0),(2,0,2)\}.
\end{aligned}
\]
Each $T_i$ meets every row, column and symbol exactly once, and the three
sets partition all nine entries. Here each triple records a row, a column and
the symbol in the first square. If we inspect the same cells in the second
square, the three cells selected by $T_i$ all contain the symbol $i$. For
example, $T_0$ selects the cells $(0,0)$, $(1,1)$ and $(2,2)$, and all three
contain $0$ in the second square. Thus the cells selected by $T_0$, $T_1$ and
$T_2$ are exactly the three symbol classes of the second square.
More generally, the symbol classes of an orthogonal mate form a resolution
of the original Latin square. Conversely, labelling the members of a
resolution by distinct symbols produces an orthogonal mate. Hence a Latin
square has an orthogonal mate if and only if it has a resolution into
transversals.

There is an additional feature for group tables. If
\[
T=\{(x,x\alpha,x(x\alpha)):x\in G\}
\]
is a transversal of the table of a finite group $G$, then, for each $g\in G$,
\[
T_g=\{(gx,x\alpha,gx(x\alpha)):x\in G\}
\]
is also a transversal. The sets $T_g$ are pairwise disjoint and cover the
whole table, so one transversal extends by left translation to a resolution.
Consequently, for a finite group, the existence of a complete mapping, a
transversal, a resolution into transversals and an orthogonal mate are all
equivalent; see also \cite[Theorem~1.1]{VLW}.

For a general finite semigroup, its Cayley table need not be a Latin square,
so an orthogonal mate is usually not a meaningful object and a single
transversal need not extend to a resolution. The notions of complete mapping
and transversal nevertheless remain meaningful and equivalent. This is the
formulation that extends from groups to semigroups.
\endgroup
 
For $n\ge 1$, write $[n]=\{1,\ldots,n\}$. The combinatorial Brandt semigroup
$B_n$ has zero $0$ and non-zero elements $E_{i\lambda}$, with
$i,\lambda\in[n]$, and multiplication
\[
E_{i\lambda}E_{j\mu}=\begin{cases}
        E_{i\mu},&\text{if }\lambda=j,\\
        0,&\text{otherwise,}
        \end{cases}
        \quad
        0x=x0=0.
\]

\begin{example}\label{Ex:Brandt_transversals}
Table 1 shows the Cayley table for $B_2$, which has five elements; in the table below
$1=E_{11}$, $2=E_{22}$, $3=E_{12}$ and $4=E_{21}$. This is the smallest inverse
semigroup which is neither a group nor a commutative idempotent semigroup.
\begin{table}[htb]
\[
    \begin{array}{c|ccccc}
        \cdot & 0 & 1 & 2 & 3 & 4 \\
        \hline
        0 & 0^{\ast,\dagger} & 0 & 0 & 0 & 0 \\
        1 & 0 & 1^{\ast} & 0 & 3^{\dagger} & 0 \\
        2 & 0 & 0 & 2^{\dagger} & 0 & 4^{\ast} \\
        3 & 0 & 0 & 3^{\ast} & 0 & 1^{\dagger} \\
        4 & 0 & 4^{\dagger} & 0 & 2^{\ast} & 0
    \end{array}
\]
\caption{The Brandt semigroup $B_2$ with two transversals.}\label{brandt}
\end{table}

{\noindent This table has exactly two transversals: one is marked by asterisks $\ast$, and the other by daggers $\dagger$. In cycle form, the complete mapping for the asterisk transversal is $(2\ 4\ 3)$ and the complete mapping for the dagger transversal is $(1\ 3\ 4)$.}
\end{example}

We now address the apparent asymmetry in the conventional definitions of complete mapping and orthomorphism. Note that a finite magma admits a complete mapping if and only if there is a permutation of its Cayley table columns so that all elements appear on the diagonal. The following shows that the same holds for permutations of rows or for permutations of both rows and columns.

Let $\Sym(X)$ denote the symmetric group on the set $X$.

\begin{prop}\label{Prp:equivs}
Let $S$ be a finite magma. The following are equivalent:
\begin{enumerate}
\item\label{P1} there exists $\alpha\in \Sym(S)$ such that
$\theta\colon S\to S,\quad x\mapsto x\cdot x\alpha$ is a bijection;
\item\label{P2} there exists $\beta\in \Sym(S)$ such that
$\psi\colon S\to S,\quad x\mapsto x\beta\cdot x$ is a bijection;
\item\label{P3} there exist $\gamma,\delta\in \Sym(S)$ such that
$\varphi\colon S\to S,\quad x\mapsto x\gamma\cdot x\delta$ is a bijection.
\end{enumerate}
\end{prop}
\begin{proof}
(\ref{P1}) $\Rightarrow$ (\ref{P2}): Set $\beta = \alpha\inv$. Then $\beta$ is a bijection and $x\psi = x\beta\cdot x = x\alpha\inv\cdot x\alpha\inv\alpha = x\alpha\inv\theta$, hence $\psi = \alpha\inv\theta$ is also a bijection.

\noindent (\ref{P2}) $\Rightarrow$ (\ref{P3}): Set $\gamma = \beta$ and let $\delta$ be the identity mapping.

\noindent (\ref{P3}) $\Rightarrow$ (\ref{P1}): Set $\alpha = \gamma\inv\delta$. Then $\alpha$ is a bijection and $x\theta = x\cdot x\alpha = x\gamma\inv\gamma \cdot x\gamma\inv\delta = x\gamma\inv\varphi$, hence $\theta = \gamma\inv\varphi$ is also a bijection.
\end{proof}

An element $0$ of a magma $S$ is a \emph{zero} (or \emph{absorbing element}) if $0\cdot x = 0 = x\cdot 0$ for all $x\in S$. As will be seen, in the theory of complete mappings for semigroups, a special role is played by semigroups with zero.

\begin{prop}\label{Prp:zero}
Let $S$ be a magma with zero $0$ and with a complete mapping $\alpha\colon S\to S$ with corresponding orthomorphism $\theta\colon S\to S,\quad x\mapsto x\cdot x\alpha$.
Then $0\alpha = 0\theta = 0$.
\end{prop}
\begin{proof}
First, $0\theta = 0\cdot 0\alpha = 0$. Next,
$0\alpha\inv\theta = 0\alpha\inv\cdot 0\alpha\inv\alpha =
0\alpha\inv\cdot 0 = 0$, and so $0 = 0\theta\inv\alpha = 0\alpha$, as claimed.
\end{proof}

For general magmas, there seem to be few necessary conditions one can state for the existence of complete mappings. The following, which is immediate from the definition, is such a condition.

{For a finite family $(X_a)_{a\in A}$, a \emph{system of distinct
representatives} is a choice of an element $r_a\in X_a$ for every $a\in A$
such that $r_a\ne r_b$ whenever $a\ne b$.}

\begin{prop}\label{Prp:onto}
If $(S,\cdot)$ is a magma with a complete mapping, then the multiplication map $\cdot\colon S\times S\to S$ is onto; indeed, the family $\{aS\mid a\in S\}$ of images of left multiplications by the elements of $S$ has a system of
distinct representatives (and similarly for right multiplications).
\end{prop}

\begin{proof}
Let $\alpha$ be a complete mapping and let $\theta$ be its orthomorphism. For
$s\in S$, the element $x=s\theta^{-1}$ satisfies
\(s=x\theta=x\cdot x\alpha,\)
so multiplication is onto. Moreover, for each $a\in S$ the element
\(a\theta=a\cdot a\alpha\)
lies in $aS$, and these elements are distinct as $a$ ranges over $S$, because
$\theta$ is a bijection. Hence the family $\{aS\mid a\in S\}$ has a system of
distinct representatives. For right multiplications, we may use
the equivalent formulation in Proposition~\ref{Prp:equivs}: if
$x\psi=x\beta\cdot x$ is bijective, then $a\psi=a\beta\cdot a\in Sa$, and
the elements $a\psi$ are distinct as $a$ ranges over $S$.
\end{proof}

Thus complete mappings correspond to single transversals.
Resolutions and orthogonal mates enter only for Latin squares, while
Proposition~\ref{Prp:onto} gives a weaker necessary condition for arbitrary
magmas.
The proposition suggests that it is natural to investigate complete mappings for magmas in which multiplication on either side by a fixed element is a bijection. These are precisely the \emph{quasigroups}, and most of the literature on complete mappings of magmas other than groups has focused on them. Many studies of complete mappings of quasigroups have a combinatorial flavour and concern transversals in Latin squares, the Cayley tables of finite quasigroups. This reflects the historical origins mentioned above: if a Latin square of order $n$ has $n$ disjoint transversals, then it has an orthogonal mate.

Another motivation for the interest in transversals in the Latin square community is \emph{Ryser's conjecture}, which states that every Latin square of odd order has a transversal. In quasigroup language, the conjecture says that every quasigroup of odd order has a complete mapping.

It would take us too far afield to describe the literature on transversals in Latin squares or, equivalently, complete mappings of quasigroups. We refer the interested reader to the surveys by I. Wanless \cite{Wanless} and R. Montgomery \cite{Montgomery}.

In this paper, we go in a different direction: our aim is to extend the Hall--Paige
existence problem from groups to finite semigroups. The point of departure is that
semigroups have ideals and principal factors, and these interact strongly with complete
mappings. The paper is organized as follows.

Section~\ref{Sec:regularity} proves that a finite semigroup with a complete mapping is
regular and records the inverse semigroup situation in which a complete mapping can be
recovered from its orthomorphism. Section~\ref{Sec:ideals} proves that complete mappings
preserve ideals and Green's $\gj$-classes, and Section~\ref{Sec:firststeps} reduces the
existence problem to principal factors. Section~\ref{Sec:Reesmatrixsemigroups} gives the
classification for Rees matrix semigroups without zero. Section~\ref{Sec:Rees0-matrix-1}
treats Rees $0$-matrix semigroups over groups with complete mappings by a Hall-type
condition on the pattern, and Section~\ref{Sec:Rees0-matrix-2} gives sufficient
conditions for Rees $0$-matrix semigroups over groups without complete mappings.
Section~\ref{Sec:linear-diagram} applies these results to full linear and diagram monoids.
Section~\ref{Sec:Tn} then applies these Rees $0$-matrix results to the full
transformation monoids $T_n$.
Section~\ref{Sec:inv-revisited} applies the principal factor
criterion to inverse semigroups. Section~\ref{Sec:CR_involutions} studies complete
regularity and involutive complete mappings, Section~\ref{Sec:bands} treats bands,
Section~\ref{Sec:strong} treats strong complete mappings, and Section~\ref{Sec:products}
studies products of all elements of a semigroup. The paper ends with open problems in
Section~\ref{Sec:problems}.

Besides the standard Green--Rees structure theory for finite semigroups, we use the
Hall--Paige theorem and Burnside's transfer theorem for the group theoretic reductions,
Hall-type matching and network flow arguments for the pattern
criteria~\cite{Hall,Gale,Edmonds}, total unimodularity for integral
transportation polytopes~\cite{HoffmanKruskal,Schrijver},
K\H{o}nig's edge-colouring theorem for regular bipartite multigraphs~\cite{Konig},
Bevis--Hall--Katz Smith form methods for incidence matrices~\cite{BevisHallKatz}, and the classical
permanent bounds used in the asymptotic enumeration of Latin squares
\cite{Bregman,Egorychev,Falikman}.

\section{Regularity}
\label{Sec:regularity}
An element $a$ of a semigroup $S$ is said to be \emph{regular} if there exists $b\in S$ such that $aba=a$. If every element of a semigroup is regular, then the semigroup itself is said to be regular.

\begin{theorem}\label{Thm:regular}
A finite semigroup with a complete mapping is regular.
\end{theorem}

\begin{proof}

Let $\alpha$ be a complete mapping of $S$, and let $\theta$ be the corresponding
orthomorphism, so that
\(x\theta=x\cdot x\alpha \quad (x\in S).\)
Put $\delta=\alpha^{-1}\theta$. Then, for every $x\in S$,
\(x\alpha^{-1}\cdot x=x\delta.\)
Since $S$ is finite, the permutations $\theta$ and $\delta$ have finite orders;
write these orders as $m$ and $n$ respectively.

Fix $a\in S$. We first choose $b\in S$ such that $ab=a\theta^{-1}$.
If $m=1$, take $b=a\alpha$. If $m>1$, take
\(b=a\alpha\cdot a\theta\alpha\cdots a\theta^{m-2}\alpha.\)
In both cases the identity $x\theta=x\cdot x\alpha$ gives $ab=a\theta^{-1}$.

Similarly, we can choose $c\in S$ such that
\(ca=a\delta^{-1}.\)
If $n=1$, take $c=a\alpha^{-1}$. If $n>1$, take
\(c=a\delta^{n-2}\alpha^{-1}\cdots a\delta\alpha^{-1}\cdot a\alpha^{-1}.\)
Using $x\alpha^{-1}\cdot x=x\delta$, we obtain $ca=a\delta^{-1}$ in both
cases.

Now put $a'=bc$. Since $\delta^{-1}=\theta^{-1}\alpha$, we have
\(aa'a=(ab)(ca)=(a\theta^{-1})(a\delta^{-1})=(a\theta^{-1})(a\theta^{-1}\alpha)=a.\)
Thus every element of $S$ is regular.

\end{proof}

If $aba=a$ for some $a,b$ in a semigroup $S$, then setting $a' = bab$, we obtain both $aa'a=a$
and $a'aa'=a'$. In this case $a'$ is said to be an \emph{inverse} of $a$. The set of all inverses of an element $a$ is denoted by $V(a)$.

Recall that an orthomorphism $\theta$ of a group $G$ uniquely determines its corresponding complete mapping by $x\inv\cdot x\theta = x\alpha$
for all $x\in G$. As discussed in the introduction, in semigroups, it may be the case that a permutation is an orthomorphism corresponding
to more than one complete mapping. For example, consider the $2$-element left-zero semigroup $S=\{0,1\}$ defined by $xy=x$ for all $x,y$.
Both the identity permutation and the permutation that exchanges the elements are complete mappings with the same corresponding orthomorphism,
namely the identity permutation.

Nevertheless, once a complete mapping and orthomorphism have been fixed, we can express the complete mapping in terms of the orthomorphism
and a particular choice of inverse.

\begin{lemma}\label{Lem:pregood_inverse}
Let $S$ be a finite semigroup and let $\alpha\colon S\to S$ be a complete mapping with orthomorphism $\theta\colon S\to S,\ x\mapsto x\cdot x\alpha$. Then every inverse mapping $x\mapsto x'\in V(x)$, $x\in S$, satisfies
\begin{equation}\label{Eqn:good2}
x\theta\cdot (x\theta)'\cdot x = x.
\end{equation}
\end{lemma}
\begin{proof}
Since $S$ is finite, $\theta$ has finite order, say, $m$.
For all $x\in S$, we claim that for all non-negative integers $i$, $xx'\cdot x\theta^i = x\theta^i$. This is clear for $i=0$ since $x'\in V(x)$.
Assuming the claim for $i$, we have
\(xx'\cdot x\theta^{i+1}=xx'\cdot x\theta^i\cdot x\theta^i\alpha=x\theta^i\cdot x\theta^i\alpha=x\theta^{i+1},\)
using the induction hypothesis in the second equality. This establishes the claim. Now taking $i = m-1$, we have $xx'\cdot x\theta^{m-1} = x\theta^{m-1}$, that is, $xx'\cdot x\theta\inv = x\theta\inv$ for all $x\in S$. Replacing $x$ with $x\theta$, we obtain \eqref{Eqn:good2}.
\end{proof}

The inverse in Theorem~\ref{Thm:regular} may be chosen in a way compatible with the fixed orthomorphism.

\begin{theorem}\label{Thm:good_inverse}
Let $S$ be a finite semigroup and let $\alpha\colon S\to S$ be a complete mapping with orthomorphism $\theta\colon S\to S,\ x\mapsto x\cdot x\alpha$. Then for each $x\in S$, there exists an inverse $\bar{x}\in V(x)$ satisfying
\begin{equation}\label{Eqn:good1}
    \bar{x}\cdot x\theta = x\alpha.
\end{equation}
\end{theorem}
\begin{proof}
By Theorem~\ref{Thm:regular}, $S$ is regular. For each $x\in S$, choose $x'\in V(x)$ and set $\bar{x} = x\alpha\cdot (x\theta)'$. By \eqref{Eqn:good2},
\(x \bar{x} x = x\cdot x\alpha\cdot (x\theta)'\cdot x = x\theta\cdot (x\theta)'\cdot x = x.\)
Next, for all $x\in S$,
\begin{align*}
    \bar{x}x\bar{x} &= x\alpha\cdot (x\theta)'\cdot x\cdot x\alpha\cdot (x\theta)' = x\alpha\cdot (x\theta)'\cdot x\theta\cdot (x\theta)' \\
    &= x\alpha\cdot (x\theta)' = \bar{x}.
\end{align*}
We have shown $\bar{x}\in V(x)$ for all $x\in S$.

Since $S$ is finite, $\alpha\inv\theta$ has finite order, say, $n$.
We claim that for all non-negative integers $i$, $x(\alpha\inv\theta)^i\cdot x'x = x(\alpha\inv\theta)^i$. This is clear for $i=0$. Assuming the claim for $i$,
we have
\begin{align*}
x(\alpha\inv\theta)^{i+1}\cdot x'x &= x(\alpha\inv\theta)^i\alpha\inv\cdot x(\alpha\inv\theta)^i\cdot x'x \\
&= x(\alpha\inv\theta)^i\alpha\inv\cdot x(\alpha\inv\theta)^i = x(\alpha\inv\theta)^{i+1},
\end{align*}
using the induction hypothesis in the second equality. This establishes the claim. Now taking $i=n-1$, we have, for all $x\in S$,
$x(\alpha\inv\theta)^{n-1}\cdot x'x = x(\alpha\inv\theta)^{n-1}$, that is,
$x\theta\inv\alpha \cdot x'x = x\theta\inv\alpha$ since $\alpha\inv\theta$ has order $n$. Replacing $x$ with $x\theta$, we obtain,
for all $x\in S$,
\(\bar{x}\cdot x\theta = x\alpha\cdot (x\theta)'\cdot {x\theta} = x\alpha.\)
This completes the proof.
\end{proof}

From now on, for a finite semigroup $S$ with a complete mapping $\alpha$, we may, without loss of generality, assume that, for each $x\in S$,
we have assigned an inverse $x'\in V(x)$ such that \eqref{Eqn:good1} holds.

A semigroup is said to be an \emph{inverse semigroup} if every element has a \emph{unique} inverse, that is, if $|V(x)| = 1$
for all $x$. Equivalently, an inverse semigroup is a regular semigroup in which the idempotents commute.

\begin{cor}\label{Cor:inv_ortho}
Let $S$ be a finite inverse semigroup with a complete mapping $\alpha\colon S\to S$ and orthomorphism $\theta\colon S\to S,\ x\mapsto x\cdot x\alpha$.
Then for all $x\in S$,
\(x\inv\cdot x\theta = x\alpha.\)
\end{cor}

\begin{proof}
In an inverse semigroup the inverse of each element is unique. Therefore the
inverse $x'$ supplied by Theorem~\ref{Thm:good_inverse} is necessarily
$x\inv$, and the displayed identity is precisely \eqref{Eqn:good1}.
\end{proof}

If $S$ is a regular semigroup with inverse mapping $x\mapsto x'\in V(x)$ and if $\alpha,\theta\colon S\to S$ are permutations such that
$x'\cdot x\theta = x\alpha$ for all $x\in S$, then it need not be the case that $\alpha$ is a complete mapping with orthomorphism $\theta$.
Indeed, let $S=\{0,1\}$ be the $2$-element right-zero semigroup defined by $xy=y$ for all $x,y\in S$ and let $x\mapsto x'$ be the identity
permutation. If $\alpha\colon S\to S$ is the identity permutation and $\theta\colon S\to S$ is a permutation that exchanges the elements, then
$x'\cdot x\theta = x\alpha$ holds, but $0\cdot 0\alpha = 0 \ne 1 = 0\theta$.

The situation is more satisfactory for inverse semigroups.

\begin{theorem}\label{Thm:inv_ortho}
Let $S$ be a finite inverse semigroup, let $\theta$ be a permutation of $S$, and define $\alpha\colon S\to S$ by $x\alpha = x\inv\cdot x\theta$ for all $x\in S$. Then $\theta$ is an orthomorphism if and only if $\alpha$ is a permutation. In this case, $\alpha$ is the unique complete mapping with
corresponding orthomorphism $\theta$.
\end{theorem}
\begin{proof}
Assume $\theta$ is an orthomorphism, so that there exists a complete mapping $\beta$ such that $x\theta = x\cdot x\beta$ for all $x\in S$. By Corollary~\ref{Cor:inv_ortho}, $x\beta = x\inv\cdot x\theta = x\alpha$ and thus $\alpha = \beta$ is a complete mapping with $\theta$ as its corresponding orthomorphism. In particular $\alpha$ is a permutation. This also shows the uniqueness of the complete mapping $\alpha$.

Conversely, assume $\alpha$ is a permutation. Let $\iota\colon S\to S,\quad x\mapsto x\inv$ denote the inverse mapping. Since $x\iota\alpha = (x\inv)\alpha = (x\inv)\inv\cdot (x\inv)\theta = x\cdot x\iota\theta$, we have that $\iota\theta$ is a complete mapping with corresponding orthomorphism $\iota\alpha$. By Corollary~\ref{Cor:inv_ortho}, $x\inv\cdot x\iota\alpha = x\iota\theta$ for all $x\in S$. Replacing $x$ with $x\inv$, we get $x\cdot x\alpha = x\theta$ for all $x\in S$. Therefore $\theta$ is an orthomorphism.
\end{proof}

\section{Ideals, Green's relations, and complete mappings}
\label{Sec:ideals}

Let $S$ be a semigroup. A subset $I\subseteq S$ is a \emph{left} (resp.\ \emph{right}) \emph{ideal} if $aI\subseteq I$ ($Ia\subseteq I$) for all $a\in S$, and is a (two-sided) \emph{ideal} if it is both a left ideal and a right ideal.

In finite semigroups, ideals are preserved by complete mappings.

\begin{prop}\label{Prp:preserve_ideals}
Let $S$ be a finite semigroup with a complete mapping $\alpha\colon S\to S$ and orthomorphism $\theta\colon S\to S,\quad x\mapsto x\cdot x\alpha$.
\begin{enumerate}
    \item If $I\subseteq S$ is a right ideal, then $I\theta = I$;
    \item If $I\subseteq S$ is an ideal, then $I\alpha = I$.
\end{enumerate}
\end{prop}
\begin{proof}
(1) If $a\in I$, then $a\theta = a\cdot a\alpha\in I$. Thus $I\theta\subseteq I$ and the desired result follows from the finiteness of $S$.

\noindent (2) Since $S$ is finite, $\alpha$ and $\theta$ have finite orders, say, $k$ and $m$, respectively. In $x\cdot x\alpha = x\theta$, replace $x$ with $x\alpha^{k-1}$ to get
\begin{equation}\label{Eqn:ppi-1}
    x\alpha^{k-1} \cdot x = x\alpha^{k-1}\theta
\end{equation}
for all $x\in S$. Now suppose $a\in I$. From \eqref{Eqn:ppi-1}, $a\alpha^{k-1}\theta = a\alpha^{k-1} \cdot a \in I$. Thus $a\alpha\inv = a\alpha^{k-1} \in I\theta^{-1} = I$ using part (1). We have shown the inference $a\in I\Rightarrow a\alpha\inv\in I$. Applying this $k-1$ times, we get $a\in I\Rightarrow a\alpha^{-(k-1)}=a\alpha\in I$. Thus $I\alpha\subseteq I$, and an appeal to finiteness finishes the proof.
\end{proof}

For an element $a\in S$, the \emph{principal left}, \emph{right},
and \emph{two-sided ideals} generated by $a$ are, respectively, $S^1a$, $aS^1$, and $S^1 aS^1$, where $S^1 = S$ if $S$ is a monoid and $S^1 = S\cup \{1\}$ where $1$ is an adjoined identity element if $S$ is not a monoid.
We now show that complete mappings preserve principal ideals and orthomorphisms preserve principal right ideals.

\begin{lemma}\label{Lem:preserve_pideals}
Let $S$ be a finite semigroup with a complete mapping $\alpha\colon S\to S$ and orthomorphism $\theta\colon S\to S,\quad x\mapsto x\cdot x\alpha$. For all $a\in S$,
\begin{enumerate}
    \item $(aS^1)\theta = a\theta\cdot S^1$, and
    \item $(S^1 aS^1)\alpha = S^1\cdot a\alpha\cdot S^1$.
\end{enumerate}
\end{lemma}

\begin{proof}

For (a), let $m$ be the length of the $\theta$-orbit of $a$. Repeatedly using
$x\theta=x(x\alpha)$ around this orbit gives
\(a=a\theta w\)
for some $w\in S^1$. Hence, for $b\in S^1$,
\((ab)\theta=ab(ab)\alpha=a\theta\, w b(ab)\alpha\in a\theta S^1.\)
This proves $(aS^1)\theta\subseteq a\theta S^1$. Conversely,
$a\theta\in (aS^1)\theta$, and Proposition~\ref{Prp:preserve_ideals} yields
$(aS^1)\theta=aS^1$. Therefore
$a\theta S^1\subseteq (aS^1)\theta$.

For (b), put $I_a=S^1aS^1$. Proposition~\ref{Prp:preserve_ideals} gives
$I_a\alpha=I_a$. Since $a\alpha\in I_a$, we have
\(S^1(a\alpha)S^1\subseteq S^1aS^1.\)
Conversely, let $K=S^1(a\alpha)S^1$. Again $K\alpha=K$, and since
$a\alpha\in K$, applying $\alpha^{-1}$ gives $a\in K$. Hence
$S^1aS^1\subseteq S^1(a\alpha)S^1$. Therefore
\((S^1aS^1)\alpha=S^1aS^1=S^1(a\alpha)S^1.\)

\end{proof}

Principal (left, right, two-sided) ideals naturally induce \emph{Green's preorders}: for all $a,b \in S$,
\begin{align*}
a\preceq_{\gl} b &\iff S^1a\subseteq S^1b &&\iff (\exists u\in S^1)\  a=ub, \\
a\preceq_{\gr} b &\iff aS^1\subseteq bS^1 &&\iff (\exists v\in S^1)\ a=bv, \\
a\preceq_{\gj} b &\iff S^1 aS^1\subseteq S^1 bS^1 &&\iff (\exists u,v\in S^1)\  a=ubv.
\end{align*}
In the second formulation of each preorder, reflexivity follows from the use of $S^1$ instead of $S$.
The equivalence relations induced by these preorders are three of \emph{Green's relations}:
\begin{align*}
a\gll b &\iff a\preceq_{\gl} b\text{ and }b\preceq_{\gl} a, \\
a\grr b &\iff a\preceq_{\gr} b\text{ and }b\preceq_{\gr} a, \\
a\gjj b &\iff a\preceq_{\gj} b\text{ and }b\preceq_{\gj} a.
\end{align*}
The Green relations satisfy $\gl\circ \gr = \gr\circ \gl$, where $\circ$ denotes composition of binary relations; hence their composite
coincides with their join. The last two Green's relations are:
\begin{align*}
\gd &= \gl\lor \gr = \gl\circ \gr = \gr\circ \gl, \\
\gh &= \gl\cap \gr.
\end{align*}
In every semigroup we have $\gd\subseteq \gj$. For finite semigroups, the main concern of this paper, $\gd = \gj$.
Thus, for finite semigroups, principal factor arguments will be stated in terms of $\mathcal J$-classes.

\begin{theorem}\label{Thm:green}
Let $S$ be a finite semigroup with a complete mapping $\alpha\colon S\to S$ and corresponding orthomorphism
$\theta\colon S\to S,\quad x\mapsto x\cdot x\alpha$. Then:
\begin{enumerate}
  \item $\theta$ preserves the preorder $\preceq_{\gr}$, hence the equivalence $\mathcal{R}$;
  \item $\alpha$ preserves the preorder $\preceq_{\gj}$, hence the equivalence $\mathcal{J}$.
\end{enumerate}
\end{theorem}
\begin{proof}
(a) Assume $a\preceq_{\gr} b$ for $a,b\in S$. Then $aS^1\subseteq bS^1$. By Lemma~\ref{Lem:preserve_pideals}(a),
\(a\theta\cdot S^1 = (aS^1)\theta\subseteq (bS^1)\theta = b\theta\cdot S^1.\)
Therefore $a\theta\preceq_{\gr} b\theta$.

(b) Assume $a\preceq_{\gj} b$ for $a,b\in S$. Then $S^1 aS^1\subseteq S^1 bS^1$. {By Lemma~\ref{Lem:preserve_pideals}(b),
\(S^1\cdot a\alpha\cdot S^1=(S^1aS^1)\alpha\subseteq(S^1bS^1)\alpha=S^1\cdot b\alpha\cdot S^1.\)
Therefore $a\alpha\preceq_{\gj} b\alpha$.}
\end{proof}

An element $a$ of a monoid $S$ is said to be a \emph{unit} if there exists $b\in S$ such that $ab=ba=1$. For a unit $a$, it is easy to see that $|V(a)|= 1$ and
so $b$ can be denoted by $a\inv$.
The units of a monoid form a group called the \emph{group of units} and denoted by $G(S)$. Suppose that $S$ is finite and $ab=1$. Left multiplication by $a$ has right inverse left multiplication by $b$, and hence it is surjective. Therefore it is bijective. Since $a(ba)=a=a1$, cancellation gives $ba=1$.

\begin{cor}\label{Cor:units}
 If a finite monoid has a complete mapping, then so does its group of units.
\end{cor}
\begin{proof}
  Let $S$ be a finite monoid with group of units $G(S)$. Let $\alpha\colon S\to S$ be a complete mapping with orthomorphism $\theta\colon S\to S,\quad x\mapsto x\cdot x\alpha$.
  Since $S$ is finite, $G(S)$ is precisely the $\gj$-class of $1$, and thus by Theorem~\ref{Thm:green}, $\alpha$ preserves $G(S)$.
  If $x\in G(S)$, then both $x$ and $x\alpha$ are units, and hence
  $x\theta=x(x\alpha)$ is also a unit. Thus $\theta$ also restricts to a
  permutation of $G(S)$. The restriction of $\alpha$ to $G(S)$ is therefore a
  complete mapping with the restriction of $\theta$ as its orthomorphism.
\end{proof}

If $1\alpha=g$, right translation by $g^{-1}$ normalizes the value at the identity.

\begin{theorem}\label{Thm:fix_ident}
  If a finite monoid has a complete mapping and orthomorphism, then it has a complete mapping and orthomorphism fixing the identity element.
\end{theorem}
\begin{proof}
    Let $S$ be a finite monoid with group of units $G(S)$. Let $\alpha\colon S\to S$ be a complete mapping with orthomorphism $\theta\colon S\to S,\quad x\mapsto x\cdot x\alpha$.
    Let $g=1\alpha$. Again, $G(S)$ is precisely the $\gj$-class of $1$, so $g\in G(S)$ by Theorem~\ref{Thm:green}. Let $\rho_{g\inv}\colon S\to S,\quad x\mapsto xg\inv$ denote
    the right translation by $g\inv$ and note that $\rho_{g\inv}$ is a bijection. Set $\beta=\alpha\rho_{g\inv}$ and $\psi=\theta\rho_{g\inv}$. For all $x\in S$, $x\cdot x\beta = x\cdot x\alpha\cdot g\inv = x\theta\rho_{g\inv} = x\psi$. Thus $\beta$ is a complete mapping with orthomorphism $\psi$. Further $1\beta = 1\alpha\cdot g\inv = gg\inv =1$
    and $1\psi = 1\cdot 1\beta = 1$.
\end{proof}

In contrast to this, it is possible for a semigroup to have a complete mapping, but not have any complete mapping that maps idempotents to idempotents.

\begin{example}
Using \textsc{Mace4}, we found that the (completely simple) semigroup $S$ given by the Cayley table
\[
\begin{array}{c|cccc}
\cdot & 1 & 2 & 3 & 4 \\
\hline
1 & 1 & 2 & 3 & 4 \\
2 & 4 & 3 & 2 & 1 \\
3 & 1 & 2 & 3 & 4 \\
4 & 4 & 3 & 2 & 1
\end{array}
\]
has eight distinct complete mappings. Their corresponding transversals can be easily seen in the table. None of those complete mappings preserve $E(S)=\{1,3\}$.
\end{example}

\section{First steps}
\label{Sec:firststeps}

Theorem~\ref{Thm:green} suggests that the study of complete mappings of finite semigroups should focus on the $\mathcal J$-classes. Let $S$ be a finite semigroup. Given $a\in S$, let $J_a$ be the $\mathcal J$-class of $a$, so that
\(J_a=\{b\in S\mid S^1aS^1=S^1bS^1\},\)
and put
\(I(a)=\{b\in S\mid b\preceq_{\gj}a\}\setminus J_a.\)
Then $S^1 aS^1 = J_a\cup I(a)$. The class $J_a$ is not necessarily a subsemigroup, so we remedy this by introducing a new element, say $J_a^0:=J_a\cup \{0\}$ ($0\notin J_a$). We then define a multiplication in $J_a^0$ as follows: for $u,v\in J_a^0$,
\[
u\times v:=
\begin{cases}
uv & \text{ if } u,v,uv\in J_a \\
0 & \text{ if } u,v\in J_a, uv\notin J_a\\
0 & \text{ if } u=0\text{ or }v=0
\end{cases}.
\]
Then $J_a^0$ is a semigroup whose multiplication agrees with that in $J_a$ whenever possible.
{We call $J_a^0$ the \emph{principal factor} associated with
$J_a$. Equivalently, one collapses the ideal $I(a)$ inside
$J_a\cup I(a)$ to zero; when $I(a)=\varnothing$, our zero-adjoined convention
simply adds a new zero to $J_a$.}
As $S$ is the union of all of its $\mathcal{J}$-classes, it is useful to have
a description of the semigroups $J_a^0$.
{The following standard result is the principal-factor form of
Rees' theorem; it follows from
\cite[Proposition~3.1.5, Proposition~3.2.1 and Theorem~3.2.3]{Howie}.}

\begingroup
Let $G$ be a finite group, let $I$ and $\Lambda$ be non-empty finite sets and
let $P=(p_{\lambda i})$ be a $\Lambda\times I$ matrix over $G\cup\{0\}$
with no zero row or column. The \emph{Rees $0$-matrix semigroup}
$\mathcal M^0(G,I,\Lambda,P)$ has underlying set
$(I\times G\times\Lambda)\cup\{0\}$ and multiplication
\[
(A,a,\alpha)(B,b,\beta)=
\begin{cases}
(A,ap_{\alpha B}b,\beta),&\text{if }p_{\alpha B}\ne0,\\
0,&\text{otherwise,}
\end{cases}
\]
with $0x=x0=0$. If $P$ has no zero entries, the non-zero part
$\mathcal M(G,I,\Lambda,P)=I\times G\times\Lambda$, with the same product,
is a \emph{Rees matrix semigroup}.

\begin{theorem}\label{t:principal-factor-Rees}
Let $S$ be a finite semigroup and let $a\in S$. Then the principal factor
$J_a^0$ is either null or isomorphic to a finite Rees $0$-matrix semigroup.
\end{theorem}
\endgroup

Let $S$ be a semigroup and $a\in S$. Then on the underlying set $S$ we can define a new binary operation as follows: for all $x,y\in S$,
\(x \times_a y = xay.\)
Then $(S,\times_a)$ is also a semigroup, called \emph{a variant} of $S$, which in general is not isomorphic to $S$. However, a group is isomorphic to any of its variants.

\begin{prop}\label{p:gp2rees}
Let $S=\mathcal M(G,I,\Lambda,P)$ be a Rees matrix semigroup. If $G$ has a complete mapping, then $S$ has a complete mapping.
\end{prop}

\begin{proof}
For $A\in I$ and $\lambda\in\Lambda$, put
\(G_{\lambda,A}=\{(A,g,\lambda)\mid g\in G\}.\)
With the multiplication inherited from $S$, this set is the variant $(G,\times_{p_{\lambda,A}})$, where $x\times_{p_{\lambda,A}}y=xp_{\lambda,A}y$. Since a group is isomorphic to each of its variants, every $G_{\lambda,A}$ has a complete mapping. The sets $G_{\lambda,A}$ partition $S$, and products chosen inside one such set remain inside it. Taking the union of these local complete mappings gives a complete mapping of $S$.
\end{proof}

Zeros in the sandwich matrix only turn some products into $0$.

\begin{prop}\label{p:adding0}
Let $S=\mathcal M(G,I,\Lambda,P)$ be a finite Rees matrix semigroup, and let $T=\mathcal M^0(G,I,\Lambda,Q)$ be obtained from $S$ by replacing some entries of $P$ by zero. If $S$ has no complete mapping, then $T$ has no complete mapping.
\end{prop}

\begin{proof}
The non-zero products in the Cayley table of $T$ are products that already occur in the Cayley table of $S$. If $T$ had a complete mapping, Proposition~\ref{Prp:zero} would force its zero to be fixed. The chosen non-zero diagonal entries would therefore form a transversal of the Cayley table of $S$, contradicting the assumption that $S$ has no complete mapping.
\end{proof}

For a $\mathcal J$-class $J$, write $J^0$ for its principal factor, obtained by adjoining a zero to $J$ and replacing products that leave $J$ by zero.

\begin{theorem}\label{t:J-reduct}
Let $S$ be a finite semigroup. Then $S$ has a complete mapping if and only if every principal factor $J_a^0$ has a complete mapping.
\end{theorem}

\begin{proof}
Suppose first that $S$ has a complete mapping $\alpha$, with orthomorphism $\theta$. Let $J$ be a $\mathcal J$-class of $S$, and put
\(I_{\le J}=\{s\in S\mid J_s\le J\},\quad I_{<J}=\{s\in S\mid J_s<J\}.\)
These are ideals, and hence right ideals. Proposition~\ref{Prp:preserve_ideals} gives
\(I_{\le J}\alpha=I_{\le J},\quad I_{<J}\alpha=I_{<J},\quad I_{\le J}\theta=I_{\le J},\quad I_{<J}\theta=I_{<J}.\)
Therefore both $\alpha$ and $\theta$ preserve the difference
\(I_{\le J}\setminus I_{<J}=J.\)
The restriction of $\alpha$ to $J$, extended by fixing the zero of $J^0$, is then a bijection of $J^0$. Its product map is the restriction of $\theta$ to $J$, again extended by fixing zero, because the selected products $x(x\alpha)$ with $x\in J$ remain in $J$. Thus $J^0$ has a complete mapping.

Conversely, suppose that every principal factor $J^0$ has a complete mapping. By Proposition~\ref{Prp:zero}, each such complete mapping fixes the zero of $J^0$, and hence restricts to a bijection $\alpha_J\colon J\to J$ whose product map is a bijection of $J$. Define $\alpha\colon S\to S$ by $x\alpha=x\alpha_J$ for $x\in J$. Since the $\mathcal J$-classes partition $S$, the map $\alpha$ is a bijection. Moreover, for $x\in J$, the product $x(x\alpha_J)$ is non-zero in $J^0$, and so the same product in $S$ lies in $J$. On each $\mathcal J$-class this product map is the bijection supplied by the complete mapping of $J^0$. Hence $x\mapsto x(x\alpha)$ is a bijection of $S$, and $\alpha$ is a complete mapping of $S$.
\end{proof}

Thus the existence problem for finite semigroups reduces to the principal factors. By the Rees theorem for principal factors, each $J_a^0$ is either a null principal factor or a Rees $0$-matrix semigroup; a null principal factor with at least two elements has no complete mapping. The remaining classification problem is therefore the classification of Rees $0$-matrix semigroups with complete mappings.

\section{Rees matrix semigroups}
\label{Sec:Reesmatrixsemigroups}

In this section we will characterize precisely which Rees matrix semigroups over groups admit complete mappings. A matrix over $G$ is \emph{normalized} if every entry in the first row and column is the identity of $G$. By \cite[Theorem~3.4.2]{Howie}, every Rees matrix semigroup is isomorphic to one with normalized sandwich matrix. We record the normalization. Choose $i_0\in I$ and $\lambda_0\in\Lambda$, and put
\[
q_{\lambda i}=p_{\lambda i_0}^{-1}p_{\lambda i}p_{\lambda_0 i}^{-1}p_{\lambda_0 i_0}.
\]
Then $Q=(q_{\lambda i})$ is normalized, and the map
\[
(i,g,\lambda)\longmapsto
(i,p_{\lambda_0 i_0}^{-1}p_{\lambda_0 i}g p_{\lambda i_0},\lambda)
\]
is an isomorphism from $\mathcal M(G,I,\Lambda,P)$ to $\mathcal M(G,I,\Lambda,Q)$.

Our goal is to prove the following result.

\begin{theorem}\label{t:ReesNon0}
A Rees matrix semigroup $\mathcal{M}(G,I,\Lambda,P)$ with normalized $P$ has a complete mapping if and only if one of the following holds:
\begin{itemize}\itemsep0pt
\item[(a)] $|G|$ is odd;
\item[(b)] $G$ has non-cyclic Sylow $2$-subgroups;
\item[(c)] $|I|\cdot |\Lambda|$ is even;
\item[(d)] The normalized sandwich matrix $P$ has at least one entry of even order.
\end{itemize}
\end{theorem}

\subsection{Going up}

We saw in Proposition~\ref{p:gp2rees} that, if $G$ has a complete mapping,
so does any Rees matrix semigroup $\mathcal{M}(G,I,\Lambda,P)$. Our next
result is a generalization of this fact; it also generalizes the result of
Hall and Paige \cite[Corollary~2]{hp}, according to which, if a group $G$ has a
normal subgroup $N$ such that $N$ and $G/N$ have complete mappings, then
$G$ has a complete mapping.

{For $N\unlhd G$, write $\bar G=G/N$ and let $\bar P$ be
the entrywise image of $P$ in $\bar G$.}

\begin{theorem}
Let $S=\mathcal{M}(G,I,\Lambda,P)$ be a Rees matrix semigroup over $G$, and
let $N$ be a normal subgroup of $G$. Suppose that
\begin{itemize}\itemsep0pt
\item[(a)] $N$ has a complete mapping;
\item[(b)] $\bar{S}=\mathcal{M}(\bar G,I,\Lambda,\bar P)$ has a complete mapping.
\end{itemize}
Then $S$ has a complete mapping.
\label{t:goingup}
\end{theorem}
\begin{proof}
Take $g_1,\ldots,g_r$ coset representatives for $N$ in $G$. Let $\bar\phi$
and $\alpha$ be the complete mappings for $\bar S$ and $N$.
Throughout this proof we use the conjugation convention
\(x^g=g^{-1}xg \quad (x,g\in G).\)
If $(i,g_uN,\lambda)\bar\phi=(i',g_vN,\lambda')$, put
\((i,g_uh,\lambda)\phi=(i',(h\alpha)^{p_{\lambda,i'}}g_v,\lambda').\)

We show that $\phi$ is a bijection. So suppose that
\((i_1,g_{u_1}h_1,\lambda_1)\phi=(i_2,g_{u_2}h_2,{\lambda_2})\phi.\)
Since $\phi$ mod $N$ is $\bar\phi$, which is a bijection, this implies
$i_1=i_2$, $\lambda_1=\lambda_2$, and $u_1=u_2$, so that also $i'_1=i'_2$,
$\lambda'_1=\lambda'_2$, and $v_1=v_2$. Then
$h_1\alpha=h_2\alpha$, so $h_1=h_2$ since $\alpha$ is a bijection on $N$.

Now we show that $\psi$ is a bijection, where $s\psi=s\cdot s\phi$. So suppose
that
\((i_1,g_{u_1}h_1,\lambda_1)\psi=(i_2,g_{u_2}h_2,{\lambda_2})\psi.\)
Again, since we have a bijection mod $N$, we conclude that $i_1=i_2$,
$\lambda_1=\lambda_2$, and $u_1=u_2$, whence also $i'_1=i'_2$,
$\lambda'_1=\lambda'_2$, and $v_1=v_2$. Now we have
\(g_{u_1}h_1p_{\lambda_1,i'_1}(h_1\alpha)^{p_{\lambda_1,i'_1}}g_{v_1}
=g_{u_1}h_2p_{\lambda_1,i'_1}(h_2\alpha)^{p_{\lambda_1,i'_1}}g_{v_1}\), so $h_1\cdot h_1\alpha=h_2\cdot h_2\alpha$. It follows that $h_1=h_2$, and we are done.
\end{proof}

The obstruction for groups without complete mappings is detected in a cyclic $2$-quotient.

\begin{lemma}\label{l:cyclic2-quotient}
Let $G$ be a finite group which does not have a complete mapping. Then $G$
has a normal subgroup $N$ of odd order such that
\(G/N\)
is a non-trivial cyclic $2$-group. Moreover, $N$ has a complete mapping.
\end{lemma}

\begin{proof}
By the Hall--Paige theorem, the Sylow $2$-subgroups of $G$ are non-trivial
and cyclic. Let $U$ be a Sylow $2$-subgroup of $G$. Since $U$ is cyclic,
$\operatorname{Aut}(U)$ is a $2$-group. The group
\(N_G(U)/C_G(U)\)
embeds in $\operatorname{Aut}(U)$, and hence is a $2$-group. On the other
hand, $U\le C_G(U)$, so the order of $N_G(U)/C_G(U)$ divides the odd number
$|N_G(U)/U|$. Therefore $N_G(U)=C_G(U)$. By Burnside's normal
$2$-complement theorem, $G$ has a normal $2'$-subgroup $N$ such that
\(G=NU, \quad N\cap U=1.\)
Thus $G/N\cong U$, which is a non-trivial cyclic $2$-group. Finally, $N$ has
odd order, and so has a complete mapping by the Hall--Paige theorem.
\end{proof}

\subsection[Rees matrix semigroups: even index product]{$G$ has no complete mapping but $|I|\cdot|\Lambda|$ is even}

Suppose that $G$ has no complete mapping. By Lemma~\ref{l:cyclic2-quotient} and Theorem~\ref{t:goingup}, it suffices to deal with the case in which $G$ is a cyclic $2$-group. We prove in this subsection that, if $|I|\cdot|\Lambda|$ is even, then a complete mapping exists.

\begin{theorem}\label{Thm:ILambda_even}
{Let $d\ge0$. If $|I|\cdot|\Lambda|$ is even, then
$\mathcal{M}(C_{2^d},I,\Lambda,P)$ has a complete mapping.}
\end{theorem}
\begin{proof}

If $d=0$, then $C_{2^d}$ is trivial and
$\mathcal{M}(C_{2^d},I,\Lambda,P)$ is a rectangular band. Hence the identity
mapping is a complete mapping. We may therefore assume that $d\ge1$.

We deal first with the case $I=\{i,j\}$, $\Lambda=\{\lambda\}$. Suppose
that $P_{\lambda,i}=p$ and $P_{\lambda,j}=q$.

Suppose first that $d=1$. If $p=q$, define a permutation $f$ by
\begin{align*}
(i,0,\lambda)f&=(i,0,\lambda), &
(i,1,\lambda)f&=(j,0,\lambda),\\
(j,0,\lambda)f&=(i,1,\lambda), &
(j,1,\lambda)f&=(j,1,\lambda).
\end{align*}
The corresponding products are
\begin{align*}
(i,0,\lambda)\cdot(i,0,\lambda)&=(i,p,\lambda), &
(i,1,\lambda)\cdot(j,0,\lambda)&=(i,p+1,\lambda),\\
(j,0,\lambda)\cdot(i,1,\lambda)&=(j,p+1,\lambda), &
(j,1,\lambda)\cdot(j,1,\lambda)&=(j,p,\lambda).
\end{align*}
If $p\ne q$, then $q=p+1$, and define $f$ by
\begin{align*}
(i,0,\lambda)f&=(i,0,\lambda), &
(i,1,\lambda)f&=(j,1,\lambda),\\
(j,0,\lambda)f&=(i,1,\lambda), &
(j,1,\lambda)f&=(j,0,\lambda).
\end{align*}
The corresponding products are
\begin{align*}
(i,0,\lambda)\cdot(i,0,\lambda)&=(i,p,\lambda), &
(i,1,\lambda)\cdot(j,1,\lambda)&=(i,q,\lambda),\\
(j,0,\lambda)\cdot(i,1,\lambda)&=(j,q,\lambda), &
(j,1,\lambda)\cdot(j,0,\lambda)&=(j,p,\lambda).
\end{align*}
Thus $f$ is a complete mapping in both cases.

 Assume next that
$d\ge2$.
We can write the Cayley table in $2\times2$ block form as follows:
\[\begin{array}{c|cc|}
&(i,y,\lambda)&(j,y,\lambda)\\
\hline
(i,x,\lambda) & (i,x+p+y,\lambda) & (i,x+q+y,\lambda) \\
(j,x,\lambda) & (j,x+p+y,\lambda) & (j,x+q+y,\lambda) \\
\hline
\end{array}\]

We are going to take $2^{d-1}$ consecutive rows and columns in each block. Observe that, if $i$ runs from $0$ to $2^{d-1}-1$, then $(i+a)+(i+b)$ takes the values $a+b,a+b+2,\ldots,a+b+2^d-2$; that is, all values in $C_{2^d}$ having the same parity as $a+b$.

In the top left block we take entries in row $x$ and column $x$ for
$x=0,\ldots,2^{d-1}-1$; this gives all entries $(i,2x+p,\lambda)$. In
the bottom left, we have to take the remaining columns $2^{d-1},\ldots,2^d-1$, and we again take rows $0,\ldots,2^{d-1}-1$, to get entries $(j,2x+p,\lambda)$.

Suppose $p$ and $q$ have opposite parity. In the top right block we take
the remaining rows $2^{d-1},\ldots,2^d-1$ and columns $0,\ldots,2^{d-1}-1$
and get entries $(i,2x+q,\lambda)$. In the bottom right block we take
the remaining rows $2^{d-1},\ldots,2^d-1$ and columns $2^{d-1},\ldots,2^d-1$ to get entries $(j,2x+q,\lambda)$.

Now suppose that $p$ and $q$ have the same parity. In the top right block we take rows $2^{d-1},\ldots,2^d-1$, but this time use columns
$1,\ldots,2^{d-1}$, and get entries $(i,2x+1+q,\lambda)$. In the bottom right block we use the remaining rows $2^{d-1},\ldots,2^d-1$ and columns
$2^{d-1}+1,\ldots2^d-1,0$ to obtain entries $(j,2x+1+q,\lambda)$.

We have used each row, column and entry once.

Now, if $|I|$ is even, we split $I$ into parts of size $2$, and deal
with each pair $(i,\lambda)$ and $(j,\lambda)$ as above.

If $|\Lambda|$ is even, we perform a dual construction. This proves the desired result.
\end{proof}

As an illustration, consider the case $G=C_4$. The entries $\lambda$ have been removed to save space.

{\tiny
\[
\begin{array}{c|cccccccc|}
&(i,0)&(i,1)&(i,2)&(i,3)&(j,0)&(j,1)&(j,2)&(j,3)\\\hline
(i,0)&(i,p)^*&(i,p+1)&(i,p+2)&(i,p+3)&(i,q)&(i,q+1)&(i,q+2)&(i,q+3)\\
(i,1)&(i,p+1)&(i,p+2)^*&(i,p+3)&(i,p)&(i,q+1)&(i,q+2)&(i,q+3)&(i,q)\\
(i,2)&(i,p+2)&(i,p+3)&(i,p)&(i,p+1)&(i,q+2)^\circ&(i,q+3)^*&(i,q)&(i,q+1)\\
(i,3)&(i,p+3)&(i,p)&(i,p+1)&(i,p+2)&(i,q+3)&(i,q)^\circ&(i,q+1)^*&(i,q+2)\\
(j,0)&(j,p)&(j,p+1)&(j,p+2)^*&(j,p+3)&(j,q)&(j,q+1)&(j,q+2)&(j,q+3)\\
(j,1)&(j,p+1)&(j,p+2)&(j,p+3)&(j,p)^*&(j,q+1)&(j,q+2)&(j,q+3)&(j,q)\\
(j,2)&(j,p+2)&(j,p+3)&(j,p)&(j,p+1)&(j,q+2)&(j,q+3)&(j,q)^\circ&(j,q+1)^*\\
(j,3)&(j,p+3)&(j,p)&(j,p+1)&(j,p+2)&(j,q+3)^*&(j,q)&(j,q+1)&(j,q+2)^\circ\\
\hline
\end{array}
\]
}

Choose the starred cells from the first four columns. If $p$ and $q$ have the same parity, choose the starred cells from the last four, otherwise choose the circled cells.

If $G$ has no complete mapping, an even product of the two index sizes gives a sufficient condition in the normalized Rees matrix case.

\begin{cor}\label{c:even}
Let $G$ be a group which does not have a complete mapping. If
$|I|\cdot|\Lambda|$ is even, then $\mathcal{M}(G,I,\Lambda,P)$ has a complete mapping.
\end{cor}
\begin{proof}
By Lemma~\ref{l:cyclic2-quotient}, there is a normal subgroup $N$ of odd order
such that $Q=G/N$ is a non-trivial cyclic $2$-group, and $N$ has a complete
mapping. By Theorem~\ref{Thm:ILambda_even},
$\mathcal{M}(Q,I,\Lambda,\bar P)$ has a complete mapping, where $\bar P$ is
the image of $P$ in $Q$. The result follows from Theorem~\ref{t:goingup}.
\end{proof}

\subsection[Rees matrix semigroups: odd index product with even-order entry]{$G$ has no complete mapping, $|I|\cdot|\Lambda|$ is odd, and the normalized matrix $P$ contains an entry of even order}

In order to reduce the problem to a smaller one, we extend the notion of complete mappings to partial binary algebras.

Given any semigroup $S$ and subset $A$ of $S$, we consider $A$ to be a partial binary algebra in which the operation $x \cdot_A y$ for $x,y \in A$ is defined as $x \cdot_A y=x \cdot_S y$ if $x \cdot_S y \in A$, and is undefined otherwise.

A bijection $f\colon A\to A$ is said to be a complete mapping of the partial binary algebra $A$ if the map $\phi\colon A\to A$ given by $x\phi=x\cdot (xf)$ is defined for all $x\in A$ and is a bijection.

If $A$ is a subsemigroup of $S$, this is the ordinary definition of complete mapping. Complete mappings on disjoint partial subalgebras glue as follows.

\begin{lemma}\label{l:partition}
Let $S$ be a semigroup, let $\{P_i\}_{i\in I}$ be a partition of $S$, and suppose that each partial algebra $P_i$ has a complete mapping $f_i$. Define $f\colon S\to S$ by $xf=xf_i$ if $x \in P_i$. Then $f$ is a complete mapping of $S$.
\end{lemma}

\begin{proof}
Since the sets $P_i$ form a partition of $S$ and each $f_i$ is a bijection of
$P_i$, the map $f$ is a bijection of $S$. For $x\in P_i$, the product
$x\cdot xf_i$ is defined in the partial algebra $P_i$ and lies in $P_i$;
moreover the map $x\mapsto x\cdot xf_i$ is a bijection of $P_i$. Hence the
product map $x\mapsto x\cdot xf$ is the disjoint union of these bijections on
the parts $P_i$, and is therefore a bijection of $S$. Thus $f$ is a complete
mapping of $S$.
\end{proof}

A square subgroup generated by one sandwich entry gives the first reduction.

\begin{lemma}\label{l:A-reduct}
Let $G$ be a group with no complete mapping. Assume that $I=\{1,2,\ldots,|I|\}$ and $\Lambda=\{1,2,\ldots,|\Lambda|\}$, where $|I|$ and $|\Lambda|$ are odd and at least $3$.
Let $S=\mathcal{M}(G,I,\Lambda,P)$ and $A$ the partial subalgebra of $S$ on the set
\((\{1\} \times G \times \{2\}) \cup (\{2\} \times G \times \{1,2\}).\)
If $A$ has a complete mapping, then so does $S$.
\end{lemma}
\begin{proof}
Consider the partition of $S$ into the following sets, with any empty set omitted:
\[
A, \{1,3\} \times G \times \{1\}, \{3\} \times G \times \{2,3\}, \{1,2\} \times G \times \{3\},
\]
\[\{1,2,3\} \times G \times \{4,\dots, |\Lambda|\},\{4,\dots,|I|\}\times G \times \{1, \dots, |\Lambda|\}.
\]
Except for $A$, all of the listed parts are subsemigroups of $S$; moreover they have complete mappings by Corollary~\ref{c:even}. As $A$ has a complete mapping by assumption, the result follows from Lemma~\ref{l:partition}.
\end{proof}

\begin{theorem}\label{t:nonidentity}
Let $S=\mathcal{M}(C_{2^d},I,\Lambda,P)$ for $d\ge 1$ and odd-cardinality sets $I,\Lambda$. Assume moreover that the matrix $P$ is in normalized form, and contains a non-identity entry. Then $S$ has a complete mapping.
\end{theorem}
\begin{proof}
Let $G=C_{2^d}$ and $n=2^d$. As the normalized matrix $P$ contains a non-identity element, we have $|I|,|\Lambda|\ge3$.

We may assume that $I=\{1,2,\dots,|I|\}$, $\Lambda=\{1,2,\dots,|\Lambda|\}$, and that $p_{2,2}$ is not the identity of $G$. By Lemma~\ref{l:A-reduct} it suffices to construct a complete mapping for the partial algebra $A$, as defined in the lemma. Assume first that $p_{2,2}$ is a cyclic generator $g$ of $G$.

Define $f\colon A\to A$ as follows: for $0 \le i\le 2^{d-1}-1$ let
\[
(1,g^i,2)f=(2,g^i,2),\; (2,g^i,2)f=(2,g^i,1),\; (2,g^i,1)f=(1,g^i,2),
\]
while $f$ fixes all other elements of $A$. Clearly, $f$ is a bijection. Let $x\phi=x\cdot (xf)$. It suffices to show that $\phi$ is surjective. Let $p=2^{d-1}$.

If $i$ is even, then
\[
\begin{aligned}
(1,g^{i/2+p},2)\phi&=(1,g^i,2), &
(2,g^{i/2+p},1)\phi&=(2,g^i,1), &
(2,g^{i/2},1)\phi&=(2,g^i,2).
\end{aligned}
\]
If $i$ is odd, then
\[
\begin{aligned}
(1,g^{(i-1)/2},2)\phi&=(1,g^i,2), &
(2,g^{(i-1)/2},2)\phi&=(2,g^i,1), &
(2,g^{(i-1)/2+p},2)\phi&=(2,g^i,2).
\end{aligned}
\]
Hence $\phi$ is surjective and $f$ is a complete mapping.

Assume next that $p_{2,2}$ is a non-identity square of $G$. {Choose a generator $a$ of $G$ and write $p_{2,2}=a^e$. Since $p_{2,2}$ is a non-identity square, $e=2^l q$ with $1\le l\le d-1$ and $q$ odd. Replacing $a$ by the generator $g=a^q$, we may write $p_{2,2}=g^{2^l}$. Put $j=2^l$.} Let
\begin{align*}
L_1 &= \{j/2, j/2+j, j/2+2j, \dots, (n-j)/2\}; \\
L_2 &= \left(\{1,2,\dots,(n-j)/2-1\} \cup \{n-j/2+1,n-j/2+2, \dots, n\}\right) \setminus L_1; \\
L_3 &= \{ (n-j)/2+1, (n-j)/2+2, \dots, n-j/2\}.
\end{align*}
Define $f\colon A\to A$ as follows.

If $k\in L_1$, then set
\[
(1,g^k,2)f=(2,g^k,2); \quad (2,g^k,2)f=(2,g^k,1); \quad (2,g^k,1)f=(1,g^k,2).
\]
If $k\in L_2$, then for all $(i,\lambda)\in \{(1,2),(2,1),(2,2)\}$, define $(i,g^k,\lambda)f= (i, g^{k+j}, \lambda)$.

Finally, for $k\in L_3$ and $(i,\lambda)\in \{(1,2),(2,1),(2,2)\}$, let $(i,g^k,\lambda)f= (i, g^{k+j-1}, \lambda)$.

We claim that $f$ is a complete mapping.

For the bijectivity of $f$, consider the sets $L_1$, $L_2+j$, and $L_3+(j-1)$ (with arithmetic modulo $n$). If $k+j \in L_1$, then either $k\in L_1$ or $k=n-j/2\in L_3$. Hence $L_1$ and $L_2+j$ are disjoint.
Similarly, if $k+j-1 \in L_1$, then $k\in L_2$, and so $L_1$ and $L_3+(j-1)$ are disjoint. Finally, if $k+j=l+j-1$, then $k=l-1$, and there are no $k \in L_2, l \in L_3$ satisfying this equation. It follows that $L_1$, $L_2+j$, and $L_3+(j-1)$ are pairwise disjoint. The injectivity, and hence bijectivity of $f$ now follows easily.

For surjectivity of the map $x\phi=x\cdot (xf)$, first let $(i,\lambda)\in\{(1,2),(2,1)\}$. If $k$ is even, then the equation $2x+j=k\pmod n$ has a solution $x'$ in $L_1\cup L_2$. We have
\((i,g^{x'},\lambda)\phi=(i,g^k,\lambda),\)
unless $i=2$, $\lambda=1$, and $x'\in L_1$, in which case
\((2,g^{x'},2)\phi=(2,g^k,1).\)
If $k$ is odd, then the equation $2x+j-1=k\pmod n$ has a solution $x'$ in $L_3$, for which
\((i,g^{x'},\lambda)\phi=(i,g^k,\lambda).\)

We now consider $i=\lambda=2$. If $k$ is an odd multiple of $j$, $2x=k \pmod n$ has a solution $x'$ in $L_1$, for which $(2,g^{x'},1)\phi=(2,g^k,2)$.

 If $k$ is even, but not an odd multiple of $j$, then $2x+2j=k \pmod n$ has a solution $x'$ in $L_2$, for which $(2,g^{x'},2)\phi=(2,g^k,2)$.

 Finally, if $k$ is odd then $2x+2j-1=k \pmod n$ has a solution $x'$ in $L_3$, for which $(2,g^{x'},2)\phi=(2,g^k,2)$.
 The result follows.
\end{proof}

An even-order sandwich entry supplies such a square after passing to the relevant cyclic subgroup.

\begin{cor}\label{c:evenentry}
Let $G$ be a group which does not have a complete mapping, let $I$ and $\Lambda$ be sets of odd cardinality, and let $P$ be a normalized matrix over $G$ containing an element $h$ of even order. Then $\mathcal{M}(G,I,\Lambda,P)$ has a complete mapping.
\end{cor}
\begin{proof}
By Lemma~\ref{l:cyclic2-quotient}, there is a normal subgroup $N$ of odd order
such that $Q=G/N$ is a non-trivial cyclic $2$-group, and $N$ has a complete
mapping. Since $h$ has even order, its image in $Q$ is non-trivial. Hence
$\mathcal{M}(Q,I,\Lambda,\bar P)$ has a complete mapping by
Theorem~\ref{t:nonidentity}. The result follows from Theorem~\ref{t:goingup}.
\end{proof}

\subsection{The converse}
Combined, Proposition~\ref{p:gp2rees}, Corollary~\ref{c:even}, and Corollary~\ref{c:evenentry} show one direction of Theorem~\ref{t:ReesNon0}. In this section, we finish the proof by showing the converse.

\begin{theorem}\label{t:converse}
Suppose that
\begin{itemize}
\item[(a)] $G$ is a group with non-trivial cyclic Sylow $2$-subgroups;
\item[(b)] $I,\Lambda$ are finite sets of odd cardinality;
\item[(c)] $P$ is a normalized matrix over $G$ in which each entry has odd order.
\end{itemize}
Then $\mathcal{M}(G,I,\Lambda,P)$ has no complete mapping.
\end{theorem}
\begin{proof}
By the Hall--Paige theorem, $G$ does not have a complete mapping. Hence, by
Lemma~\ref{l:cyclic2-quotient}, there is a normal subgroup $N$ of odd order
such that $G/N$ is a non-trivial cyclic $2$-group. Let $p$ be the unique
element of order $2$ in $G/N$. Every element of $G/N$ occurs $|N|$ times as the image of an element of $G$. Since $|N|$ is odd, the product of these images is the product of all elements of $G/N$. Pairing every element with its inverse leaves only the unique involution $p$.

In $S=\mathcal{M}(G,I,\Lambda,P)$, each element has the form $(i,g,\lambda)$;
the product of all the elements of $G$ occurring in the middle of such triples
is $p^{|I|\cdot|\Lambda|}=p$ modulo $N$. Suppose that a complete mapping $\phi$ exists. Then the products of the middle elements of $s\phi$, or of $s\psi$, mod~$N$, are also equal to $p$.

Each element of the last product occurs in an element of $S$ of the form
$(i,g,\lambda)(j,h,\mu)$, where the two factors run over all elements of $S$.
The middle coordinate is $g p_{\lambda j}h$. Since $G/N$ is a $2$-group and each sandwich entry $p_{\lambda j}$ has odd order, every $p_{\lambda j}$ maps to the identity in $G/N$. Hence the middle coordinate is congruent to $gh$ modulo $N$. We know that the
product of the $g$s and the product of the $h$s is $p$ modulo $N$. Thus we have a contradiction $p=p^2$ modulo $N$.
\end{proof}

Theorem~\ref{t:ReesNon0} now follows.

\section[Rees 0-matrix semigroups over groups with complete mappings]{Rees $0$-matrix semigroups over groups with\\complete mappings}
\label{Sec:Rees0-matrix-1}

If $P$ is a $\Lambda\times I$ matrix over $G\cup\{0\}$, the \emph{pattern} of
$P$ is the $\Lambda \times I$ matrix $Q$ obtained by replacing all elements of
$G$ by the identity $1$.

\begin{definition}\label{d:local-packet}
Let $S=\mathcal M^0(G,I,\Lambda,P)$. For $i\in I$ and $\lambda\in\Lambda$, write
\(H_{i,\lambda}=\{(i,g,\lambda)\mid g\in G\}.\)
If $p_{\lambda j}\ne0$, multiplication of elements of $H_{i,\lambda}$ by
elements of $H_{j,\mu}$ maps into $H_{i,\mu}$ and is given by
\((i,x,\lambda)(j,y,\mu)=(i,xp_{\lambda j}y,\mu).\)
The triple
\((H_{i,\lambda},H_{j,\mu},H_{i,\mu})\)
is called a \emph{local packet}. We use the same term for the restricted multiplication table on $H_{i,\lambda}\times H_{j,\mu}$. {In this terminology, $H_{i,\lambda}$ is the source class, $H_{j,\mu}$ is the compatible multiplier class and $H_{i,\mu}$ is the target class; compatibility means that $p_{\lambda j}\ne0$.}
\end{definition}

\begingroup

We introduce now the terminology needed for the next theorem.

\begin{quoteddefinition}
Let $Q=(Q_{\lambda i})_{\lambda\in\Lambda,i\in I}$ be a
$\Lambda\times I$ zero-one matrix. A non-negative real matrix
$H=(h_{\lambda i})_{\lambda\in\Lambda,i\in I}$ is said to be a
\emph{balanced support weighting} of $Q$ if $h_{\lambda i}=0$ whenever
$Q_{\lambda i}=0$, every row sum is $|I|$ and every column sum is
$|\Lambda|$. It is said to be \emph{positive} if $h_{\lambda i}\ge1$
whenever $Q_{\lambda i}=1$.
\end{quoteddefinition}

{A balanced support weighting will be used in
Theorem~\ref{c:Rees0}, while the stronger positivity condition will be used in
Theorem~\ref{t:Rees0-general-anchored}. When $Q$ is an $n\times n$ matrix,
$H$ is a balanced support weighting precisely when $H/n$ is a doubly
stochastic matrix whose entries vanish wherever $Q$ has a zero. Thus the
square case belongs naturally to the classical theory of doubly stochastic
matrices. Decomposing $H/n$ into permutation matrices supported on $Q$
recovers the corresponding perfect-matching interpretation and gives alternative
route to obtain the square Hall condition appearing in Theorem~\ref{c:Rees0}.} 

We also recall the following network terminology. Let $N$ be a finite set and
let $$c\colon N\times N\to[0,\infty].$$ The pair $(N,c)$ is said to be a
\emph{finite directed network}; the elements of $N$ are its nodes and
$c(x,y)$ is the capacity of the arc from $x$ to $y$. In particular,
$c(x,y)=0$ means that this arc cannot carry a positive flow. A \emph{flow}
is a function $f\colon N\times N\to\mathbb R$ such that
$f(x,y)+f(y,x)=0$ and $f(x,y)\le c(x,y)$, for all $x,y\in N$. A
\emph{demand} is a function $d\colon N\to\mathbb R$, with negative values
interpreted as supplies. For $S,T\subseteq N$, we will write
$c(S,T)=\sum_{x\in S,y\in T}c(x,y)$ and $d(S)=\sum_{x\in S}d(x)$, and we
will write $f(N,x)=\sum_{y\in N}f(y,x)$. 

The demand $d$ is said to be
\emph{feasible} if there is a flow $f$ such that $f(N,x)\ge d(x)$, for all
$x\in N$.
Gale proved the following theorem
\cite[Feasibility Theorem]{Gale}.

\begin{quotedtheorem}[Gale]
A demand $d$ on a finite directed network $(N,c)$ is feasible if and only if
$d(N\setminus S)\le c(S,N\setminus S)$, for every $S\subseteq N$.
\end{quotedtheorem}

%
%
%
\endgroup

{The next result is the main theorem of this section. Here,} $1$ denotes the trivial group.
\begin{theorem}\label{c:Rees0}
Let $Q$ be a $\Lambda\times I$ zero-one matrix. Then the following are
equivalent:
\begin{enumerate}\itemsep0pt
\item For any group $G$ which possesses a complete mapping, and any matrix
$P$ over $G\cup\{0\}$ with pattern $Q$, the Rees $0$-matrix semigroup
$\mathcal{M}^0(G,I,\Lambda,P)$ has a complete mapping.
\item The Rees $0$-matrix semigroup $\mathcal{M}^0(1,I,\Lambda,Q)$ has a
complete mapping.
\item For any $s$ rows of $Q$, there are at least $s|I|/|\Lambda|$ columns
which contain non-zero entries in some of the chosen rows.
\item For any $r$ columns of $Q$, there are at least $r|\Lambda|/|I|$ rows
which contain non-zero entries in some of the chosen columns.
\item {The pattern $Q$ admits a balanced support weighting.}

\end{enumerate}
\end{theorem}

\begin{lemma}\label{l:factorL}
Let $G$ be a group with a normal subgroup $N$ having a complete mapping.
Suppose that $\mathcal{M}^0(\bar G,I,\Lambda,\bar P)$ has a complete
mapping. Then $\mathcal{M}^0(G,I,\Lambda,P)$ has a complete mapping.
\end{lemma}

\begin{proof}
By Proposition~\ref{Prp:zero}, the complete mapping of
$\mathcal M^0(\bar G,I,\Lambda,\bar P)$ fixes zero. On the non-zero elements we can
use exactly the same lifting construction as in Theorem~\ref{t:goingup}.
Since $\bar p_{\lambda i}=0$ if and only if
$p_{\lambda i}=0$, every non-zero quotient product used by the complete
mapping lifts to a non-zero product. The proof of Theorem~\ref{t:goingup}
therefore shows that this lifted map is bijective and that its product map is
bijective on the non-zero elements. Extending it by $0f=0$ gives a complete mapping of
$\mathcal M^0(G,I,\Lambda,P)$.
\end{proof}

We now prove Theorem \ref{c:Rees0}. 

Taking $N=G$ in Lemma~\ref{l:factorL} gives (b) $\Rightarrow$ (a) in Theorem~\ref{c:Rees0}. Conversely, (a) $\Rightarrow$ (b), since the trivial group has a complete mapping.

For (b) $\Rightarrow$ (d), let $|I|=m$ and $|\Lambda|=n$. Note
that, for $S=\mathcal{M}^0(1,I,\Lambda,Q)$, we can suppress the element of $G$,
which is necessarily the identity; so non-zero elements of the semigroup are identified with pairs in
$I\times\Lambda$. Let $A$ be the Cayley table of $S$.

Suppose that columns $j_1,\ldots,j_r$ of $Q$ have non-zero entries only in rows
$\lambda_1,\ldots,\lambda_s$. Consider the \emph{columns} of $A$ indexed
by $(j_1,\mu),\ldots,(j_r,\mu)$ for all $\mu\in\Lambda$. These $rn$ columns
contain non-zero entries only in the $ms$ rows $(i,\lambda_1),\ldots,(i,\lambda_s)$
for $i\in I$. A complete mapping fixes $0$, and the associated product map is
bijective, so the chosen entries in these non-zero columns must be non-zero and
must lie in distinct non-zero rows. Thus, for the existence of a transversal, we
require $ms\ge nr$.

The dual argument shows that (b) implies (c).

Next we show that (c) and (d) are equivalent. Starting with the
$n\times m$ zero-one matrix $Q$, define a square zero-one matrix $Q^*$ by indexing rows by $I\times\Lambda$, indexing columns by
$\Lambda\times I$, and setting
\(Q^*_{(i,\lambda),(\mu,j)}=Q_{\lambda j}.\)
Equivalently, $Q^*$ has $m$ block rows and $n$ block columns, indexed by the first coordinate of each index pair, with each block being
a copy of $Q$; hence its total size is $mn\times mn$. By abuse of notation, we will below identify rows with their indices.

For a square zero-one matrix, a \emph{one-transversal} is a collection of non-zero entries meeting every row and every column exactly once. Equivalently, it is a perfect matching in the associated bipartite graph.

If $P$ is a zero-one matrix, and $F$ is a set of rows of $P$, let $N_P(F)$ denote the set of columns of $P$ which contain a non-zero entry in
one of the rows in $F$.

Now let $R\subseteq I\times\Lambda$ be a set of rows of $Q^*$, and put
\(L(R)=\{\lambda\in\Lambda\mid(i,\lambda)\in R \text{ for some } i\in I\}.\)
We claim that
\(N_{Q^*}(R)=\Lambda\times N_Q(L(R)).\)
Indeed, the entries in $Q^*$ depend on the second coordinates of their
row and column index pairs, but not on their first coordinates.
If (c) holds, then
\[
        |N_{Q^*}(R)|=n|N_Q(L(R))|
        \ge n\frac{m|L(R)|}{n}=m|L(R)|\ge |R|.
\]
Thus $Q^*$ satisfies Hall's row condition. Conversely, if Hall's row condition
holds for $Q^*$ and $L\subseteq\Lambda$, apply it to
$R=I\times L$. Then
\(m|L|=|R|\le |N_{Q^*}(R)|=n|N_Q(L)|,\)
which is precisely (c). Hence (c) is equivalent to Hall's row condition for
$Q^*$.

The same argument applied to columns shows that (d) is equivalent to Hall's
column condition for $Q^*$. Since $Q^*$ is square, Hall's row condition and
Hall's column condition are equivalent, each being equivalent to the existence
of a one-transversal in $Q^*$. Therefore (c) and (d) are equivalent; moreover,
under either condition $Q^*$ has a one-transversal.

\begingroup
We {show next} that (e) implies (c) and (d). Let
$\Sigma\subseteq\Lambda$. The total weight in the rows of $\Sigma$ is
$m|\Sigma|$ and is supported in the columns of $N_Q(\Sigma)$, whose total
weight is $n|N_Q(\Sigma)|$. Therefore
$m|\Sigma|\le n|N_Q(\Sigma)|$. The dual argument proves (d).

We now prove that (c) implies (e). On the node set
$\Lambda\cup I$, assign demand $-m$ to each row node and demand $n$ to each
column node. Give the arc $\lambda\to i$ infinite capacity when
$Q_{\lambda i}=1$ and give every other arc capacity $0$.

Let $S=A\cup B$, where $A\subseteq\Lambda$ and $B\subseteq I$. A direct
calculation gives
\[
        d((\Lambda\cup I)\setminus S)=m|A|-n|B|.
\]
If some support arc goes from $A$ to $I\setminus B$, then
$c(S,(\Lambda\cup I)\setminus S)=\infty$, so Gale's inequality is
automatic. Otherwise $N_Q(A)\subseteq B$, the cut capacity is $0$ and Gale's
inequality becomes $m|A|\le n|B|$. Condition (c) gives
\[
        m|A|\le n|N_Q(A)|\le n|B|,
\]
so every cut satisfies Gale's inequality. Notice that it is enough to check
the smallest possible set $B$, namely $B=N_Q(A)$; enlarging $B$ only
increases the right-hand side.

Gale's theorem now gives a feasible flow $f$. The sum of the demands is
$-mn+nm=0$, and antisymmetry gives
$\sum_{x\in\Lambda\cup I}f(\Lambda\cup I,x)=0$. Since every net inflow is
at least its demand, all these inequalities are equalities. Put
$h_{\lambda i}=f(\lambda,i)$. The reverse arc $i\to\lambda$ has capacity
$0$, so $h_{\lambda i}\ge0$; if $Q_{\lambda i}=0$, both directions have
capacity $0$ and $h_{\lambda i}=0$. Equality at a row node gives
$\sum_i h_{\lambda i}=m$, and equality at a column node gives
$\sum_\lambda h_{\lambda i}=n$. Thus $H=(h_{\lambda i})$ is a balanced
support weighting. Together with the first paragraph of the proof and the
equivalence of (c) and (d), this proves that (c), (d) and (e) are
equivalent.
\endgroup

It remains to show that (c) implies (b).
This depends on the following
combinatorial result.

\begin{theorem}\label{t:kh}
	Let $f$ be a bijection from the set of cells of an $m\times n$ array $A$ to
	the set of cells of another $m\times n$ array $B$.
	\begin{enumerate}\itemsep0pt
		\item It is possible to premultiply $f$ by a permutation $p_1$ of $A$ which
		fixes the columns and permutes independently the elements in each column so
		that the resulting bijection does not map two cells in the same row to two
		cells in the same column.
		\item It is possible to postmultiply $f$ by a permutation $p_2$ of $B$ which
		fixes the rows and permutes independently the elements of each row such that
		the resulting bijection does not map two cells in the same row to two cells
		in the same column.
	\end{enumerate}
	\label{arrays}
\end{theorem}

\begin{proof}
	(a) We proceed by induction on $m$, the case $m=1$ being trivial.
	Suppose that $m\ge 2$. For each $i\in\{1,\ldots,n\}$, let $C_i$
	denote the set of cells in the $i$-th column of $B$, and put
\(C_i^-:=f^{-1}(C_i).\)
	Thus $|C_i^-|=m$ for each $i$. Our goal is to find a permutation of $A$
	which fixes the columns and permutes independently the elements of each
	column such that, under the resulting bijection, no two cells in the same
	row of $A$ are mapped to the same column of $B$. Equivalently, for each
	$i\in\{1,\ldots,n\}$, at most one cell of each row of $A$ is mapped into $C_i$;
	since each row has $n$ cells and $B$ has $n$ columns, this means exactly
	one cell of each row is mapped into each $C_i$.

	Consider the bipartite graph $\Gamma$ whose left vertices are the columns of
	$A$ and whose right vertices are the columns of $B$. A column $K$ of $A$
	is adjacent to the $i$-th column of $B$ if and only if
\(K\cap C_i^-\ne\emptyset.\)
	For any set $S$ of columns of $A$, there are $m|S|$ cells in
	$\bigcup S$. These cells are contained in the union of the sets $C_i^-$
	for which the $i$-th column of $B$ is a neighbour of some column in $S$.
	Since each $C_i^-$ has size $m$, it follows that
\(m|S|\le m|N_\Gamma(S)|,\)
	and hence $|N_\Gamma(S)|\ge |S|$. By Hall's marriage theorem \cite[Theorem~1]{Hall}, $\Gamma$ has a
	matching which saturates the columns of $A$; since both parts have $n$
	vertices, this matching is perfect.

	Let $M$ be such a perfect matching. For each column $K$ of $A$, let
	$i(K)$ be the index such that $K$ is matched by $M$ to the $i(K)$-th
	column of $B$. Choose a cell
\(x_K\in K\cap C_{i(K)}^-.\)
	Let $y_K$ be the cell in the first row of $A$ and in the column $K$.
	Define a permutation $p^*$ of $A$ by swapping $y_K$ and $x_K$ in each
	column $K$, and fixing all other cells. Thus $p^*$ fixes the columns of
	$A$ and permutes independently the elements of each column.

	Let $A^*$ be the $m\times n$ array obtained from $A$ by applying $p^*$. The first
	row of $A^*$ is mapped by $f$ to cells lying in distinct columns of $B$:
	indeed, if $K$ is matched to the $i(K)$-th column of $B$, then
\(y_Kp^*f=x_Kf\in C_{i(K)}.\)

	Now let $A'$ be the $(m-1)\times n$ array obtained from $A^*$ by deleting
	its first row. For each $i\in\{1,\ldots,n\}$, put
\(D_i=A'\cap C_i^-.\)
	Since the first row of $A^*$ contains exactly one element of $C_i^-$, we
	have $|D_i|=m-1$ for each $i$.

	Let $B'$ be any $(m-1)\times n$ array, and let $C_i'$ denote the $i$-th
	column of $B'$. Choose a bijection
\(f'\colon A'\longrightarrow B'\)
	such that
\((f')^{-1}(C_i')=D_i \quad\text{for every }i\in\{1,\ldots,n\}.\)
	Such a bijection exists because both $D_i$ and $C_i'$ have size $m-1$.

	By induction, there is a permutation $q'$ of $A'$ which fixes the columns
	and permutes independently the elements of each column such that the
	resulting bijection $q'f'$ does not map two cells in the same row of $A'$
	to two cells in the same column of $B'$. Extend $q'$ to a permutation $q$
	of $A^*$ by fixing every cell in the first row of $A^*$.

	Then $qf$ has the required property on the rows of $A^*$. The first row
	has the required property by the construction of $p^*$, and the remaining
	rows have the required property by the induction hypothesis applied to
	$f'$, because $qf$ and $qf'$ induce the same partition of $A'$ according
	to the target column. Therefore
\(p_1=p^*q\)
	is the desired permutation of $A$.

	\medskip

	(b) Apply part (a) to the inverse bijection $f^{-1}$, viewed as a
	bijection from the cells of $B^{\mathrm t}$ to the cells of
	$A^{\mathrm t}$. This gives a permutation $q$ of $B^{\mathrm t}$ which
	fixes the columns of $B^{\mathrm t}$, equivalently fixes the rows of $B$,
	and such that $qf^{-1}$ does not map two cells in the same row of
	$B^{\mathrm t}$ to two cells in the same column of $A^{\mathrm t}$.
	Translating back, this says that $(qf^{-1})^{-1}=fq^{-1}$ does not map
	two cells in the same row of $A$ to two cells in the same column of $B$.
	Thus $p_2=q^{-1}$ is the required permutation of $B$.
\end{proof}

Now assume (c) of Theorem~\ref{c:Rees0}, and choose a one-transversal
of $Q^*$. This transversal induces a bijection
\(\tau\colon I\times\Lambda\longrightarrow I\times\Lambda\)
from row indices to inverted column indices, such that if
\((a,\lambda)\tau=(j,\mu),\)
then $Q_{\lambda j}=1$. As in the previous cases, we identify $I\times\Lambda$ with $\mathcal M^0(1,I,\Lambda,Q)\setminus \{0\}$ by ignoring the (constant) group coordinate.

Applying Theorem~\ref{t:kh} to $f=\tau$ yields a permutation $p_1$ as in part (a) of the theorem. Define a map $\alpha$ on $\mathcal M^0(1,I,\Lambda,Q)$ by
\(0\alpha=0,\quad (a,\lambda)\alpha=((a,\lambda)p_1)\tau\quad\text{for }(a,\lambda)\in I\times\Lambda.\)
Thus $\alpha$ is the extension of $p_1\tau$ that fixes $0$. We show that it is a complete mapping.

Indeed, $\alpha$ is clearly a bijection. Because $p_1$ preserves the second coordinate of its argument, if
\(((a,\lambda)p_1)\tau=(j,\mu),\)
then $Q_{\lambda j}=1$, and hence
\((a,\lambda)\cdot (a,\lambda)\alpha\ne 0.\)

Now let $(a,\lambda)\ne (a',\lambda')$. If $a\ne a'$, then
\[
(a,\lambda)\cdot (a,\lambda)\alpha
\quad\text{and}\quad
(a',\lambda')\cdot (a',\lambda')\alpha
\]
have different first coordinates. If $a=a'$ and $\lambda\ne\lambda'$, then, by the properties of $p_1$,
\((a,\lambda)\alpha \quad\text{and}\quad (a',\lambda')\alpha\)
have different second coordinates, and hence so do
\[
(a,\lambda)\cdot (a,\lambda)\alpha
\quad\text{and}\quad
(a',\lambda')\cdot (a',\lambda')\alpha.
\]
Hence the function
\((a,\lambda)\mapsto (a,\lambda)\cdot (a,\lambda)\alpha\)
is injective and, by finiteness, bijective on the non-zero elements. Together with $0\alpha=0$, this proves that $\alpha=p_1\tau$ is a complete mapping. Hence (b) holds.

\subsection{The combinatorial case}
The trivial group case has two useful concrete forms. First, the Brandt
semigroups $B_n$, introduced on page \pageref{brandt}, admit an exact enumeration in
terms of Latin squares. Second, the square zero-one specialization of
Theorem~\ref{c:Rees0} is precisely Hall's marriage condition.

Recall that $B_n$ denotes the combinatorial Brandt semigroup with matrix units
$E_{i\lambda}$, where $i,\lambda\in[n]$. In the notation of this section,
\(B_n=\mathcal M^0(1,[n],[n],\Delta_n),\)
where $\Delta_n$ is the $[n]\times[n]$ identity pattern. This Rees description
allows the complete mappings of $B_n$ to be described by Latin squares.

\begin{prop}\label{prop:brandt-latin}
The complete mappings of $B_n$ are in explicit bijection with the Latin squares
of order $n$. More precisely, a bijection $\alpha\colon B_n\to B_n$ is a complete
mapping if and only if
\(0\alpha=0, \quad E_{i\lambda}\alpha=E_{\lambda,L_{i\lambda}} \quad (i,\lambda\in[n])\)
for a Latin square $L=(L_{i\lambda})$ of order $n$.
\end{prop}

\begin{proof}
Let $L$ be a Latin square of order $n$ and define $\alpha$ by the displayed
formulae. Since each column of $L$ is a permutation of $[n]$, the assignment
$(i,\lambda)\mapsto (\lambda,L_{i\lambda})$ is a bijection of $[n]^2$; hence
$\alpha$ is a bijection of $B_n$. If $\theta$ is the corresponding
orthomorphism, then
\(0\theta=0, \quad E_{i\lambda}\theta =E_{i\lambda}E_{\lambda,L_{i\lambda}} =E_{i,L_{i\lambda}}.\)
Since each row of $L$ is a permutation of $[n]$, the assignment
$(i,\lambda)\mapsto (i,L_{i\lambda})$ is also a bijection of $[n]^2$. Thus
$\theta$ is a bijection, and $\alpha$ is a complete mapping.

Conversely, let $\alpha$ be a complete mapping of $B_n$, and let $\theta$ be its
orthomorphism. Since $0\theta=0$ and $\theta$ is injective, no non-zero element
can be mapped to $0$ by $\theta$. Therefore, for each matrix unit
$E_{i\lambda}$, the element $E_{i\lambda}\alpha$ must be a matrix unit whose
first index is $\lambda$; otherwise
$E_{i\lambda}(E_{i\lambda}\alpha)=0$. Hence there are uniquely determined
entries $L_{i\lambda}\in[n]$ such that
\(E_{i\lambda}\alpha=E_{\lambda,L_{i\lambda}}.\)
It follows also that $0\alpha=0$. The injectivity of $\alpha$ on the matrix
units implies that each column of $L$ is a permutation of $[n]$, while the
injectivity of $\theta$, whose non-zero part is
$E_{i\lambda}\mapsto E_{i,L_{i\lambda}}$, implies that each row of $L$ is a
permutation of $[n]$. Thus $L$ is a Latin square. The two constructions are
mutually inverse.
\end{proof}

In particular, $B_n$ has exactly $L(n)$ complete mappings, where $L(n)$ denotes
the number of Latin squares of order $n$. Thus $B_2$, the semigroup displayed in
Example~\ref{Ex:Brandt_transversals}, has exactly two complete mappings. From \cite{McKayWanless} we get that 
$B_3$ has $12$ complete mappings and $B_4$ has $576$.

The bijection above also transfers the asymptotic enumeration of Latin squares directly to complete mappings of $B_n$. The next proposition records the resulting growth rate.
\begin{prop}\label{p:brandt-latin-asymptotic}
As $n\to\infty$, the number $L(n)$ of complete mappings of $B_n$ satisfies
\[
        \log L(n)=n^2\log n-2n^2+o(n^2).
\]
Consequently,
\[
        L(n)=\left((1+o(1))\frac{n}{e^2}\right)^{n^2}.
\]
\end{prop}

\begin{proof}
Construct a Latin square row by row. After $n-r$ rows have been chosen, form
the $n\times n$ zero-one matrix $A_r$ whose rows are indexed by the columns
of the partial Latin square and whose columns are indexed by its symbols. Put
a $1$ in position $(j,s)$ when the symbol $s$ is still available in column
$j$. Every row and every column of $A_r$ has sum $r$, and the possible next
rows are the perfect matchings of this availability graph. Their number is
$\operatorname{per}A_r$.

Br\'egman's inequality gives
$\operatorname{per}A_r\le (r!)^{n/r}$~\cite[Theorem~4]{Bregman}. Since
$r^{-1}A_r$ is doubly stochastic, the van der Waerden inequality gives
\[
        \operatorname{per}A_r\ge r^n\frac{n!}{n^n}.
\]
This inequality is stated by Egorychev~\cite[p.~299]{Egorychev} and by
Falikman~\cite[Theorem~1]{Falikman}. These bounds hold after every partial
Latin square, so multiplying them for $r=1,\ldots,n$ gives
\[
        \frac{(n!)^{2n}}{n^{n^2}}
        \le L(n)
        \le \prod_{r=1}^n(r!)^{n/r}.
\]

For the lower bound, Stirling's formula gives
\[
        2n\log(n!)-n^2\log n
        =n^2\log n-2n^2+O(n\log n).
\]
For the upper bound, use
$\log(r!)=r\log r-r+O(\log(r+1))$. Then
\[
\begin{aligned}
        n\sum_{r=1}^n\frac{\log(r!)}r
        &=n\log(n!)-n^2
          +O\left(n\sum_{r=1}^n\frac{\log(r+1)}r\right)\\
        &=n^2\log n-2n^2+O\bigl(n(\log n)^2\bigr).
\end{aligned}
\]
Taking logarithms in the two-sided estimate proves the first formula.
Write its error term as $n^2\varepsilon_n$, where $\varepsilon_n\to0$.
Exponentiation gives
\[
        L(n)=\left(e^{\varepsilon_n}\frac{n}{e^2}\right)^{n^2}
             =\left((1+o(1))\frac{n}{e^2}\right)^{n^2},
\]
and the proposition is proved.
\end{proof}

For a square zero-one matrix and the trivial group, Theorem~\ref{c:Rees0}
corresponds to Hall's theorem.

\begingroup
The following corollary makes this correspondence precise: it identifies
the support condition in Theorem~\ref{c:Rees0} with the usual perfect
matching condition in the square case over the trivial group. A
\emph{one-transversal} of $Q$ is a set of $|I|$ entries equal to $1$ meeting
every row and every column exactly once.
{Recall that, for $\Sigma\subseteq\Lambda$, we let
$N_Q(\Sigma)=\{i\in I:Q_{\lambda i}=1
\text{ for some }\lambda\in\Sigma\}$.}
\begin{cor}\label{cor:rees0-hall-square}
Let $I$ and $\Lambda$ be finite sets with $|I|=|\Lambda|$, and let $Q$ be a
$\Lambda\times I$ zero-one matrix. Then the following are equivalent:
\begin{enumerate}\itemsep0pt
\item $\mathcal M^0(1,I,\Lambda,Q)$ has a complete mapping;
\item $Q$ has a one-transversal;
\item the bipartite graph with biadjacency matrix $Q$ has a perfect matching;
\item $\operatorname{per}Q\ne0$;
\item $|N_Q(\Sigma)|\ge|\Sigma|$ for every $\Sigma\subseteq\Lambda$.
\end{enumerate}
\end{cor}
\begin{proof}
The trivial group has a complete mapping, so Theorem~\ref{c:Rees0} shows that
(a) holds precisely when the row and column inequalities in conditions (c)
and (d) of that theorem hold. Since $|I|=|\Lambda|$, its row inequality is
exactly (e). Hall's theorem makes (c) and (e) equivalent. A perfect matching
is the same as a one-transversal, so (b) and (c) are equivalent, and the
permanent of a zero-one matrix is non-zero precisely when such a transversal
exists, so (b) and (d) are equivalent. Finally, in a square bipartite graph a
perfect matching also saturates the column class, so the column inequality of
Theorem~\ref{c:Rees0} follows. Thus, the five conditions are equivalent.
\end{proof}
\endgroup

\section[Rees 0-matrix semigroups over groups without complete mappings]{Rees $0$-matrix semigroups over groups without complete mappings}
\label{Sec:Rees0-matrix-2}

Let $S=\mathcal M^0(G,I,\Lambda,P)$ be a finite Rees $0$-matrix semigroup
whose maximal subgroup does not have a complete mapping. By the Hall--Paige
theorem, the Sylow $2$-subgroups of $G$ are non-trivial and cyclic. By Lemma~\ref{l:cyclic2-quotient}, we may replace $G$ with its cyclic $2$-group quotient.
If at least one of $|I|$ and $|\Lambda|$ is even, we will show that the existence of a complete mapping depends only on the pattern of $P$, not on its individual entries. When both index sets are odd, this no longer holds: the quotient values of the non-zero sandwich entries enter through
incidence equations, which we address in Theorem~\ref{t:incidence-criterion}.

A \emph{near-transversal} in a finite multiplication table is a set of entries meeting each row, each column, and each product value at most once, and having one fewer entry than a full transversal.

For a compatible pair of $\mathcal{H}$-classes in
\(S=\mathcal M^0(G,I,\Lambda,P), \quad p_{\lambda j}\ne0,\)
the local multiplication has the form
\[
H_{i,\lambda}\times H_{j,\mu}
        \longrightarrow H_{i,\mu},
        \quad
        ((i,x,\lambda),(j,y,\mu))\longmapsto (i,xp_{\lambda j}y,\mu).
\]
This is the local packet of Definition~\ref{d:local-packet}.
With appropriate parametrization, this is equivalent to the multiplication table of a variant of $G$.
When $G$ has no complete mapping, the local packet will not have a full transversal. Instead, we will use a near-transversal. The remaining entry of each packet will then be paired with one from another packet.

For instance, if $G=C_2=\{1,\varepsilon\}$ and the sandwich entry is $u=1$, the
local table is
\[
\begin{array}{c|cc}
        x\backslash y & 1 & \varepsilon \\ \hline
        1 & 1 & \varepsilon \\
        \varepsilon & \varepsilon & 1
\end{array}
\]
and there is no two-cell transversal with distinct rows, columns, and products.
Each single cell is, however, a near-transversal. If the omitted row is $r$ and
the omitted column is $c$, then the omitted product is $\varepsilon ruc$. The
variables in the proof record these omitted rows and columns, and the incidence
equations express the condition that the omitted products of one packet are
filled by the residual products of other packets.

When one index set is even, it turns out that the omitted entries can be paired by a fixed-point-free
involution. When both index sets are odd, at least one other cycle is required;
the product around each such cycle leads to an obstruction in the cyclic
$2$-group quotient.

\begin{definition}\label{d:Q-admissible-routing}
Let $Q$ be a $\Lambda\times I$ zero-one matrix, and put
\(A=I\times\Lambda.\)
A \emph{$Q$-admissible routing} is a bijection
\(\Theta\colon A\longrightarrow A\)
with the following two properties. If
\((i,\lambda)\Theta=(j_{i,\lambda},\mu_{i,\lambda}),\)
then
\(Q_{\lambda,j_{i,\lambda}}=1,\)
and the map
\(\Pi_\Theta\colon A\longrightarrow A, \quad (i,\lambda)\Pi_\Theta=(i,\mu_{i,\lambda}),\)
is a bijection.
\end{definition}

\begin{lemma}\label{l:Rees0-class-routing}
Let $Q$ be a $\Lambda\times I$ zero-one matrix satisfying the equivalent
conditions (c) and (d) of Theorem~\ref{c:Rees0}. Then $Q$ admits a
$Q$-admissible routing.
\end{lemma}

\begin{proof}
By Theorem~\ref{c:Rees0}, the Rees $0$-matrix semigroup
\(\mathcal M^0(1,I,\Lambda,Q)\)
has a complete mapping. By Proposition~\ref{Prp:zero}, a complete mapping of
a semigroup with zero fixes zero, and therefore restricts to a bijection on
the non-zero elements. We identify the non-zero elements of
$\mathcal M^0(1,I,\Lambda,Q)$ with $I\times\Lambda$, and denote this
restricted bijection by $\Theta$.

If
\((i,\lambda)\Theta=(j,\mu),\)
then the product
\((i,\lambda)(j,\mu)\)
is non-zero; otherwise the associated product map would send both $0$ and
$(i,\lambda)$ to $0$. Hence $Q_{\lambda j}=1$. Moreover, the product map
sends $(i,\lambda)$ to $(i,\mu)$, and is bijective on the non-zero elements.
This says that
\((i,\lambda)\longmapsto (i,\mu)\)
is a bijection of $I\times\Lambda$. Thus $\Theta$ is a $Q$-admissible
routing.
\end{proof}

\begin{lemma}\label{l:Rees0-support-necessary}
Let
\(S=\mathcal M^0(G,I,\Lambda,P)\)
be a finite Rees $0$-matrix semigroup, and let $Q$ be the pattern of $P$.
If $S$ has a complete mapping, then $Q$ satisfies conditions \textup{(c)} and
\textup{(d)} of Theorem~\ref{c:Rees0}.
\end{lemma}

\begin{proof}
Let $m=|I|$, $n=|\Lambda|$, and let $q=|G|$. Let $\alpha$ be a complete
mapping of $S$, and let
\(x\theta=x\cdot x\alpha\)
be the associated orthomorphism. By Proposition~\ref{Prp:zero}, both
$\alpha$ and $\theta$ fix $0$.

Let $J\subseteq I$, with $|J|=r$, and let
\(L=N_Q(J)=\{\lambda\in\Lambda\mid Q_{\lambda j}=1\text{ for some }j\in J\}.\)
Write $|L|=s$. Consider the set of non-zero elements of $S$ whose first index
lies in $J$:
\(C_J=\{(j,g,\mu)\mid j\in J,\ g\in G,\ \mu\in\Lambda\}.\)
It has size $rnq$. Since $\alpha$ is a bijection, exactly $rnq$ elements
$x\in S\setminus\{0\}$ satisfy $x\alpha\in C_J$. If
\(x=(i,h,\lambda) \quad\text{and}\quad x\alpha=(j,g,\mu)\)
with $j\in J$, then $x\theta=x(x\alpha)$ is non-zero; otherwise $\theta$ would
send both $0$ and $x$ to $0$. Hence $p_{\lambda j}\ne0$, so
$\lambda\in L$. Therefore all these $rnq$ elements $x$ lie in the set
\(\{(i,h,\lambda)\mid i\in I,\ h\in G,\ \lambda\in L\},\)
which has size $msq$. Thus
\(rnq\le msq,\)
and so
\[
        s\ge \frac{rn}{m}.
\]
This is condition \textup{(d)} of Theorem~\ref{c:Rees0}.

Condition \textup{(c)} holds by a dual argument.
\end{proof}

Our next result establishes a collection of near-transversals in each local packet in the special case of a cyclic $2$-group.

\begin{lemma}\label{l:cyclic2-near-transversal}
Let $C$ be a cyclic group of order $2^d$, with $d\ge 1$, and let
$\varepsilon$ be the unique element of order $2$ in $C$. For every
$u,r,c\in C$, the table
\((x,y)\longmapsto xuy \quad (x,y\in C)\)
contains $|C|-1$ cells meeting every row except $r$, every column except $c$, and representing every element of $C$ except
\(\varepsilon ruc.\)
\end{lemma}

\begin{proof}
Inside this proof, we will use additive notation for $C$. Thus
\(C=\mathbb Z/q\mathbb Z, \quad q=2^d, \quad \varepsilon=q/2.\)
For $1\le t\le q-1$, put
\(a_t=(-1)^{t+1}t, \quad s_0=0, \quad s_t=a_1+\cdots+a_t.\)
Then
\(s_{2k-1}=k, \quad s_{2k}=-k.\)
Hence
\(s_0,s_1,\ldots,s_{q-1}\)
run through all elements of $C$, with $s_0=0$ and $s_{q-1}=\varepsilon$.
Also $a_1,\ldots,a_{q-1}$ run through all non-zero elements of $C$. Therefore
the cells
\((s_{t-1},a_t), \quad 1\le t\le q-1,\)
use every row except $\varepsilon$, every column except $0$, and every non-zero entry
 in the table $(x,y)\mapsto x+y$.

We may now pass to the table $(x,y)\mapsto x+u+y$, and translate rows and columns so
that the omitted row is $r$ and the omitted column is $c$. The omitted entry
is then
\(r+u+c+\varepsilon.\)
Returning to multiplicative notation, this is
\(\varepsilon ruc.\)
\end{proof}

The missing entries in the local tables can be associated in pairs if one index set has even cardinality.

\begin{theorem}\label{t:Rees0-noncomp-even}
Let
\(S=\mathcal M^0(G,I,\Lambda,P)\)
be a finite Rees $0$-matrix semigroup. Suppose that $G$ does not have a
complete mapping, and that at least one of $|I|$ and $|\Lambda|$ is even. Let
$Q$ be the pattern of $P$. Then $S$ has a complete mapping if and only if
$Q$ satisfies the equivalent conditions \textup{(c)} and \textup{(d)} of
Theorem~\ref{c:Rees0}.
\end{theorem}

\begin{proof}
Necessity is shown in Lemma~\ref{l:Rees0-support-necessary}. Assume conversely that
$Q$ satisfies conditions \textup{(c)} and \textup{(d)} of Theorem~\ref{c:Rees0}.

By Lemma~\ref{l:cyclic2-quotient}, there is a normal subgroup $N$ of $G$ of
odd order such that $G/N$ is a non-trivial cyclic $2$-group. Since $N$ has a
complete mapping, Lemma~\ref{l:factorL} shows that it is enough to prove the
result after passing to the quotient by $N$. Thus, for the rest of the proof,
we may assume that
\(G=C_{2^d}\)
is cyclic of order $2^d$, with $d\ge1$. Let $\varepsilon$ denote the unique
element of order $2$ in $G$.

We first prove the case where $|\Lambda|$ is even. By
Lemma~\ref{l:Rees0-class-routing}, choose a $Q$-admissible routing
\(\Theta\colon I\times\Lambda\longrightarrow I\times\Lambda.\)
For $a=(i,\lambda)\in I\times\Lambda$, write
\(a\Theta=(j_a,\mu_a), \quad u_a=p_{\lambda,j_a}.\)
Then $u_a\in G$ for every $a$.

Choose a fixed-point-free involution
\(\delta\colon \Lambda\longrightarrow\Lambda.\)
Since $\Pi_\Theta$ is a bijection, there is a unique permutation
\(\rho\colon I\times\Lambda\longrightarrow I\times\Lambda\)
such that, for $a=(i,\lambda)$,
\(a\rho=(i,\lambda')\text{ for some }\lambda' \in \Lambda \quad\text{and}\quad \mu_{a\rho}=\mu_a\delta.\)
The map $\rho$ is a fixed-point-free involution.

For each two-element orbit $\{a,a\rho\}$ of $\rho$, choose one representative
$a$ and define
\(r_{a\rho}=1, \quad r_a=\varepsilon u_{a\rho}u_a^{-1}.\)
Then, for every $a\in I\times\Lambda$, we have
\begin{equation}\label{eq:Rees0-even-pairing}
        r_au_a=\varepsilon r_{a\rho}u_{a\rho}.
\end{equation}

We now define a map $f$ on the reduced Rees $0$-matrix semigroup. First, set
\(0f=0.\)
Next, let $a=(i,\lambda)\in I\times\Lambda$. Applying
Lemma~\ref{l:cyclic2-near-transversal} to the table
\((x,y)\longmapsto xu_ay\)
with omitted row $r_a$ and omitted column $1$, we obtain a bijection
\(\alpha_a\colon G\setminus\{r_a\}\longrightarrow G\setminus\{1\}\)
such that
\(\{\,xu_a(x\alpha_a)\mid x\in G\setminus\{r_a\}\,\} = G\setminus\{\varepsilon r_au_a\}.\)
For $x\in G$, define
\[
        (i,x,\lambda)f=
        \begin{cases}
        (j_a,x\alpha_a,\mu_a), & \text{if }x\ne r_a,\\[2mm]
        (j_a,1,\mu_a\delta), & \text{if }x=r_a.
        \end{cases}
\]

The map $f$ is bijective. Indeed, fix an $\mathcal H$-class
\(H_{j,\mu}=\{(j,g,\mu)\mid g\in G\}.\)
There is a unique $a$ such that
\(a\Theta=(j,\mu).\)
Then $H_a f$ contains every element of $H_{j,\mu}$
except $(j,1,\mu)$. There is also a unique $b=(i, \lambda)$ such that
\(b\Theta=(j,\mu\delta),\)
and such that $f$ sends $(i,r_b, \lambda)$ to $(j,1,\mu)$. Thus every
$\mathcal{H}$-class is in the image of $f$, and hence $f$ is a bijection.

It remains to check that $f$ induces an orthomorphism. Let $a=(i,\lambda)$. For $x\ne r_a$, the products
\((i,x,\lambda)(j_a,x\alpha_a,\mu_a)\)
run through every element of
\(H_{i,\mu_a}\)
except
\((i,\varepsilon r_au_a,\mu_a).\)
If $x=r_a$, we instead obtain the product
\((i,r_a,\lambda)(j_a,1,\mu_a\delta) = (i,r_au_a,\mu_a\delta).\)
Now, $a\rho$ has the same first coordinate as $a$ and satisfies
\(\mu_{a\rho}=\mu_a\delta.\)
The product omitted by applying $f$ to $H_{a\rho}$ without group coordinate $r_{a\rho}$ is
\((i,\varepsilon r_{a\rho}u_{a\rho},\mu_a\delta),\)
which is equal to
\((i,r_au_a,\mu_a\delta)\)
by \eqref{eq:Rees0-even-pairing}. Thus the residual product from $H_a$ fills exactly the product omitted by $H_{a\rho}$. Since $\rho$ is
a fixed-point-free involution, every omitted product is achieved exactly once.
Hence the map
\(x\longmapsto x\cdot xf\)
is a bijection. Therefore $f$ is a complete mapping of the reduced semigroup,
and Lemma~\ref{l:factorL} lifts this complete mapping to the original
semigroup.

This proves the theorem when $|\Lambda|$ is even.

If $|I|$ is even, we may apply the case already proved to the opposite
semigroup $S^{\mathrm{op}}$. The opposite semigroup is again a Rees
$0$-matrix semigroup, with the two index sets interchanged and with transposed
pattern. Conditions (c) and (d) of Theorem~\ref{c:Rees0} are interchanged
under transposition, and hence still hold. Therefore $S^{\mathrm{op}}$ has a
complete mapping; say $h$ is such a mapping.

In terms of the multiplication of $S$, the associated product map in
$S^{\mathrm{op}}$ is
\(x\longmapsto (xh)x.\)
This map is bijective. Since $h$ is a bijection, putting $y=xh$ gives
\(y(yh^{-1})=(xh)x.\)
Hence
\(y\longmapsto y(yh^{-1})\)
is a bijection of $S$. Thus $h^{-1}$ is a complete mapping of $S$.
\end{proof}

A square pattern matrix with a one-transversal satisfies the Hall inequalities.

\begin{cor}\label{c:Rees0-even-transversal}
Let
\(S=\mathcal M^0(G,I,\Lambda,P)\)
be a finite Rees $0$-matrix semigroup such that $G$ does not have a complete
mapping. Suppose that $|I|=|\Lambda|$ is even and that the pattern of $P$
contains a one-transversal. Then $S$ has a complete mapping.
\end{cor}

\begin{proof}
Choose one such one-transversal and let $P'$ be the matrix obtained from $P$ by
replacing every entry outside the transversal by $0$. The pattern of $P'$ is a
permutation matrix, and hence satisfies conditions (c) and (d) of
Theorem~\ref{c:Rees0}. By Theorem~\ref{t:Rees0-noncomp-even},
\(S'=\mathcal M^0(G,I,\Lambda,P')\)
has a complete mapping; call it $f$. For every non-zero $x\in S'$, the product
$x\cdot xf$ is non-zero, and therefore uses one of the entries on the chosen
transversal. The matrices $P$ and $P'$ agree on those entries. Thus the same
map $f$, on the same underlying set and with $0f=0$, is also a complete mapping
of $S$.
\end{proof}

We now turn our attention to the considerably more complicated case in which $G$ does not have a complete mapping and that $|I|$ and $|\Lambda|$ are both odd.

Let $\Gamma$ be a finite directed multigraph. For an edge $e$, write $e^-$
for its tail and $e^+$ for its head. If $H$ is an abelian group, the groups
$H^{E(\Gamma)}$ and $H^{V(\Gamma)}$ are taken with pointwise multiplication.
We define the \emph{incidence homomorphism}
\(\partial_\Gamma\colon H^{E(\Gamma)}\longrightarrow H^{V(\Gamma)}\)
 by
\[
(\partial_\Gamma z)_v
        =
        \left(
        \prod_{\substack{e\in E(\Gamma)\\ e^+=v}}z_e
        \right)
        \left(
        \prod_{\substack{e\in E(\Gamma)\\ e^-=v}}z_e^{-1}
        \right)
        \quad (v\in V(\Gamma)),
\]
 where empty products are $1_H$. Thus an edge contributes $z_e$ at its head and
$z_e^{-1}$ at its tail.

Let $\pi_0(\Gamma)$ be the set of connected components of the underlying
undirected graph of $\Gamma$. The \emph{componentwise augmentation homomorphism}
\(\varepsilon_\Gamma:H^{V(\Gamma)} \longrightarrow H^{\pi_0(\Gamma)}\)
is defined by
\[
(\varepsilon_\Gamma d)_W
        =
        \left(
        \prod_{v\in W}d_v
        \right)
        \quad (W\in\pi_0(\Gamma)).
\]
Here ``componentwise'' means that one product is taken separately on each
connected component $W$. The kernel of $\varepsilon_\Gamma$ is therefore
\[
        \ker\varepsilon_\Gamma
        =
        \left\{
        d\in H^{V(\Gamma)}:
        \left(
        \prod_{v\in W}d_v
        \right)=1_H
        \text{ for every }W\in\pi_0(\Gamma)
        \right\}.
\]
For $z\in H^{E(\Gamma)}$ and $W\in\pi_0(\Gamma)$,
\[
\begin{aligned}
        (\varepsilon_\Gamma(\partial_\Gamma z))_W
        &=
       \left(
        \prod_{v\in W}(\partial_\Gamma z)_v
        \right) \\[1mm]
        &=
        \left(
        \prod_{v\in W}
        \left[
        \left(
        \prod_{\substack{e\in E(\Gamma)\\ e^+=v}}z_e
        \right)
        \left(
        \prod_{\substack{e\in E(\Gamma)\\ e^-=v}}z_e^{-1}
        \right)
        \right]
        \right) \\[1mm]
        &=
        \left(
        \prod_{e\in E(W)}\left(z_ez_e^{-1}\right)
        \right) \\[1mm]
        &=1_H,
\end{aligned}
\]
where $E(W)$ denotes the set of edges whose endpoints lie in $W$. Since $W$ is
a connected component of the underlying undirected graph, every edge with one
endpoint in $W$ has both endpoints in $W$. Loops and parallel edges are counted
with their multiplicities in $E(W)$. This proves
\(\im\partial_\Gamma \subseteq \ker\varepsilon_\Gamma.\)

{Fix a connected component $W$ of the underlying undirected
graph. A loop is an edge $e$ of $\Gamma$ with $e^-=e^+$.  Such an edge contributes both $z_e$ and
$z_e^{-1}$ at the same vertex, so it contributes trivially to
$\partial_\Gamma$; in the oriented incidence matrix its column is therefore
the zero column. We may delete all loops without changing the image of
$\partial_\Gamma$.

We now compare our notation with that of Bevis, Hall and Katz
\cite[Section~II and Theorem~1(iii)]
{BevisHallKatz}. Their graph $\mathcal G$ is the loopless graph underlying
$W$, their vertices $v_1,\ldots,v_m$ are our vertices and their edges
$e_1,\ldots,e_n$ are our non-loop edges. Choose their incidence function so
that $f(e)=(e^+,e^-)$. Their incidence matrix $A$ then has a $1$ in the row
of the head and a $-1$ in the row of the tail. Thus $A$ is precisely the
integer exponent matrix of the restriction of $\partial_\Gamma$ to $W$.
Their Theorem~1(iii) gives unimodular matrices $S$ and $T$ for which $SAT$
has $m-1$ diagonal entries equal to $1$, one zero row and all remaining
entries equal to $0$. Hence the cokernel has one free generator, corresponding
to the single component product obstruction. We now prove the analogous
criterion directly with coefficients in an arbitrary abelian group.}

\begin{theorem}\label{t:incidence-criterion}
Let $\Gamma$ be a finite directed multigraph, let $H$ be an abelian group, and
let $d_v\in H$ for $v\in V(\Gamma)$. The system
\begin{equation}\label{eq:incidence-equation}
        (\partial_\Gamma z)_v=d_v
        \quad (v\in V(\Gamma))
\end{equation}
has a solution $z\in H^{E(\Gamma)}$ if and only if, for every connected
component $W$ of the underlying undirected graph of $\Gamma$,
\[
        \left(
        \prod_{v\in W}d_v
        \right)=1_H.
\]
\end{theorem}

\begin{proof}
	The necessity of the component product condition was shown above.

For sufficiency the equations may be solved separately on the connected
components. Fix one connected component $W$ and assume
\[
        \left(
        \prod_{v\in W}d_v
        \right)=1_H.
\]
Choose a spanning tree of the underlying undirected graph of $W$ with
root $o$, and direct the chosen tree towards $o$. We refer to this partial orientation of edges as the \emph{auxiliary orientation} or \emph{auxiliary direction}.

  If $v\ne o$, let $T_v$ be the
set of vertices in the subtree rooted at $v$, and let $e_v$ be the tree edge
from $v$ to its parent in the auxiliary orientation. Put
\[
        y_{e_v}
        =
        \left(
        \prod_{w\in T_v}d_w
        \right)^{-1}.
\]
If the direction of $e_v$ in $\Gamma$ agrees with the auxiliary direction, set
$z_{e_v}=y_{e_v}$; if the direction of $e_v$ in $\Gamma$ is opposite to the
auxiliary direction, set $z_{e_v}=y_{e_v}^{-1}$. Put $z_e=1_H$ for every edge
outside the chosen tree.

{
For our calculations, we work in the auxiliary orientation. When a tree edge has the opposite orientation in $\Gamma$, the inversion in the definition of $z_{e_v}$ compensates for the reversal. Hence we obtain the same result as in the original orientation.}

\begingroup
Let $v\ne o$, and let $C(v)$ be the set of children of $v$. The tree edges
entering $v$ are the edges $e_u$ with $u\in C(v)$, and the only tree edge
leaving $v$ is $e_v$. Hence
\[
\begin{aligned}
(\partial_\Gamma z)_v
&=\left(\prod_{u\in C(v)}y_{e_u}\right)y_{e_v}^{-1}
 =\left(\prod_{u\in C(v)}\prod_{w\in T_u}d_w^{-1}\right)
  \left(\prod_{w\in T_v}d_w\right)\\
&=\left(\prod_{u\in C(v)}\prod_{w\in T_u}d_w^{-1}\right)
  \left(d_v\prod_{u\in C(v)}\prod_{w\in T_u}d_w\right)
 =d_v.
\end{aligned}
\]
\endgroup
At the root there is no outgoing tree edge. Thus
\[
(\partial_\Gamma z)_o
=\prod_{u\in C(o)}y_{e_u}
=\left(\prod_{w\in W\setminus\{o\}}d_w\right)^{-1}
=d_o.
\]
The last equality follows from
\[
        \left(
        \prod_{w\in W}d_w
        \right)=1_H.
\]

Thus we constructed a solution on $W$. Repeating the construction on each
component gives $z\in H^{E(\Gamma)}$ with $\partial_\Gamma z=d$.
\end{proof}

\begin{definition}\label{d:defect-graph}
Let $A$ be a finite set and let $T,K\in\operatorname{Sym}(A)$. The
\emph{$(T,K)$-defect graph} $\Gamma(T,K)$ is the directed multigraph with
vertex set $A$ and, for each $x\in A$, two directed edges
\(R_x\colon x\longrightarrow xK^{-1}, \quad C_x\colon xT^{-1}\longrightarrow xK^{-1}.\)
\end{definition}
We note that parallel edges will be treated as distinct in all of the following calculations, and that loops will always provide self-cancelling contributions to the involved products.

\begin{lemma}\label{l:defect-graph-components}
The connected components of the underlying undirected graph of $\Gamma(T,K)$ are exactly the orbits of the group
$\langle T,K\rangle$ on $A$.
\end{lemma}

\begin{proof}
First, every edge of $\Gamma(T,K)$ has its two endpoints in the same
$\langle T,K\rangle$-orbit: the edge $R_x$ joins $x$ to $xK^{-1}$, and the edge
$C_x$ joins $xT^{-1}$ to the same vertex $xK^{-1}$. Therefore each connected
component is contained in a $\langle T,K\rangle$-orbit.

For the reverse inclusion, fix $x\in A$. The edge $R_x$ shows that $x$ is
connected to $xK^{-1}$. Applying the same observation to $xK$ shows that $x$
is also connected to $xK$. Moreover, $R_x$ connects $x$ to $xK^{-1}$, while
$C_x$ connects $xT^{-1}$ to the same vertex $xK^{-1}$; hence $x$ is connected
to $xT^{-1}$. Applying this last conclusion to $xT$ shows that $x$ is also
connected to $xT$. Thus every connected component is closed under $K^{\pm1}$
and $T^{\pm1}$. It is therefore a union of $\langle T,K\rangle$-orbits.
Combined with the first paragraph, this proves the claim.
\end{proof}

\begin{lemma}\label{l:defect-equations}
Let $C$ be a finite cyclic $2$-group, let $A$ be a finite index set, and let $T$ and
$K$ be permutations of $A$. For each $a\in A$, let $b_a\in C$. Then the
system
\begin{equation}\label{eq:defect-equation-system}
        r_{aK}r_a^{-1}c_{aK}c_{aT}^{-1}=b_a
        \quad (a\in A)
\end{equation}
has a solution with $r_a,c_a\in C$ if and only if, for every orbit $B$ of the
group $\langle T,K\rangle$ on $A$, one has
\[
        \left(\prod_{a\in B} b_a\right)=1.
\]
\end{lemma}

\begin{proof}
First suppose that a solution exists, and let $B$ be an orbit of
$\langle T,K\rangle$ on $A$. Since $B$ is invariant under both $T$ and $K$, the
maps $a\mapsto aK$ and $a\mapsto aT$ are permutations of $B$. Multiplying
\eqref{eq:defect-equation-system} over all $a\in B$ gives
\[
        \left(\prod_{a\in B} r_{aK}\right)
        \left(\prod_{a\in B} r_a^{-1}\right)
        \left(\prod_{a\in B} c_{aK}\right)
        \left(\prod_{a\in B} c_{aT}^{-1}\right)
        =
        \left(\prod_{a\in B} b_a\right).
\]
The first two products on the left are inverse to one another, because
$a\mapsto aK$ is a permutation of $B$. The last two products are inverse to one
another, because both $a\mapsto aK$ and $a\mapsto aT$ are permutations of $B$.
Thus the left-hand side is $1$, and hence $\left(\prod_{a\in B}b_a\right)=1$.

Conversely, assume that $\left(\prod_{a\in B}b_a\right)=1$ for every orbit $B$ of
$\langle T,K\rangle$. Consider the defect graph $\Gamma(T,K)$ of
Definition~\ref{d:defect-graph}. Assign the variable $r_x$ to the edge $R_x$
and the variable $c_x$ to the edge $C_x$. At a vertex $a\in A$, the incoming
edges are exactly $R_{aK}$ and $C_{aK}$, while the outgoing edges are exactly
$R_a$ and $C_{aT}$. Therefore the incidence equation
$(\partial_{\Gamma(T,K)}z)_a=b_a$ is
\(r_{aK}c_{aK}r_a^{-1}c_{aT}^{-1}=b_a.\)
The group $C$ is cyclic, hence abelian, so this incidence equation is
precisely the equation indexed by $a$ in \eqref{eq:defect-equation-system}.
By Lemma~\ref{l:defect-graph-components}, the connected components of
$\Gamma(T,K)$ are precisely the orbits of $\langle T,K\rangle$. The assumed
orbit product conditions are therefore exactly the component product conditions
in Theorem~\ref{t:incidence-criterion}. Applying that theorem with $H=C$ and
$d_a=b_a$ gives elements $r_x,c_x\in C$ satisfying
\eqref{eq:defect-equation-system}.
\end{proof}

\begingroup
We introduce now the notation needed for the next theorem. Let
$S=\mathcal M^0(G,I,\Lambda,P)$, where $G$ has no complete mapping, $|I|$
and $|\Lambda|$ are odd and the pattern $Q$ of $P$ satisfies the equivalent
conditions of Theorem~\ref{c:Rees0}. Fix $N\unlhd G$ of odd order such that
$C=G/N$ is a non-trivial cyclic $2$-group, let $\varepsilon$ be the unique
element of order $2$ in $C$ and write $\bar p_{\lambda j}=Np_{\lambda j}$
when $p_{\lambda j}\ne0$. By $A$ we will denote $I\times\Lambda$.

Let $\Theta$ and $\Psi$ be $Q$-admissible routings of $A$. For
$a=(i,\lambda)$, write $a\Theta=(j_a,\mu_a)$ and
$a\Psi=(j'_a,\mu'_a)$, and put $u_a=\bar p_{\lambda,j_a}$,
$v_a=\bar p_{\lambda,j'_a}$, $T=\Psi\Theta^{-1}$ and
$K=\Pi_\Psi\Pi_\Theta^{-1}$. The pair $(\Theta,\Psi)$ is said to be
\emph{defect-compatible modulo $N$} if, for every orbit $B$ of
$\langle T,K\rangle$ on $A$,
\begin{equation}\label{eq:Rees0-odd-defect-condition}
        \prod_{a\in B}v_a
        =
        \varepsilon^{|B|}\prod_{a\in B}u_a.
\end{equation}

We have the following theorem.

\begin{theorem}\label{t:Rees0-odd-defect-switch}
If there is a defect-compatible pair of $Q$-admissible routings, then $S$
has a complete mapping.
\end{theorem}
\endgroup

\begin{proof}
Let $(\Theta,\Psi)$ be a defect-compatible pair of $Q$-admissible routings.
By Lemma~\ref{l:cyclic2-quotient}, a subgroup $N$ with the stated properties
exists. Since $N$ has odd order, it has a complete mapping. Lemma~\ref{l:factorL}
therefore reduces the proof to the quotient Rees $0$-matrix semigroup
\(\bar S=\mathcal M^0(C,I,\Lambda,\bar P).\)
Thus it is enough to construct a complete mapping in the quotient case. By abuse of notation, we may assume that $G$ denotes the cyclic $2$-group $C$, the
entries of $P$ are in $C \cup \{0\}$, and that
$\varepsilon$ is the unique element of order $2$ in $G$.

We first record the consequences of admissibility that are used in the
construction. If $a=(i,\lambda)$, $a\Theta=(j_a,\mu_a)$ and
$a\Psi=(j'_a,\mu'_a)$, then $Q_{\lambda,j_a}=Q_{\lambda,j'_a}=1$; hence
$u_a=p_{\lambda,j_a}$ and $v_a=p_{\lambda,j'_a}$ are non-zero sandwich entries.
Moreover, $\Theta$ and $\Psi$ are bijections, and so are $\Pi_\Theta$ and
$\Pi_\Psi$. Hence $T$ and $K$ are permutations of $A$, and their defining identities are
\((aT)\Theta=a\Psi, \quad (aK)\Pi_\Theta=a\Pi_\Psi \quad (a\in A).\)

For each $a\in A$, put
\(b_a=\varepsilon v_a u_{aK}^{-1}.\)
Let $B$ be an orbit of $\langle T,K\rangle$ on $A$. Since $K$ preserves $B$,
\[
        \left(\prod_{a\in B}u_{aK}\right)=\left(\prod_{a\in B}u_a\right).
\]
Using the hypothesis \eqref{eq:Rees0-odd-defect-condition}, we obtain
\[
\begin{aligned}
        \left(\prod_{a\in B} b_a\right)
        &=\left(\prod_{a\in B}\left(\varepsilon v_a u_{aK}^{-1}\right)\right)\\
        &=\left(
        \varepsilon^{|B|}
        \left(\prod_{a\in B} v_a\right)
        \left(\prod_{a\in B}u_{aK}\right)^{-1}
        \right)\\
        &=\left(
        \varepsilon^{|B|}
        \left(\prod_{a\in B} v_a\right)
        \left(\prod_{a\in B}u_a\right)^{-1}
        \right)\\
        &=1.
\end{aligned}
\]
Lemma~\ref{l:defect-equations} therefore gives elements $r_a,c_a\in G$, for
$a\in A$, such that
\(r_{aK}r_a^{-1}c_{aK}c_{aT}^{-1} = \varepsilon v_a u_{aK}^{-1} \quad (a\in A).\)
Since $G$ is abelian and $\varepsilon^2=1$, this is equivalent to
\begin{equation}\label{eq:Rees0-odd-fill-equation}
        r_av_ac_{aT}
        =
        \varepsilon r_{aK}u_{aK}c_{aK}
        \quad (a\in A).
\end{equation}

We now define a bijection $f\colon \bar S\longrightarrow\bar S$. Set $0f=0$. For
$a=(i,\lambda)\in A$, write $a\Theta=(j_a,\mu_a)$ and consider the group table of the local packet $H_a\times H_{a\Theta}$, given by
\((x,y)\longmapsto xu_ay.\)
Apply Lemma~\ref{l:cyclic2-near-transversal} with omitted row $r_a$ and omitted
column $c_a$. This gives a bijection
\(\alpha_a\colon G\setminus\{r_a\}\longrightarrow G\setminus\{c_a\}\)
such that
\(\{\,xu_a(x\alpha_a)\mid x\in G\setminus\{r_a\}\,\} = G\setminus\{\varepsilon r_au_ac_a\}.\)
For $x\ne r_a$, define
\((i,x,\lambda)f=(j_a,x\alpha_a,\mu_a).\)

It remains to map the one element omitted from each local packet. Write
$a\Psi=(j'_a,\mu'_a)$. Since $(aT)\Theta=a\Psi$, the $\mathcal{H}$-class indexed by
$a\Psi$ is the class used as the second argument in the near-transversal indexed by $aT$, and that
near-transversal omits precisely the element with group coordinate $c_{aT}$.
We therefore define
\((i,r_a,\lambda)f=(j'_a,c_{aT},\mu'_a).\)

This defines a bijection on the non-zero elements. Indeed, the
near-transversal indexed by $a$ maps $H_{i,\lambda} \setminus
\{(i,r_a,\lambda)\}$ onto the target $\mathcal H$-class indexed by
$a\Theta$, with the single element of group coordinate $c_a$ removed. The
remaining source elements are sent to the omitted elements of the target classes
indexed by $(aT)\Theta$. Since $T$ is a permutation of $A$, every omitted target
element is filled exactly once. Together with $0f=0$, this proves that $f$ is a
bijection of $\bar S$.

It remains to check the orthomorphism. For $x\ne r_a$, the products
\((i,x,\lambda)(j_a,x\alpha_a,\mu_a)\)
run through every element of the product $\mathcal H$-class indexed by
\(a\Pi_\Theta=(i,\mu_a)\)
except the element with group coordinate
\(\varepsilon r_au_ac_a.\)
Thus the near-transversal indexed by $a$ leaves exactly one product missing.
The residual element of $H_{i,\lambda}$ gives
\((i,r_a,\lambda)(j'_a,c_{aT},\mu'_a) = (i,r_av_ac_{aT},\mu'_a).\)
By the definition of $K$,
\((aK)\Pi_\Theta=a\Pi_\Psi=(i,\mu'_a),\)
so this residual product lies in the same product $\mathcal H$-class as the
missing product from the local packet indexed by $aK$. Equation
\eqref{eq:Rees0-odd-fill-equation} says that its group coordinate is exactly
that missing coordinate:
\(r_av_ac_{aT} = \varepsilon r_{aK}u_{aK}c_{aK}.\)
Since $K$ is a permutation of $A$, every omitted product is filled exactly once.
Therefore the map $x\mapsto x\cdot xf$ is a bijection on the non-zero elements
of $\bar S$; it also sends $0$ to $0$. Hence $f$ is a complete mapping of
$\bar S$.

Finally, Lemma~\ref{l:factorL} lifts this complete mapping from $\bar S$ to
$S$. Therefore $S$ has a complete mapping.
\end{proof}

The previous results can be further simplified if the defect graph is connected.

\begin{cor}\label{c:Rees0-odd-connected-defect}
With the hypotheses and notation of Theorem~\ref{t:Rees0-odd-defect-switch},
suppose that the group $\langle T,K\rangle$
acts transitively on $A=I\times\Lambda$.
If
\[
        \left(\prod_{a\in A} v_a\right)
        =
        \varepsilon
        \left(\prod_{a\in A} u_a\right),
\]
then $S$ has a complete mapping.
\end{cor}

\begin{proof}
Since $|I|$ and $|\Lambda|$ are both odd, the set $A=I\times\Lambda$ has odd
cardinality. Thus
\(\varepsilon^{|A|}=\varepsilon.\)
The result is therefore the transitive-orbit case of
Theorem~\ref{t:Rees0-odd-defect-switch}.
\end{proof}

We can obtain further simplifications if we restrict $\Psi$ in terms of $\Theta$.

\begin{definition}\label{d:Rees0-general-residual-routing}
Let
\(S=\mathcal M^0(G,I,\Lambda,P)\)
be a finite Rees $0$-matrix semigroup, let $Q$ be the pattern of $P$, and put
$A=I\times\Lambda$. Let $\Theta$ be a $Q$-admissible routing. For
$a=(i,\lambda)\in A$, write
\(a\Theta=(j_a,\mu_a).\)
A permutation $R$ of $A$ is called residual for $\Theta$ if
\(p_{\lambda,j_{aR}}\ne0 \quad(a=(i,\lambda)\in A).\)
For such $R$ define
\(a\mu_R=(i,\mu_{aR}) \quad(a=(i,\lambda)\in A).\)
\end{definition}

\begingroup
Retain the notation of Definition~\ref{d:Rees0-general-residual-routing}
and suppose that $G$ has no complete mapping. Fix $N\unlhd G$ of odd order
such that $C=G/N$ is a non-trivial cyclic $2$-group, let $\varepsilon$ be
the unique element of order $2$ in $C$ and write
$\bar p_{\lambda j}=Np_{\lambda j}$ when $p_{\lambda j}\ne0$. If $R$ is
residual for $\Theta$ and $a=(i,\lambda)$, put
$u_a=\bar p_{\lambda,j_a}$ and $w_a=\bar p_{\lambda,j_{aR}}$. The
permutation $R$ is said to be \emph{cycle-compatible modulo $N$} if
$\mu_R$ is a bijection and, for every cycle $B$ of $R$,
\begin{equation}\label{eq:general-residual-cycle-condition}
        \prod_{a\in B}w_a
        =
        \varepsilon^{|B|}\prod_{a\in B}u_a.
\end{equation}

The next theorem gives the residual-routing criterion.

\begin{theorem}\label{t:Rees0-general-residual-routing}
With the notation above, if $R$ is cycle-compatible modulo $N$, then $S$
has a complete mapping.
\end{theorem}
\endgroup

\begin{proof}
Let $R$ be a cycle-compatible residual permutation.
By Lemma~\ref{l:cyclic2-quotient}, such a subgroup $N$ exists. Since $N$ has
odd order, it has a complete mapping. As in the previous theorem, and by Lemma~\ref{l:factorL}, it suffices
to prove the result for
$ \mathcal M^0(C,I,\Lambda,\bar P)$. 

For each $a\in A$, consider the $C$-equation
              \begin{equation}\label{eq:general-residual-column-equation}
        c_{aR}c_a^{-1}=\varepsilon w_a^{-1}u_a.
\end{equation}
This is an incidence equation on the directed graph whose edges are
$a\to aR$. On a cycle $B$ of $R$, the product of the right-hand sides in
\eqref{eq:general-residual-column-equation} is
\[
        \left(\prod_{a\in B}\left(\varepsilon w_a^{-1}u_a\right)\right)
        =
        \varepsilon^{|B|}
        \left(\prod_{a\in B}w_a\right)^{-1}
        \left(\prod_{a\in B}u_a\right),
\]
which is $1$ by \eqref{eq:general-residual-cycle-condition}. Theorem~\ref{t:incidence-criterion}
therefore gives a solution of \eqref{eq:general-residual-column-equation}.
Put
\(r_a=\varepsilon u_a^{-1}c_a^{-1}.\)
Then
\begin{equation}\label{eq:general-residual-row-choice}
        \varepsilon r_au_ac_a=1.
\end{equation}

For $a=(i,\lambda)$, apply Lemma~\ref{l:cyclic2-near-transversal} to the table of the local packet on $H_a \times H_{a\Theta}$, given by
\((x,y)\longmapsto xu_ay \quad(x,y\in C),\)
with omitted row $r_a$ and omitted column $c_a$. By
\eqref{eq:general-residual-row-choice}, the omitted product is the identity
of $C$. Hence there is a bijection
\(\alpha_a\colon C\setminus\{r_a\}\longrightarrow C\setminus\{c_a\}\)
such that
\(\{xu_a(x\alpha_a)\mid x\in C\setminus\{r_a\}\}=C\setminus\{1\}.\)

Define $f$ on $\mathcal M^0(C,I,\Lambda,\bar P)$ by $0f=0$ and, for
$a=(i,\lambda)$, by \((i,x,\lambda)f=(j_a,x\alpha_a,\mu_a)\), for all \(x\ne r_a\), while \((i,r_a,\lambda)f=(j_{aR},c_{aR},\mu_{aR})\).
The map $f$ is a bijection. Indeed, the first line maps all but one element of
$H_a$ onto the $\mathcal H$-class indexed by $a\Theta$, omitting exactly the
coordinate $c_a$. The omitted coordinate in the class indexed by $a\Theta$ is
filled by the residual element from the class indexed by $aR^{-1}$. Since
$\Theta$ and $R$ are permutations of $A$, every non-zero element is used once.

The product map is also bijective. The products coming from the first line run
through all non-identity coordinates in the product $\mathcal H$-class indexed
by $a\Pi_\Theta$. The residual product from $a$ has group coordinate
\(r_aw_ac_{aR} = \varepsilon u_a^{-1}c_a^{-1}w_ac_{aR} =1,\)
where the last equality follows from
\eqref{eq:general-residual-column-equation}. This residual product lies in the
product $\mathcal H$-class indexed by $a\mu_R$. Since $\mu_R$ is a bijection,
the residual products fill precisely the identity coordinate elements omitted
by the near-transversals. Thus $x\mapsto x\cdot xf$ is a bijection, and $f$ is
a complete mapping of the quotient Rees $0$-matrix semigroup. Lemma~\ref{l:factorL}
lifts it to $S$.
\end{proof}

\begingroup
We introduce now the notation and the external results needed for the next
theorem. Let $S=\mathcal M^0(G,\mathcal K,\mathcal Y,P)$, where $G$ has no
complete mapping and $\kappa=|\mathcal K|$ and $\nu=|\mathcal Y|$ are odd.
Fix $N\unlhd G$ of odd order such that $C=G/N$ is a non-trivial cyclic
$2$-group and let $\varepsilon$ be the unique element of order $2$ in $C$.
For $I\in\mathcal Y$ and $Q\in\mathcal K$, write $I\sim Q$ when
$p_{I,Q}\ne0$, and write $\bar p_{I,Q}=Np_{I,Q}$ when $I\sim Q$. We regard
$\{(Q,I):I\sim Q\}$ as a bipartite support relation on
$\mathcal K\times\mathcal Y$. A positive balanced support weighting is a
positive weighting in the sense defined before Theorem~\ref{c:Rees0}, with
margins $\nu$ on $\mathcal K$ and $\kappa$ on $\mathcal Y$.

Let $A$ be a real matrix. The matrix $A$ is said to be \emph{totally
unimodular} if every square submatrix has determinant $0$, $1$ or $-1$. A
polyhedron is said to be \emph{integral} if all its vertices are integral.

We first prove the following fact.

\begin{theorem}
The oriented node-edge incidence matrix of a finite directed graph is totally
unimodular.
\end{theorem}

\begin{proof}
Let $A$ be a square submatrix of an oriented incidence matrix. We use
induction on the order of $A$. If a column of $A$ has at most one non-zero
entry, expansion along that column reduces the assertion to a smaller
submatrix. Otherwise every column of $A$ contains both non-zero entries of
the corresponding column of the full oriented incidence matrix, namely one
$1$ and one $-1$. Therefore the sum of the rows of $A$ is zero and
$\det A=0$. The result follows.
\end{proof}

We shall use the following standard consequence of total unimodularity.

\begin{theorem}
If $A$ is totally unimodular and $b$ is integral, then every vertex of the
polyhedron $\{x\in\mathbb R^n:x\ge0,\ Ax=b\}$ is integral.
\end{theorem}

\begin{proof}
Let $x$ be a vertex. Among the equations $Ax=b$ and the active coordinate
equations $x_j=0$, choose $n$ linearly independent equations. Their coefficient
matrix $B$ is obtained by adjoining unit rows to rows of $A$. The matrix $B$ is
totally unimodular, since a determinant involving a unit row expands to a
minor of $A$. Since $B$ is nonsingular, $\det B=\pm1$. Its right-hand side is
integral, and Cramer's rule gives $x\in\mathbb Z^n$.
\end{proof}

We shall also use the following regular case of K\H{o}nig's line-colouring theorem.

\begin{quotedtheorem}[K\H{o}nig]
Every finite $d$-regular bipartite multigraph is the disjoint union of $d$
perfect matchings.
\end{quotedtheorem}

\begin{proof}
Let $A$ be a set of vertices in one bipartite class. Double-counting the
edges incident with $A$ gives $d|A|\le d|N(A)|$. Hall's theorem gives a
perfect matching. Removing this matching leaves a $(d-1)$-regular bipartite
multigraph, and hence induction completes the proof. The theorem is proved.
\end{proof}

Let $\sigma\colon\mathcal Y\to\mathcal K$ be a map and let
$\rho\in\operatorname{Sym}(\mathcal Y)$. The pair $(\sigma,\rho)$ is said
to be an \emph{anchored pair modulo $N$}  if $I\sim I\sigma$ and
$I\sim(I\rho)\sigma$, for all $I\in\mathcal Y$, and the following condition
holds on every cycle of $\rho$. For $I_0\in\mathcal Y$, let $m$ be the
least positive integer such that $I_0\rho^m=I_0$, and put
$I_t=I_0\rho^t$, for $0\le t\le m$. Then
$(I_0,I_1,\ldots,I_{m-1})$ is a cycle of $\rho$, and we require
\begin{equation}\label{eq:general-anchored-cycle-condition}
        \prod_{t=0}^{m-1}\bar p_{I_t,I_{t+1}\sigma}
        =
        \varepsilon^m\prod_{t=0}^{m-1}\bar p_{I_t,I_t\sigma}.
\end{equation}
The condition is independent of the initial member of the cycle.

So we have the following theorem.

\begin{theorem}\label{t:Rees0-general-anchored}
With the notation above, if the support of $P$ admits a positive balanced
support weighting and an anchored pair modulo $N$, then $S$ has a complete
mapping.
\end{theorem}
\endgroup

\begin{proof}
Let $h$ be a positive balanced support weighting and let $(\sigma,\rho)$ be an anchored pair modulo $N$.
We will show that we can apply Theorem~\ref{t:Rees0-general-residual-routing}.

Choose $Q_\ast\in\mathcal K$, and choose a fixed-point-free involution
\(\lambda\colon \mathcal K\setminus\{Q_\ast\}\longrightarrow \mathcal K\setminus\{Q_\ast\}.\)
For $Q\in\mathcal K$, put
\(c_\sigma(Q)=|\{I\in\mathcal Y\mid I\sigma=Q\}|.\)
Define
\(h'_{Q,I}=h_{Q,I}-\delta_{Q,I\sigma}.\)
Then $h'_{Q,I}\ge0$, and $h'_{Q,I}=0$ unless $I\sim Q$. The row and column sums of the corresponding matrix are
\[
\left(\sum_{I\in\mathcal Y}h'_{Q,I}\right)=\nu-c_\sigma(Q),
        \quad
        \left(\sum_{Q\in\mathcal K}h'_{Q,I}\right)=\kappa-1.
\]
We will show that we can replace $h'$ with an integer-valued matrix.

Let $\Gamma$ be the bipartite graph with vertex classes $\mathcal K$ and
$\mathcal Y$ and edge set
\(E(\Gamma)=\{(Q,I)\in\mathcal K\times\mathcal Y\mid I\sim Q\}.\)
Consider the polytope of non-negative edge weights
$x=(x_{Q,I})_{(Q,I)\in E(\Gamma)}$ satisfying the constraints
\[
\sum_{I:I\sim Q}x_{Q,I}=\nu-c_\sigma(Q)
        \quad(Q\in\mathcal K),
\]
and
\[
\sum_{Q:I\sim Q}x_{Q,I}=\kappa-1
        \quad(I\in\mathcal Y).
\]
It is non-empty, because $(h'_{Q,I})$ is a feasible point, and it is bounded.
After multiplying the equations belonging to one vertex class by $-1$, its
constraint matrix is the oriented node-edge incidence matrix of $\Gamma$ and
is therefore totally unimodular by the incidence matrix theorem stated above.
The right-hand side is integral, so the Hoffman--Kruskal theorem stated above
shows that every vertex of this polytope is integral. Choosing one vertex and
extending it by zero outside $E(\Gamma)$ gives a non-negative integer matrix
$(u_{Q,I})$ supported on $I\sim Q$ and satisfying the same constraints.

For each $I\in\mathcal Y$ and $Q\in\mathcal K$, create $u_{Q,I}$ slots of type $(Q,I)$. For each fixed $I$, there are $\kappa-1$ such slots, so assign them bijectively
to the elements of $\mathcal K\setminus\{Q_\ast\}$. For $B \in \mathcal K$, define $A(B,I)\in\mathcal K$ by
\(A(Q_\ast,I)=I\sigma,\)
and, for $B\ne Q_\ast$, by taking $A(B,I)$ to be the first coordinate of the type of the
slot assigned to $B$ (for the fixed $I$). Then
\(I\sim A(B,I) \quad(B\in\mathcal K,\ I\in\mathcal Y).\)
For each fixed $Q\in\mathcal K$, the number of pairs $(B,I)$ with
$A(B,I)=Q$ is
\[
        c_\sigma(Q)+\left(\sum_{I\in\mathcal Y}u_{Q,I}\right)=\nu.
\]
Thus the pairs $(B,I)$ define a $\nu$-regular bipartite multigraph between two
copies of $\mathcal K$, with an edge from $B$ to $A(B,I)$ for each
$I\in\mathcal Y$. By the regular bipartite multigraph theorem stated above, this multigraph is
the union of $\nu$ perfect matchings. Index the colours by the elements of
$\mathcal Y$, and let $C(B,I)$ be the colour of the edge corresponding to
$(B,I)$.

Put $A_0=\mathcal K\times\mathcal Y$, and define
\((B,I)\Theta=(A(B,I),C(B,I)).\)

	 This is a $Q$-admissible routing.

Indeed, the support condition follows from
$I\sim A(B,I)$. For each colour $J\in\mathcal Y$, the edges of colour $J$
form a perfect matching from the left copy of $\mathcal K$ to the right copy.
Consequently, for every $(Q,J)\in\mathcal K\times\mathcal Y$, there is exactly
one pair $(B,I)$ such that
\(A(B,I)=Q, \quad C(B,I)=J.\)
Thus $\Theta$ is a bijection. Moreover,
\((B,I)\Pi_\Theta=(B,C(B,I)).\)
At each fixed left vertex $B$, the $\nu$ incident edges receive all $\nu$
colours exactly once. Hence $I\mapsto C(B,I)$ is a permutation of
$\mathcal Y$ for every $B$, and $\Pi_\Theta$ is a bijection. Both conditions
in Definition~\ref{d:Q-admissible-routing} are therefore satisfied.

 Define a permutation $R$ of $A_0$ by
\((B,I)R=(B\lambda,I) \quad(B\ne Q_\ast),\)
and
\((Q_\ast,I)R=(Q_\ast,I\rho).\)

We claim that the permutation $R$ is residual for $\Theta$. Let $a=(B,I)$. If $B\ne Q_\ast$, then
$aR=(B\lambda,I)$, and the first coordinate of $(aR)\Theta$ is
$A(B\lambda,I)$; hence $I\sim A(B\lambda,I)$. If $B=Q_\ast$, then
$aR=(Q_\ast,I\rho)$, and the first coordinate of $(aR)\Theta$ is
\(A(Q_\ast,I\rho)=(I\rho)\sigma;\)
hence $I\sim(I\rho)\sigma$ by hypothesis. These are exactly the non-zero
sandwich entry conditions in
Definition~\ref{d:Rees0-general-residual-routing}. Furthermore,
\((B,I)\mu_R=(B,C(B\lambda,I)) \quad(B\ne Q_\ast),\)
whereas
\((Q_\ast,I)\mu_R=(Q_\ast,C(Q_\ast,I\rho)).\)
For $B\ne Q_\ast$, the map $I\mapsto C(B\lambda,I)$ is a permutation of
$\mathcal Y$; for $B=Q_\ast$, the map is the composition of $\rho$ with the
permutation $J\mapsto C(Q_\ast,J)$. Thus the associated map $\mu_R$ is a bijection.

It remains to check \eqref{eq:general-residual-cycle-condition}. The two-cycles
$(B,I)\leftrightarrow(B\lambda,I)$, with $B\ne Q_\ast$, contribute the same two
quotient sandwich entries to both sides of
\eqref{eq:general-residual-cycle-condition}, and $\varepsilon^2=1$. On the
row $Q_\ast$, the cycles of $R$ are exactly the cycles of $\rho$, and
\eqref{eq:general-residual-cycle-condition} is precisely
\eqref{eq:general-anchored-cycle-condition}. The result follows from Theorem~\ref{t:Rees0-general-residual-routing}.
\end{proof}

We will later need to consider the special case of $S_3$. Here, the quotient by $A_3$ records only the sign of a sandwich entry.

\begin{cor}\label{c:Rees0-S3-anchored}
Let
\(S=\mathcal M^0(S_3,\mathcal K,\mathcal Y,P)\)
be a finite Rees $0$-matrix semigroup satisfying the hypotheses of
Theorem~\ref{t:Rees0-general-anchored}, except that the quotient entries
$\bar p_{I,Q}$ are replaced by their signs
\(s(Q,I)=\sgn(p_{I,Q}).\)
If, for every cycle $(I_0,I_1,\ldots,I_{m-1})$ of $\rho$,
\begin{equation}\label{eq:S3-anchored-cycle-condition}
        (-1)^m
        \left(
        \prod_{t=0}^{m-1}
        \frac{s((I_{t+1})\sigma,I_t)}{s((I_t)\sigma,I_t)}
        \right)=1,
\end{equation}
then $S$ has a complete mapping.
\end{cor}

\begin{proof}
Take $N=A_3$. Then $S_3/N\cong\{\pm1\}$, the element of order $2$ is $-1$, and the image
of a sandwich entry is its sign. The condition
\eqref{eq:S3-anchored-cycle-condition} is exactly
\eqref{eq:general-anchored-cycle-condition} in this quotient. The result
follows from Theorem~\ref{t:Rees0-general-anchored}.
\end{proof}

In the odd-by-odd case the pattern alone does not decide the existence of a
complete mapping. Suppose, for example, that all non-zero entries of the
sandwich matrix become the identity in the cyclic $2$-group quotient. Then
$u_a=v_a=1$ for all $a$. Since
 $|I\times\Lambda|$
is odd, at least one orbit of $\langle T,K\rangle$ on $I\times\Lambda$ has odd
cardinality. On such an orbit, condition
\eqref{eq:Rees0-odd-defect-condition} would force
$1=\varepsilon$,
which is impossible. Thus the obstruction in
Theorem~\ref{t:Rees0-odd-defect-switch} depends on the images of the sandwich
matrix entries in the cyclic $2$-group quotient. If these entries do not supply
the required element $\varepsilon$, the obstruction to this defect-switch construction persists.

\begin{theorem}\label{t:converse0}
Suppose that
\begin{itemize}
\item[(a)] $G$ is a group with non-trivial cyclic Sylow $2$-subgroups;
\item[(b)] $I$ and $\Lambda$ are finite sets of odd cardinality;
\item[(c)] $P$ is obtained from a normalized matrix over $G$ in which each
entry has odd order by replacing some entries with $0$.
\end{itemize}
Then
$\mathcal M^0(G,I,\Lambda,P)$
has no complete mapping.
\end{theorem}

\begin{proof}
This follows directly from Theorem~\ref{t:converse} and
Proposition~\ref{p:adding0}.
\end{proof}

\begingroup

\section{Full linear and diagram monoids}
\label{Sec:linear-diagram}

The aim of this section is to apply the previous results to full linear
monoids and to some diagram monoids.

\subsection{Full linear monoids}

The following result was proved on the more general context of independence algebras by Fountain and Lewin
\cite[Propositions~1.3 and~1.4]{FountainLewin}.

\begin{quotedtheorem}[Fountain and Lewin]
Let $V$ be finite-dimensional and let $\alpha,\beta\in\End(V)$. Then
$\alpha\mathrel{\mathcal L}\beta$ if and only if
$\im(\alpha)=\im(\beta)$, $\alpha\mathrel{\mathcal R}\beta$ if and only if
$\ker(\alpha)=\ker(\beta)$ and $\alpha\mathrel{\mathcal D}\beta$ if and
only if $\rank(\alpha)=\rank(\beta)$. Moreover, $\mathcal D=\mathcal J$
and every positive-rank principal factor of $\End(V)$ is completely
$0$-simple.
\end{quotedtheorem}

Dalla Volta and Gavioli proved that $\operatorname{GL}_k(q)$ has a complete mapping when $k\ge2$ and $q$ is odd
\cite[Theorem~3]{DallaVoltaGavioli1997}. We record now the full classification, which follows from the
Hall--Paige theorem.

\begin{prop}\label{p:linear-group-complete}
Let $d\ge1$. The following are equivalent:
\begin{enumerate}
\item $\operatorname{GL}_d(q)$ has a complete mapping;
\item neither $d=1$ with $q$ odd nor $(d,q)=(2,2)$ holds.
\end{enumerate}
\end{prop}

\begin{proof}
Suppose first that $d=1$. Then
$\operatorname{GL}_1(q)=\mathbb F_q^\times$ is cyclic. Its order is odd when
$q$ is even, while for odd $q$ its Sylow $2$-subgroup is non-trivial and
cyclic. Thus we can now suppose that $d\ge2$. For $1\le i,j\le d$, by
$E_{ij}$ we will denote the $d\times d$ matrix whose $(i,j)$-entry is $1$
and whose remaining entries are $0$. If $q$ is odd, the two diagonal
involutions with respective first two diagonal entries $(-1,1)$ and
$(1,-1)$ generate $C_2\times C_2$. If $q>2$ is even, the subgroup
$\{I+aE_{12}:a\in\mathbb F_q\}$ is elementary abelian of order $q$. If
$q=2$ and $d\ge3$, the transvections $I+E_{12}$ and $I+E_{13}$ generate
$C_2\times C_2$. Finally, $\operatorname{GL}_2(2)\cong S_3$ has a cyclic
Sylow $2$-subgroup. The result follows from Theorem~\ref{Thm:HP}. 
\end{proof}

The aim of the next theorem is to classify the full linear monoids having a
complete mapping.

\begin{theorem}\label{t:linear-monoid}
Let $V$ have dimension $d\ge1$ over $\mathbb F_q$. Every proper
principal factor of $\End_{\mathbb F_q}(V)$ has a complete mapping. Moreover,
the following are equivalent:
\begin{enumerate}
\item $\End_{\mathbb F_q}(V)$ has a complete mapping;
\item $\operatorname{GL}_d(q)$ has a complete mapping;
\item neither $d=1$ with $q$ odd nor $(d,q)=(2,2)$ holds.
\end{enumerate}
\end{theorem}

\begin{proof}
To prove this result we determine the Rees representation and show that every proper principal
factor satisfies the equivalent conditions of Theorem~\ref{c:Rees0}.

Under our zero-adjoined convention,
the rank-$0$ principal factor is the two-element semilattice obtained from the
zero endomorphism by adjoining a new zero; its identity mapping is complete. 

Let $1\le k<d$ and put
$W=\mathbb F_q^k$. For each $k$-subspace $U$ choose an isomorphism
$\iota_U\colon W\to U$, and for each $(d-k)$-subspace $K$ choose a
surjection $\pi_K\colon V\to W$ with kernel $K$. Every rank-$k$
endomorphism with kernel $K$ and image $U$ is uniquely of the form
$\pi_Kg\iota_U$, where $g\in\operatorname{GL}_k(q)$. Thus
\(J_k^0\cong \mathcal M^0(\operatorname{GL}_k(q),\mathcal K_{d-k},\mathcal U_k,P_k),\)
and, with our right-action convention, $p_{U,K}=\iota_U\pi_K$ when
$V=U\oplus K$, while $p_{U,K}=0$ otherwise.

The complement graph is square because the numbers of $k$-subspaces and
$(d-k)$-subspaces are equal. It is regular because $\operatorname{GL}(V)$
is transitive on both classes and preserves complementarity. Every subspace
has a complement and, since $0<k<d$, there are also non-complementary pairs.
Thus, if $N$ is the number of vertices on either side and $c$ is the degree,
then $0<c<N$. The value $N/c$ on every support edge is a positive balanced
support weighting. Therefore the pattern satisfies the equivalent conditions
of Theorem~\ref{c:Rees0}.

Hence, existence of a complete mapping is determined by the top principle factor.
If $\operatorname{GL}_k(q)$ has a complete mapping, then
$J_k^0$ has one by Theorem~\ref{c:Rees0}. By  Proposition~\ref{p:linear-group-complete}, only
two types of proper factor remain.

We first observe that two subspaces $U,U'$ of the same dimension have a common
complement. In fact, write $U=A\oplus B$ and $U'=A\oplus B'$, where
$A=U\cap U'$.
{Choose an isomorphism
$\phi\colon B\to B'$ and set
$D=\{b+b\phi:b\in B\}\le B\oplus B'$. The decomposition
$U+U'=A\oplus B\oplus B'$ shows directly that $D\cap U=D\cap U'=0$ and that
$\dim D$ is the codimension of both $U$ and $U'$ in $U+U'$. Hence $D$ is a
complement of both subspaces inside $U+U'$. If $C$ is a complement of $U+U'$
in $V$, then $D\oplus C$ is a common complement of $U$ and $U'$.} 

{Suppose first that $k=1$ and $q$ is odd. The number
$N=(q^d-1)/(q-1)=1+q+\cdots+q^{d-1}$ has the parity of $d$.} If $d$ is even,
Theorem~\ref{t:Rees0-noncomp-even} applies. Thus we can now suppose that
$d\ge3$ is odd. 

{We now apply Theorem~\ref{t:Rees0-general-anchored}.} Let $H$ be the Hall $2'$-subgroup of
$\mathbb F_q^\times$, put $C=\mathbb F_q^\times/H$ and let $\varepsilon$
be the image of $-1$. Choose independent vectors $e_0,e_1,e_2$, put
$L_i=\langle e_i\rangle$, and choose functionals $\lambda_i$ whose values on
these vectors are
\[
\begin{array}{c|ccc}
 &e_0&e_1&e_2\\ \hline
\lambda_0&1&0&-1\\
\lambda_1&1&1&0\\
\lambda_2&0&1&1
\end{array}
\]
and which vanish on the remaining basis vectors. Put
$K_i=\ker(\lambda_i)$, take $\iota_{L_i}(1)=e_i$ and take
$\pi_{K_i}=\lambda_i$. Therefore $p_{L_i,K_j}=\lambda_j(e_i)$.
{Hence the $P$-entries of the form $p_{L_i,K_j}$ correspond to the transpose of the above table. Thus the diagonal
entries $p_{L_i,K_i}$} are $1$, while $p_{L_0,K_1}=1$, $p_{L_1,K_2}=1$ and
$p_{L_2,K_0}=-1$. Let $\rho$ contain the cycle $(L_0,L_1,L_2)$ and put
$L_i\sigma=K_i$. Pair the remaining lines, choose a common complementary
hyperplane for each pair, map both lines to it under $\sigma$ and let
$\rho$ exchange them. The anchored cycle condition {\eqref{eq:general-anchored-cycle-condition}} is automatic on the
transpositions and holds on the distinguished cycle because the image of
$-1$ is $\varepsilon$. Thus Theorem~\ref{t:Rees0-general-anchored} applies.

Suppose now that $k=2$, $q=2$ and $d\ge3$. The number of $2$-subspaces is
the odd integer $(2^d-1)(2^{d-1}-1)/3$. Take $H=A_3$ in
$\operatorname{GL}_2(2)\cong S_3$. Choose
$V=\langle e_1,e_2,e_3\rangle\oplus W$ and put
\[
\begin{array}{lll}
U_0=\langle e_2,e_3\rangle,
&U_1=\langle e_1,e_3\rangle,
&U_2=\langle e_1+e_2,e_3\rangle,\\
K_0=W\oplus\langle e_1\rangle,
&K_1=W\oplus\langle e_1+e_2\rangle,
&K_2=W\oplus\langle e_2+e_3\rangle.
\end{array}
\]
Use the ordered bases $(e_2,e_3)$, $(e_1,e_3)$ and
$(e_1+e_2,e_3)$ for $U_0,U_1,U_2$, and choose $\pi_{K_i}$ such that
$p_{U_i,K_i}=1$. Then $p_{U_0,K_1}$ and $p_{U_2,K_0}$ are the identity,
while
\[
        p_{U_1,K_2}=\begin{pmatrix}1&1\\0&1\end{pmatrix}.
\]
This matrix is a transvection and has odd sign in $S_3$. Use the cycle
$(U_0,U_1,U_2)$ and pair the remaining $2$-subspaces by common complements, {as in the case $k=1$ with odd $q$.}
Theorem~\ref{t:Rees0-general-anchored} applies once again. We have proved
that every proper principal factor has a complete mapping.

We now can finish the proof of the equivalence statement. The implication (a)$\Rightarrow$(b) follows from
Corollary~\ref{Cor:units}, and (b)$\Leftrightarrow$(c) is the previous
proposition. Suppose that (b) holds. The top principal factor is
$\operatorname{GL}_d(q)$ with zero adjoined and hence has a complete mapping.
All the proper principal factors have complete mappings by (2). Therefore
Theorem~\ref{t:J-reduct} gives (a). The theorem is proved.
\end{proof}

\subsection{The partition monoid}

The aim of this subsection is to classify the partition monoids having a
complete mapping. Let $[n]'=\{1',\ldots,n'\}$ be a disjoint copy of $[n]$.
The partition monoid $\mathcal P_n$ consists of the set partitions of
$[n]\cup[n']$. Its diagrams have the points of $[n]$ in an upper row and
the points of $[n]'$ in a lower row, with points joined precisely when they
belong to the same block. {In diagrams, one usually draws only enough edges for each block to appear as a connected component of the resulting graph.}

To multiply two partitions, place the first
diagram above the second, identify the two middle rows and read the connected
components meeting the outer rows. A block meeting both outer rows is said
to be \emph{transversal}, and the rank is the number of transversal blocks.
These definitions are given in
\cite[Section~2.1]{EastMitchellRuskucTorpey}.

\begingroup
\color{black}
For example, the diagram below represents the rank-$2$ partition
\[
\alpha=\{\{1,2,1'\},\{3,3',4'\},\{4\},\{2'\}\}.
\]
\endgroup
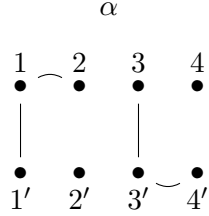
\begin{center}
\begin{tikzpicture}[x=0.78cm,y=1.15cm]
\node at (1.5,1.9) {$\alpha$};
\foreach \x/\lab in {0/1,1/2,2/3,3/4}{
  \node (a\lab) at (\x,1) {$\bullet$};
  \node at (\x,1.28) {$\lab$};
  \node (ap\lab) at (\x,0) {$\bullet$};
  \node at (\x,-0.28) {$\lab'$};
}
\draw (a1) to[bend left=28] (a2);
\draw (a1) -- (ap1);
\draw (a3) -- (ap3);
\draw (ap3) to[bend right=28] (ap4);
\end{tikzpicture}
\end{center}

Proposition~2.1 and Remark~2.2(v)--(vii) of the same
paper give the description of the $\mathcal J$-classes, their regularity and
the maximal subgroups. Thus we have the following theorem.

\begin{quotedtheorem}[East, Mitchell, Ru\v{s}kuc and Torpey
{\cite[Proposition~2.1 and Remark~2.2(vii)]{EastMitchellRuskucTorpey}}]
In $\mathcal P_n$, Green's $\mathcal J$-classes are the rank classes and are
regular. A maximal subgroup in rank $r$ is isomorphic to $S_r$. 
\end{quotedtheorem}

A set partition of $[n]$ together with $r$ distinguished blocks is said to
be an \emph{$r$-marked partition}. We order the marked blocks by their least
elements. An $r$-subset meeting every marked block exactly once is said to be
a \emph{transversal} of the marked partition.

\begingroup
For example,
$I=(\{1,3\}^{\bullet},\{2,4\}^{\bullet},\{5\})$ is a $2$-marked
partition and $A=\{1,4\}$ is a transversal. In the diagram below the two
marked blocks are joined by arcs, the unmarked singleton is isolated and the
circles mark the selected points of the transversal $A$, not the
marked blocks themselves.
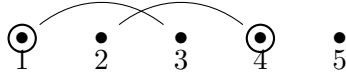
\begin{center}
\begin{tikzpicture}[x=1.05cm,y=0.75cm]
\foreach \x/\lab in {0/1,1/2,2/3,3/4,4/5}{
  \node (m\lab) at (\x,0) {$\bullet$};
  \node at (\x,-0.38) {$\lab$};
}
\draw (m1) to[bend left=38] (m3);
\draw (m2) to[bend left=38] (m4);
\draw[thick] (m1) circle (0.18cm);
\draw[thick] (m4) circle (0.18cm);
\end{tikzpicture}
\end{center}
\endgroup
By $\mathcal I_{n,r}$ we will
denote the set of $r$-marked partitions. Then
\[
        |\mathcal I_{n,r}|=N_{n,r}=\sum_{j=r}^n\binom{j}{r}S(n,j).
\]
Here $S(n,j)$ is a Stirling number of the second kind.

\begingroup
\color{black}
For $D\subseteq[n]$, put $D'=\{d':d\in D\}$. If
$I\in\mathcal I_{n,r}$, define $e_I\in\mathcal P_n$ to have one
transversal block $B\cup B'$ for each marked block $B$ of $I$, and two
separate non-transversal blocks $C$ and $C'$ for each unmarked block $C$
of $I$. Thus $e_I$ is a projection of rank $r$.
\endgroup

\begin{samepage}
\begingroup
\color{black}
For example, if
\[
I=(\{1,2\}^{\bullet},\{3\}^{\bullet},\{4\}),
\]
then
\[
        e_I=\{\{1,2,1',2'\},\{3,3'\},\{4\},\{4'\}\}.
\]
The corresponding diagram is shown below.
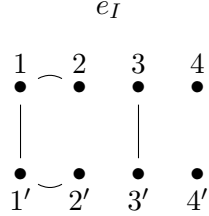
\endgroup

\begin{center}
\begin{tikzpicture}[x=0.78cm,y=1.15cm]
\node at (1.5,1.9) {$e_I$};
\foreach \x/\lab in {0/1,1/2,2/3,3/4}{
  \node (b\lab) at (\x,1) {$\bullet$};
  \node at (\x,1.28) {$\lab$};
  \node (bp\lab) at (\x,0) {$\bullet$};
  \node at (\x,-0.28) {$\lab'$};
}
\draw (b1) to[bend left=28] (b2);
\draw (bp1) to[bend right=28] (bp2);
\draw (b1) -- (bp1);
\draw (b3) -- (bp3);
\end{tikzpicture}
\end{center}
\end{samepage}

For $I,Q\in\mathcal I_{n,r}$, form the middle-row graph obtained by
identifying the lower row of $e_I$ with the upper row of $e_Q$. Write
$I\sim Q$ if every component which meets a marked block contains exactly
one marked block of $I$ and exactly one marked block of $Q$. When
$I\sim Q$, these components induce a bijection from the ordered marked
blocks of $I$ to those of $Q$; by $p_{I,Q}\in S_r$ we will denote the
corresponding permutation. Put $p_{I,Q}=0$ when $I\nsim Q$, and write
$P_r=(p_{I,Q})_{I,Q\in\mathcal I_{n,r}}$.

The following proposition gives the Rees representation used below.

\begin{prop}\label{p:partition-Rees-data}
Let $1\le r\le n$. The rank-$r$ principal factor of $\mathcal P_n$ is
isomorphic to
\(\mathcal M^0(S_r,\mathcal I_{n,r},\mathcal I_{n,r},P_r).\)
Moreover, the support relation is reflexive and symmetric, and
$p_{I,I}=1$, for all $I\in\mathcal I_{n,r}$.
\end{prop}

\begin{proof}
Let $J$ be the rank-$r$ $\mathcal J$-class. In
\cite[Proposition~2.1]{EastMitchellRuskucTorpey} it is proved that an
$\mathcal R$-class is determined by the upper partition and the upper blocks
meeting transversal blocks. Thus the $\mathcal R$-classes of $J$ are
indexed by $\mathcal I_{n,r}$, and reflection gives the same index set for
the $\mathcal L$-classes. In
\cite[Remark~2.2(vii)]{EastMitchellRuskucTorpey} it is proved that the
maximal subgroup at $e_I$ is $S_r$. Rees' theorem gives the displayed
representation. The product $e_Ie_Q$ has rank $r$ precisely when
$I\sim Q$, and the connections through the middle row give the permutation
$p_{I,Q}$. Reflection gives symmetry, while $e_I^2=e_I$ gives reflexivity
and the diagonal identity. The proposition is proved.
\end{proof}

We now prove a technical lemma which will be useful in the exceptional
ranks.

\begin{lemma}\label{l:marked-partition-pairing}
Let $r\in\{2,3\}$ and $n\ge r+2$. Suppose that $N_{n,r}$ is odd, put
$A_0=\{1,\ldots,r\}$, and remove three members of
$\mathcal I_{n,r}$, none of which has $A_0$ as a transversal. The remaining
members can be paired so that the two members in every pair have a common
transversal.
\end{lemma}

\begin{proof}
Let $A$ be an $r$-set different from $A_0$. We construct a marked partition,
say $C_A$, as follows. Mark the singletons belonging to $A\cap A_0$. Write
$A_0\setminus A=\{a_1<\cdots<a_s\}$ and
$A\setminus A_0=\{b_1<\cdots<b_s\}$, mark the blocks $\{a_i,b_i\}$ and
leave all the remaining points as unmarked singletons. Both $A_0$ and $A$
are transversals of $C_A$, and the partitions $C_A$ are distinct. Choose one
transversal of every remaining marked partition, initially choosing $A_0$
for every $C_A$. Whenever the fibre over $A\ne A_0$ is odd, change the
chosen transversal of $C_A$ from $A_0$ to $A$. Every fibre over
$A\ne A_0$ is then even. Since $N_{n,r}-3$ is even, the fibre over $A_0$
is also even. Pair the marked partitions within each fibre. The lemma is
proved.
\end{proof}

The aim of the next theorem is to classify the partition monoids having a
complete mapping.

\begin{theorem}\label{t:partition-monoid}
Every proper principal factor of $\mathcal P_n$ has a complete mapping.
Moreover, the following are equivalent:
\begin{enumerate}
\item $\mathcal P_n$ has a complete mapping;
\item $S_n$ has a complete mapping;
\item $n=1$ or $n\ge4$.
\end{enumerate}
\end{theorem}

\begin{proof}
The proof goes as follows:
\begin{enumerate}
\item[(1)] we start by treating all proper principal factors, but the exceptional ranks
$2$ and $3$;
\item[(2)] we treat the exceptional ranks;
\item[(3)] we prove the equivalences in the statement.
\end{enumerate}

{Regarding  (1),} the non-zero part of the rank-$0$ principal factor
is a rectangular band, and the principal factor is this rectangular band
with zero adjoined. Every element is idempotent, and hence the identity
mapping is complete.

Fix $1\le r<n$, put $N=N_{n,r}$ and let $d(I)$ be the support degree of $I$
in the rank-$r$ factor. The matrix
\[
h_{Q,I}=
\begin{cases}
N-d(I)+1,&Q=I,\\
1,&Q\ne I\text{ and }I\sim Q,\\
0,&I\nsim Q
\end{cases}
\]
has every row and column sum equal to $N$, by reflexivity and symmetry of
the support, and is positive on every support edge. The diagonal support
also gives a one-transversal of the pattern.

For $r=1$, the maximal subgroup is trivial. For $r\ge4$, the maximal
subgroup $S_r$ has a complete mapping by Theorem~\ref{Thm:HP}, since it
contains $\langle(12),(34)\rangle\cong C_2\times C_2$. These factors are
covered by Theorem~\ref{c:Rees0}, and hence (1) is proved.

We now prove (2). It remains to consider $r=2,3$. If $N$ is even,
Theorem~\ref{t:Rees0-noncomp-even} applies. Thus we can now suppose that
$N$ is odd. Since $N_{3,2}=6$ and $N_{4,3}=10$, we have $n\ge4$ for $r=2$
and $n\ge5$ for $r=3$. For $r=2$, take the following three marked
partitions, where all unlisted points lie in unmarked singeltons:
\[
\begin{array}{c|c|c}
 &\text{unmarked block}&\text{marked singleton blocks}\\ \hline
I_0&\{1,2\}&\{3\},\{4\}\\
I_1&\{1,3\}&\{2\},\{4\}\\
I_2&\{1,4\}&\{2\},\{3\}.
\end{array}
\]

\begingroup
The non-trivial parts of the corresponding projections are represented below. A horizontal arc in
each row is obtained from an unmarked block, while a vertical string originates from a marked singleton
block.
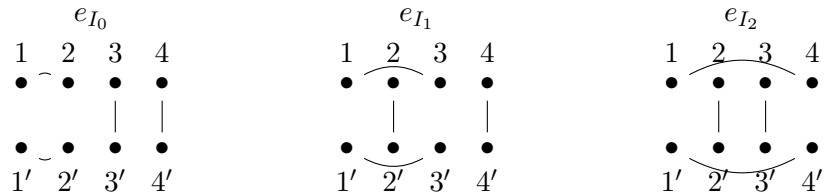
\begin{center}
\begin{tikzpicture}[x=0.62cm,y=0.86cm]
\begin{scope}[xshift=0cm]
\node at (1.5,2.0) {$e_{I_0}$};
\foreach \x/\lab in {0/1,1/2,2/3,3/4}{
  \node (a\lab) at (\x,1) {$\bullet$};
  \node at (\x,1.48) {$\lab$};
  \node (ap\lab) at (\x,0) {$\bullet$};
  \node at (\x,-0.48) {$\lab'$};
}
\draw (a1) to[bend left=28] (a2);
\draw (ap1) to[bend right=28] (ap2);
\draw (a3)--(ap3);
\draw (a4)--(ap4);
\end{scope}
\begin{scope}[xshift=4.3cm]
\node at (1.5,2.0) {$e_{I_1}$};
\foreach \x/\lab in {0/1,1/2,2/3,3/4}{
  \node (b\lab) at (\x,1) {$\bullet$};
  \node at (\x,1.48) {$\lab$};
  \node (bp\lab) at (\x,0) {$\bullet$};
  \node at (\x,-0.48) {$\lab'$};
}
\draw (b1) to[bend left=28] (b3);
\draw (bp1) to[bend right=28] (bp3);
\draw (b2)--(bp2);
\draw (b4)--(bp4);
\end{scope}
\begin{scope}[xshift=8.6cm]
\node at (1.5,2.0) {$e_{I_2}$};
\foreach \x/\lab in {0/1,1/2,2/3,3/4}{
  \node (c\lab) at (\x,1) {$\bullet$};
  \node at (\x,1.48) {$\lab$};
  \node (cp\lab) at (\x,0) {$\bullet$};
  \node at (\x,-0.48) {$\lab'$};
}
\draw (c1) to[bend left=28] (c4);
\draw (cp1) to[bend right=28] (cp4);
\draw (c2)--(cp2);
\draw (c3)--(cp3);
\end{scope}
\end{tikzpicture}
\end{center}
\endgroup

For $r=3$, add the marked
singleton $\{5\}$ to all three marked partitions. In the middle-row joins for $(I_0,I_1)$ and
$(I_1,I_2)$, the ordered marked blocks are paired in order, while for
$(I_2,I_0)$ they are interchanged. Thus Proposition~\ref{p:partition-Rees-data}
gives the sandwich permutations $1$, $1$ and $(12)$, respectively. The
diagonal entries are identities, and none of the three marked partitions has
$A_0$ as a transversal.

{By Lemma~\ref{l:marked-partition-pairing}, we can choose a pairing of $\mathcal I_{n,r}\setminus\{I_0,I_1,I_2\}$ in which all paired members have a common transversal. We again apply Theorem~\ref{t:Rees0-general-anchored}.}

 If $I,I' \in \mathcal I_{n,r} \setminus\{I_0,I_1,I_2\}$ are paired with common transversal $A$,
let $Q_A$ be the discrete marked partition with marked singletons indexed by
$A$. Put $I\sigma=I'\sigma=Q_A$ and let $\rho$ exchange $I$ and $I'$. On
the three distinguished indices put $I_i\sigma=I_i$ and let $\rho$ act as
$(I_0,I_1,I_2)$. For $r=2$, take $1\unlhd S_2$; for $r=3$, take
$A_3\unlhd S_3$. The anchored cycle condition {\eqref{eq:general-anchored-cycle-condition}}  is automatic on the
transpositions and holds on the distinguished cycle because its signs are
$1$, $1$ and $-1$. Therefore Theorem~\ref{t:Rees0-general-anchored} applies.
It is proved that the rank-$2$ and rank-$3$ proper factors have complete
mappings, and hence (2) is proved.

We finish by proving (3). The implication (a)$\Rightarrow$(b) follows from
Corollary~\ref{Cor:units}, and (b)$\Leftrightarrow$(c) follows from
Theorem~\ref{Thm:HP}. Suppose that (b) holds. The top principal factor is
$S_n$ with zero adjoined and hence has a complete mapping. Every proper
principal factor has a complete mapping by (1) and (2). Therefore
Theorem~\ref{t:J-reduct} gives (a). The theorem is proved.
\end{proof}

\subsection{Other diagram monoids}

{In this subsection, we assume that the reader has} a basic knowledge of the planar partition
monoid, the Motzkin monoid and the Jones monoid. As reference we suggest
\cite[Section~2.2]{EastMitchellRuskucTorpey}. A semigroup is said to be
\emph{aperiodic} if all its subgroups are trivial. A regular $*$-semigroup
is a semigroup with an involutory anti-automorphism $x\mapsto x^*$ satisfying
$x=xx^*x$, for all $x$, and an element $p$ is said to be a \emph{projection}
if $p=p^*=p^2$.

{In the partition monoid, let ${}^*$ denote reflection in the horizontal axis.} In \cite[Section~2.1]{EastMitchellRuskucTorpey}, 
the partition monoid is shown to satisfy
$\alpha^{**}=\alpha$, $(\alpha\beta)^*=\beta^*\alpha^*$ and
$\alpha\alpha^*\alpha=\alpha$. In
\cite[Section~2.2]{EastMitchellRuskucTorpey}, the planar partition,
Motzkin and Jones monoids are defined as submonoids of $\mathcal P_n$ and
are stated to be closed under $*$. The three identities therefore restrict
to each of these submonoids, so they are regular $*$-monoids. Their
$\mathcal J$-classes, regularity and maximal subgroups are given by
\cite[Proposition~2.1 and Remark~2.2(iv)--(vii)]
{EastMitchellRuskucTorpey}. Thus we have the following theorem.

\begin{quotedtheorem}[East, Mitchell, Ru\v{s}kuc and Torpey]
The planar partition, Motzkin and Jones monoids are finite regular
$*$-monoids. Their $\mathcal J$-classes are the rank classes, all these
classes are regular and every maximal subgroup is trivial.
\end{quotedtheorem}

We now prove a lemma which will be useful in what follows.

\begin{lemma}\label{l:projection-matching}
In a regular $*$-semigroup, every $\mathcal R$-class and every
$\mathcal L$-class contains a unique projection, and $*$ interchanges the two
sets of classes.
\end{lemma}

\begin{proof}
Let $x$ belong to the regular $*$-semigroup. Then $xx^*$ is a projection in
the $\mathcal R$-class of $x$, and $x^*x$ is a projection in its
$\mathcal L$-class. Let $p,q$ be $\mathcal R$-related projections. Then
$pq=q$ and $qp=p$. Taking $*$ in the first equality gives $qp=q$, and hence
$p=q$. The $\mathcal L$-statement is dual, and the anti-automorphism
interchanges the $\mathcal R$- and $\mathcal L$-class sets. The lemma is
proved.
\end{proof}


\begin{prop}\label{p:aperiodic-regular-star}
Every finite aperiodic regular $*$-semigroup has a complete mapping.
\end{prop}

\begin{proof}
Let $S$ be a finite aperiodic regular $*$-semigroup. Every principal factor
is a Rees $0$-matrix semigroup over a maximal
subgroup
. Since $S$ is
aperiodic, all the maximal subgroups are trivial. By the previous lemma, the
projections give a diagonal one-transversal in the pattern of every
principal factor. The result follows from
Theorems~\ref{c:Rees0} and~\ref{t:J-reduct}. The proposition is proved.
\end{proof}

As a corollary we obtain the following result.

\begin{cor}\label{c:planar-diagram-monoids}
{The planar partition monoid, the Motzkin monoid and the Jones
monoid, the latter also called the Temperley--Lieb monoid in
\cite[Section~2.2]{EastMitchellRuskucTorpey}, have complete mappings.}
\end{cor}

\begin{proof}
By the external theorem above, these monoids are finite aperiodic regular
$*$-monoids. The result follows from the previous proposition. The corollary
is proved.
\end{proof}

\endgroup
\section[Complete mappings of full transformation monoids]{Complete mappings of the full transformation monoids $T_n$}
\label{Sec:Tn}

In this section, we classify which of the full transformation monoids $T_n$
have a complete mapping. 

By Theorem~\ref{t:J-reduct}, the problem  reduces to the principal factors of
$T_n$. The rank-1 factor is a band, while the factors of ranks at least 4 are
Rees $0$-matrix semigroups over groups with complete mappings. The rank-3 and
rank-2 factors are treated by verifying the conditions established in the
preceding sections for Rees $0$-matrix semigroups over groups without complete
mappings.

Clearly, $T_1$ has a complete mapping. Moreover, the top principal factors of
$T_2$ and $T_3$ are $S_2$ and $S_3$, respectively, with zero adjoined. Since
$S_2$ and $S_3$ have no complete mappings, neither $T_2$ nor $T_3$ has a
complete mapping. Therefore, we may assume that $n\ge4$.

{It is worth observing that the family $T_n$ was the principal
motivating test case for much of Section~\ref{Sec:Rees0-matrix-2}. In
particular, a substantial part of the general machinery developed there for
Rees $0$-matrix semigroups over groups without complete mappings was abstracted
from arguments first found while analysing the rank-$2$ and rank-$3$ principal
factors of $T_n$.}

\subsection{The common Rees pattern for the non-zero rank factors}
\label{Sec:Tn-Rees-pattern}
\label{Sec:Tn-rank-greater-three}

Let $n\ge4$, and let $T_n$ be the full transformation monoid on
$X=\{1,\ldots,n\}$. For $1\le k\le n$, let
\({J_k}:=\{f\in T_n\mid |\im(f)|=k\}\)
be the $\mathcal J$-class of all transformations of rank $k$. For
$1<k\le n$, the principal factor $J_k^0$ is completely $0$-simple.

By Rees' theorem,
 ${J_k^0}\cong \mathcal M^0(S_k,I,\Lambda,P)$,
where $I$ is the set of partitions of $X$ into $k$ non-empty blocks, $\Lambda$
is the set of all $k$-subsets of $X$, and $P=(p_{\lambda i})$ is a
$\Lambda\times I$ matrix over $S_k\cup\{0\}$ such that
\(p_{\lambda i}\ne0\quad\Longleftrightarrow\quad\lambda\text{ is a transversal of the partition }i.\)
Thus the zero-one pattern $Q_k$ of $P$ is given by
\((Q_k)_{\lambda i}=1 \quad\Longleftrightarrow\quad \lambda\text{ is a transversal of }i.\)

\begin{prop}\label{p:Tn-pattern-Hall}
Let $1<k\le n$. The pattern $Q_k$ satisfies the equivalent conditions (c) and
(d) of Theorem~\ref{c:Rees0}.
\end{prop}

\begin{proof}
It is enough to prove condition (c) of Theorem~\ref{c:Rees0}. Consider the
bipartite graph whose left vertices are the partitions of $X$ into $k$ non-empty
blocks and whose right vertices are the $k$-subsets of $X$, with a partition
joined to a $k$-set precisely when the $k$-set is a transversal of the partition.
We prove that, for every $C\subseteq\Lambda$ with $|C|=r$,
\[
        |N(C)|\ge r\frac{|I|}{|\Lambda|},
\]
where $N(C)$ is the set of partitions admitting at least one transversal in $C$.

Put $t=n-k$. A $k$-partition is obtained from $k$ singleton blocks by
distributing $t$ extra elements among these blocks. For each integer partition
$\lambda=(\lambda_1,\ldots,\lambda_\ell)$ of $t$, with $\ell\le k$, let
$\mathcal T_\lambda$ be the set of partitions whose block sizes are
\[
        \underbrace{1,\ldots,1}_{k-\ell},
        1+\lambda_1,\ldots,1+\lambda_\ell.
\]
The sets $\mathcal T_\lambda$ are disjoint and
\(I=\bigsqcup_\lambda \mathcal T_\lambda.\)
If $P\in\mathcal T_\lambda$, then the number of transversals of $P$ is
\[
        d_\lambda:=\left(\prod_{i=1}^\ell(1+\lambda_i)\right).
\]

We claim that every $k$-subset of $X$ has the same number of neighbours in
$\mathcal T_\lambda$. The symmetric group on $X$ acts transitively on
$\Lambda$, preserves the transversal relation, and acts transitively on
$\mathcal T_\lambda$, since the latter consists exactly of the partitions with a
fixed multiset of block sizes. Hence the bipartite graph between
$\mathcal T_\lambda$ and $\Lambda$ is biregular. If $d_{\lambda,R}$ denotes the
common degree on the $\Lambda$-side, then double-counting edges gives $|\mathcal T_\lambda|d_\lambda=|\Lambda|d_{\lambda,R}$.

Let $N_\lambda(C):=N(C)\cap\mathcal T_\lambda$. 
The number of edges from $C$ to $\mathcal T_\lambda$ is $r d_{\lambda,R}$, while
each partition in $N_\lambda(C)$ is incident with at most $d_\lambda$ of these
edges. Hence $r d_{\lambda,R}\le d_\lambda |N_\lambda(C)|$,
and therefore
\[
        |N_\lambda(C)|\ge r\frac{d_{\lambda,R}}{d_\lambda}
        =r\frac{|\mathcal T_\lambda|}{|\Lambda|}.
\]
Summing over all integer partitions $\lambda$ of $t$ gives
\[
        |N(C)|=\left(\sum_\lambda |N_\lambda(C)|\right)
        \ge r\left(\sum_\lambda\frac{|\mathcal T_\lambda|}{|\Lambda|}\right)
        =r\frac{|I|}{|\Lambda|}.
\]
This is condition (c) of Theorem~\ref{c:Rees0}. The equivalence with condition
(d) is part of that theorem.
\end{proof}

For $k=2,3$, we require the parities of the relevant index sets. The next result is well known.

\begin{prop}\label{p:Tn-small-rank-parity}
Let $n\ge4$. Then:
\begin{enumerate}
\item $S(n,2)$ is odd;
\item $\binom n2$ is odd if and only if $n\equiv2,3\pmod4$;
\item $\binom n3$ is odd if and only if $n\equiv3\pmod4$;
\item $S(n,3)$ is odd if and only if $n$ is odd.
\end{enumerate}
\end{prop}

\subsection{Ranks at least 4}
\label{Sec:Tn-ranks-at-least-four}

\begin{prop}\label{p:Tn-ranks-at-least-four}
Let $n\ge4$ and let $4\le k\le n$. The rank-$k$ principal factor $J_k^0$ of
$T_n$ has a complete mapping.
\end{prop}

\begin{proof}
The rank-$k$ principal factor has the Rees form
$J_k^0\cong\mathcal M^0(S_k,I,\Lambda,P)$,
where $Q_k$ is the transversal pattern. Proposition~\ref{p:Tn-pattern-Hall}
shows that $Q_k$ satisfies the equivalent conditions (c) and (d) of
Theorem~\ref{c:Rees0}. Since $k\ge4$, the group $S_k$ has a complete mapping
by the Hall--Paige theorem. Theorem~\ref{c:Rees0} applies.
\end{proof}

\subsection[The rank-3 factor when n is not congruent to 3 modulo 4]{The rank-3 factor when $n\not\equiv 3\pmod 4$}
\label{Sec:Tn-rank-three-nonexceptional}

\begin{prop}\label{p:Tn-rank3-nonexceptional}
Let $n\ge4$. If $n\not\equiv3\pmod4$, then the rank-3 principal factor
$J_3^0$ of $T_n$ has a complete mapping.
\end{prop}

\begin{proof}
The rank-3 principal factor has the Rees form
 $J_3^0\cong\mathcal M^0(S_3,I,\Lambda,P)$,
where $I$ is the set of rank-3 kernels, $\Lambda$ is the set of rank-3
images, and $Q_3$ is the transversal pattern. Proposition~\ref{p:Tn-pattern-Hall}
shows that $Q_3$ satisfies the equivalent conditions (c) and (d) of
Theorem~\ref{c:Rees0}. The group $S_3$ has no complete mapping. Moreover
\[
|I|=S(n,3),
        \quad
        |\Lambda|=\binom n3.
\]
By Proposition~\ref{p:Tn-small-rank-parity}, these two numbers are both odd
exactly when $n\equiv3\pmod4$. Therefore, if $n\not\equiv3\pmod4$, at least
one of $|I|$ and $|\Lambda|$ is even, and Theorem~\ref{t:Rees0-noncomp-even}
applies.
\end{proof}

\subsection[The rank-3 factor]{The rank-3 factor when $n\equiv 3\pmod 4$}
\label{t3:Sec:T3-transformations}

Assume throughout this subsection that
$n\equiv 3\pmod 4
       \ \text{and}\
        n\ge7$.
Let $X_n=\{1,2,\ldots,n\}$ and $T_3(n)=\{t\colon X_n\to X_n\mid\rank(t)=3\}$.
The non-zero part of $J_3^0$ is $T_3(n)$. We will show the existence of a complete mapping in this exceptional case by
verifying the hypotheses of Corollary~\ref{c:Rees0-S3-anchored}, which is the
$S_3$ instance of Theorem~\ref{t:Rees0-general-anchored}.

For a triple $y=(y_1,y_2,y_3)$ and $\pi\in S_3$, put
$y\pi=(y_{1\pi},y_{2\pi},y_{3\pi})$.

\subsubsection{Canonical rank-3 classes}

Let $\mathcal K$ be the set of ordered partitions
 $Q=(Q_1,Q_2,Q_3)$
of $X_n$ into three non-empty blocks such that
 $1\in Q_1$ and 
        $\min(X_n\setminus Q_1)\in Q_2$.
For $Q\in\mathcal K$ and for a triple $y=(y_1,y_2,y_3)$ of pairwise distinct
points of $X_n$, let $t_{Q,y}$ be the transformation sending every point of
$Q_r$ to $y_r$.

Let
\[
        \mathcal Y=\binom{X_n}{3}.
\]
For $I=\{i_1<i_2<i_3\}\in\mathcal Y$, put
\(i(I)=(i_1,i_2,i_3).\)
For $Q\in\mathcal K$ and $I\in\mathcal Y$, define
\(\Omega_{Q,I}:=\{t_{Q,i(I)\pi}\mid\pi\in S_3\}.\)
For $Q=(Q_1,Q_2,Q_3)\in\mathcal K$, write $\Tr(Q)$ for the set of transversals
of $Q$, that is, the set of all $I\in\mathcal Y$ meeting each block $Q_r$ in
exactly one point.

\begin{prop}\label{t3:prop:canonical}
Every element of $T_3(n)$ has a unique expression
\(t=t_{Q,y},\)
where $Q\in\mathcal K$ and $y=(y_1,y_2,y_3)$ has pairwise distinct coordinates.
\end{prop}

\begin{proof}
Let $t\in T_3(n)$. Its kernel is a partition of $X_n$ into three non-empty
classes. There is a unique way to order these classes as $Q=(Q_1,Q_2,Q_3)$ by
our convention for $\mathcal K$: $1\in Q_1$, the least element outside $Q_1$ lies in $Q_2$, and
the remaining class is $Q_3$. If $y_r$ is the common value of $t$ on $Q_r$, then
$y=(y_1,y_2,y_3)$ has pairwise distinct coordinates and $t=t_{Q,y}$. The same
argument gives uniqueness.
\end{proof}

\subsubsection[The Rees data of the rank-3 factor]{The Rees data of the rank-3 factor}
In the following, $S_3$ always acts on triples by permuting coordinates.

\begin{prop}\label{t3:prop:product-rule}
Let $s=t_{B,p}$ and $t=t_{Q,y}$, where $B,Q\in\mathcal K$ and $p,y$ have
pairwise distinct coordinates. Then $st\in T_3(n)$ if and only if the three
points $p_1,p_2,p_3$ lie in three distinct blocks of $Q$. If this holds, there
is a unique $\theta\in S_3$ such that
\(p_r\in Q_{r\theta}\quad(r=1,2,3),\)
in which case
\(st=t_{B,y\theta}.\)
\end{prop}

\begin{proof}
If $x\in B_r$, then $xs=p_r$, so $x(st)=p_rt$. Hence $st$ has rank 3
exactly when $p_1,p_2,p_3$ lie in distinct blocks of $Q$. In that case the
image of $B_r$ is the coordinate of $y$ indexed by the block of $Q$ containing
$p_r$, namely $y_{r\theta}$.
\end{proof}

For $Q\in\mathcal K$ and $I\in\Tr(Q)$, let $\theta(Q,I)\in S_3$ be defined by
\(x_r=i(I)_{r\theta(Q,I)},\)
where $x_r$ is the unique point of $I\cap Q_r$. Put
\(s(Q,I):=\sgn(\theta(Q,I)).\)

\begin{prop}\label{t3:prop:rank-three-Rees-data}
The principal factor $J_3^0$ has a Rees representation
\[
J_3^0\cong \mathcal M^0(S_3,\mathcal K,\mathcal Y,P).
\]
For $I\in\mathcal Y$ and $Q\in\mathcal K$, the sandwich entry $p_{I,Q}$ is
non-zero if and only if $I\in\Tr(Q)$; in that case
\(\sgn(p_{I,Q})=s(Q,I).\)
\end{prop}

\begin{proof}
	The first statement is clear, except for the nature of $P$.
On the individual $\mathcal{H}$-classes, we can use the index scheme
\(\Omega_{B,I}=\{t_{B,i(I)\pi}\mid\pi\in S_3\},\quad \Omega_{Q,J}=\{t_{Q,i(J)\sigma}\mid\sigma\in S_3\}.\)
Proposition~\ref{t3:prop:product-rule} gives a non-zero product precisely when
$I$ is a transversal of $Q$. In that case, if $x_r$ is the unique point in $I\cap Q_r$, then
$x_r=i(I)_{r\theta(Q,I)}$, and
\(t_{B,i(I)\pi}\,t_{Q,i(J)\sigma} =t_{B,i(J)\pi{\theta(Q,I)\inv}\sigma}.\)
Thus the corresponding sandwich entry may be taken to be $\theta(Q,I)\inv$, whose
sign is $s(Q,I)$.
\end{proof}

\subsubsection[Verification of the abstract Rees hypotheses]{Verification of the abstract Rees hypotheses}
Put
\[
\kappa=|\mathcal K|=S(n,3),
        \quad
        \nu=|\mathcal Y|=\binom n3.
\]
By Proposition~\ref{p:Tn-small-rank-parity}, both $\kappa$ and $\nu$ are odd.

\begin{prop}\label{t3:prop:common-kernels}
Let $I,J\in\mathcal Y$, and put $r=|I\cap J|$. Then
\(|\{Q\in\mathcal K\mid I\in\Tr(Q)\text{ and }J\in\Tr(Q)\}|=(3-r)!3^{n-6+r}.\)
In particular, any two image triples admit a common compatible kernel.
\end{prop}

\begin{proof}
Fix $I=\{i_1,i_2,i_3\}$. Colour $i_a$ by colour $a$. Every point of
$I\cap J$ keeps its colour. The remaining $3-r$ points of $J\setminus I$ must
receive the $3-r$ unused colours, which can be done in $(3-r)!$ ways. Every
point outside $I\cup J$ may then be coloured arbitrarily in $3$ ways. Thus we
obtain $(3-r)!3^{n-6+r}$ colourings in which each colour class meets both $I$
and $J$ exactly once.

Such a colouring gives three labelled non-empty blocks. Recanonicalizing those
blocks gives a kernel $Q\in\mathcal K$ for which both $I$ and $J$ are
transversals. Conversely, if $Q\in\mathcal K$ has both $I$ and $J$ as
transversals, then each block of $Q$ contains a unique point of $I$; labelling
that block by the corresponding index recovers a unique colouring counted above.

The last statement follows from the assumption $n\ge7$.
\end{proof}

We will remove the sign obstruction of Corollary~\ref{c:Rees0-S3-anchored} by a three-cycle of compatible transversals.

\begin{prop}\label{t3:prop:positive-three-cycle}
Define
\[
\begin{aligned}
P_0&=(\{1\},\{2,4\},\{3,5,6,\ldots,n\}),\\
P_1&=(\{1\},\{2,3\},\{4,5,6,\ldots,n\}),\\
P_2&=(\{1\},\{2,3,6\},\{4,5,7,8,\ldots,n\}),
\end{aligned}
\]
and \(I_0=\{1,3,4\}\), \(I_1=\{1,2,5\}\) and
\(I_2=\{1,2,7\}\).
Then $I_j\in\Tr(P_i)$ for all $i,j\in\{0,1,2\}$, and
\(s(P_0,I_0)=-1,\quad s(P_i,I_j)=+1\quad\text{for all other pairs }(i,j).\)
Consequently
\[
        (-1)^3
        \left(
        \frac{s(P_1,I_0)}{s(P_0,I_0)}
        \right)
        \left(
        \frac{s(P_2,I_1)}{s(P_1,I_1)}
        \right)
        \left(
        \frac{s(P_0,I_2)}{s(P_2,I_2)}
        \right)=1.
\]
\end{prop}

\begin{proof} That $I_j\in\Tr(P_i)$ for all $i,j\in\{0,1,2\}$ can be readily observed. The remaining assertions follow by direct calculation.
\end{proof}
It remains to construct the matrix satisfying the conditions of Theorem~\ref{t:Rees0-general-anchored}.

\begin{prop}\label{t3:prop:harmonic}
For $Q=(Q_1,Q_2,Q_3)\in\mathcal K$, put
\(d(Q):=|\Tr(Q)|=|Q_1||Q_2||Q_3|.\)
Define
\[
        h_{Q,I}:=
        \begin{cases}
        \nu/d(Q), & I\in\Tr(Q),\\
        0, & I\notin\Tr(Q).
        \end{cases}
\]
Then
\[
\left(\sum_{I\in\mathcal Y}h_{Q,I}\right)=\nu\quad(Q\in\mathcal K),
\]
and
\[
\left(\sum_{Q\in\mathcal K}h_{Q,I}\right)=\kappa\quad(I\in\mathcal Y).
\]
Moreover, if $I\in\Tr(Q)$, then $h_{Q,I}\ge1$.
\end{prop}

\begin{proof}
The row sums are immediate from $|\Tr(Q)|=d(Q)$. For a fixed $I\in\mathcal Y$, put
\[
        c(I):=\left(\sum_{Q:I\in\Tr(Q)}\frac1{d(Q)}\right).
\]
The natural action of $\Sym(X_n)$ is transitive on $\mathcal Y$, and
recanonicalization preserves the multiset of block sizes of a kernel, hence
preserves $d(Q)$. Therefore $c(I)$ is independent of $I$; write $c(I)=c$.
Summing over all $I$ gives
\[
        \nu c=
        \left(\sum_{Q\in\mathcal K}\left(\sum_{I\in\Tr(Q)}\frac1{d(Q)}\right)\right)
        =\left(\sum_{Q\in\mathcal K}1\right)=\kappa.
\]
Thus $c=\kappa/\nu$, and multiplying by $\nu$ gives the column sum. Finally,
$d(Q)=|\Tr(Q)|\le\nu$, so $h_{Q,I}=\nu/d(Q)\ge1$ whenever $I\in\Tr(Q)$.
\end{proof}

\begin{definition}\label{t3:def:anchored-selector}
Fix the triples $I_0,I_1,I_2$ and kernels $P_0,P_1,P_2$ from
Proposition~\ref{t3:prop:positive-three-cycle}. Since $\nu$ is odd,
$\mathcal Y\setminus\{I_0,I_1,I_2\}$ has even cardinality. Pair its elements
arbitrarily. For each pair $\{I,I'\}$, choose a common compatible kernel
$P(I,I')$ with $I,I'\in\Tr(P(I,I'))$, which exists by
Proposition~\ref{t3:prop:common-kernels}. The \emph{anchored selector} is the map
$\sigma\colon \mathcal Y\to\mathcal K$ defined by
$I_a\sigma=P_a$, with $a=0,1,2$,
and
 $I\sigma=I'\sigma=P(I,I')$
for each paired two-element set $\{I,I'\}$.
\end{definition}

\begin{prop}\label{t3:prop:abstract-hypotheses}
The Rees data of $J_3^0$ satisfy the hypotheses of
Theorem~\ref{t:Rees0-general-anchored}, with the modifications from Corollary~\ref{c:Rees0-S3-anchored}.
\end{prop}

\begin{proof}
The support relation is $I\in\Tr(Q)$. We take $N=A_3$, so that
$S_3/N\cong\{\pm1\}$, and the quotient entry $\bar p_{I,Q}$ in
Theorem~\ref{t:Rees0-general-anchored} is $s(Q,I)$ by
Proposition~\ref{t3:prop:rank-three-Rees-data}. The matrix $h$ required in
Theorem~\ref{t:Rees0-general-anchored} is the matrix of
Proposition~\ref{t3:prop:harmonic}.

Choose $B_\ast\in\mathcal K$. Since $\kappa$ is odd, there is a
fixed-point-free involution on $\mathcal K\setminus\{B_\ast\}$. Let $\sigma$
be the anchored selector of Definition~\ref{t3:def:anchored-selector}. Define a
permutation $\rho$ of $\mathcal Y$ by
\(I_0\rho=I_1, \quad I_1\rho=I_2, \quad I_2\rho=I_0,\)
and, for every paired two-element set $\{I,I'\}$ in
Definition~\ref{t3:def:anchored-selector}, by
\(I\rho=I', \quad I'\rho=I.\)

The condition $I\in\Tr(I\sigma)$ follows from the definition of $\sigma$. The
condition $I\in\Tr((I\rho)\sigma)$ holds for the two-cycles because paired
triples have the same selected kernel, and it holds for the distinguished
three-cycle because Proposition~\ref{t3:prop:positive-three-cycle} gives
$I_j\in\Tr(P_i)$ for all $i,j\in\{0,1,2\}$. Proposition~\ref{t3:prop:harmonic}
gives $h_{I\sigma,I}\ge1$.

For a two-cycle $I\leftrightarrow I'$ of $\rho$ outside
$\{I_0,I_1,I_2\}$, one has $I\sigma=I'\sigma$. Condition~\eqref{eq:S3-anchored-cycle-condition}
is then immediate. For the remaining cycle
$I_0\to I_1\to I_2\to I_0$, condition~\eqref{eq:S3-anchored-cycle-condition}
is exactly the last assertion of Proposition~\ref{t3:prop:positive-three-cycle}.
Thus all hypotheses of Theorem~\ref{t:Rees0-general-anchored} and Corollary~\ref{c:Rees0-S3-anchored} hold.
\end{proof}

\begin{cor}\label{t3:thm:right-half}
If $n\equiv3\pmod4$ and $n\ge7$, then the rank-3 principal factor $J_3^0$
of $T_n$ has a complete mapping.
\end{cor}

\begin{proof}
By Proposition~\ref{t3:prop:rank-three-Rees-data}, the rank-3 principal
factor has the Rees form required in Theorem~\ref{t:Rees0-general-anchored}.
By Proposition~\ref{t3:prop:abstract-hypotheses}, the hypotheses of that theorem
are satisfied, with the modifications from Corollary~\ref{c:Rees0-S3-anchored}. Hence $J_3^0$ has a complete mapping.
\end{proof}

Combining this with the preceding results, the rank-3 principal factor is covered for all $n\ge4$.
\begin{theorem}\label{t3:thm:main}
For every $n\ge4$, the rank-3 principal factor $J_3^0$ of $T_n$ has a
complete mapping.
\end{theorem}

\begin{proof}
If $n\not\equiv3\pmod4$, the result is Proposition~\ref{p:Tn-rank3-nonexceptional}.
It remains to consider $n\equiv3\pmod4$. Then $n\ge7$, and
Corollary~\ref{t3:thm:right-half} gives a complete mapping of $J_3^0$.
\end{proof}

\subsection[The rank-2 factor from Rees 0-matrix theory]{The rank-2 factor from Rees $0$-matrix theory}
\label{Sec:Tn-rank-two-Rees-observation}

\begin{prop}\label{p:Tn-rank-two-indirect}
For every $n\ge4$, the rank-2 principal factor $J_2^0$ of $T_n$ has a
complete mapping by the Rees $0$-matrix criteria of
Section~\ref{Sec:Rees0-matrix-2}.
\end{prop}

\begin{proof}
The rank-2 principal factor has the Rees form
\(J_2^0\cong \mathcal M^0(S_2,\mathcal K,\mathcal Y,P),\)
where $\mathcal K$ is the set of canonical ordered bipartitions
$Q=(Q_1,Q_2)$ of $X_n$, with $1\in Q_1$, and
$\mathcal Y=\binom{X_n}{2}$. A pair $Y\in\mathcal Y$ is compatible with
$Q\in\mathcal K$ precisely when $Y$ is a transversal of $Q$.

If $n\equiv0,1\pmod4$, then
\[
        |\mathcal Y|=\binom n2
\]
is even, and Theorem~\ref{t:Rees0-noncomp-even} applies.

Assume now that $n\equiv2,3\pmod4$. Then both $|\mathcal K|=S(n,2)$ and
$|\mathcal Y|=\binom n2$ are odd, by
Proposition~\ref{p:Tn-small-rank-parity}. We verify the hypotheses of
Theorem~\ref{t:Rees0-general-anchored}, taking $N=1$. Identify $S_2$ with $\{\pm1\}$ by the
sign map. If $Y=\{y_1<y_2\}$ and $Q=(Q_1,Q_2)$ separates $Y$, let $s(Q,Y)$ be
$+1$ when $y_1\in Q_1$ and $y_2\in Q_2$, and $-1$ otherwise. This is the image
in the cyclic $2$-group quotient of the corresponding sandwich entry.

For $Q=(Q_1,Q_2)$ put
\(d(Q)=|Q_1||Q_2|.\)
Define
\[
        h_{Q,Y}:=
        \begin{cases}
        |\mathcal Y|/d(Q), & Y\text{ is a transversal of }Q,\\
        0, & \text{otherwise}.
        \end{cases}
\]
Then
\[
\left(\sum_{Y\in\mathcal Y}h_{Q,Y}\right)=|\mathcal Y|
        \quad(Q\in\mathcal K).
\]
For fixed $Y$, the quantity
\[
        \left(\sum_{Q:Y\text{ transversal of }Q}\frac1{d(Q)}\right)
\]
is independent of $Y$, because permutations of $X_n$, followed by
recanonicalizing the two blocks, preserve $d(Q)$ and the transversal relation.
Summing over all $Y$ gives
\[
\left(\sum_{Q\in\mathcal K}h_{Q,Y}\right)=|\mathcal K|
        \quad(Y\in\mathcal Y).
\]
Moreover $h_{Q,Y}\ge1$ whenever $Y$ is a transversal of $Q$.

Choose
\(Y_0=\{2,3\}, \quad Y_1=\{1,2\}, \quad Y_2=\{1,3\},\)
and
\[
P_0=(\{1,2\},X_n\setminus\{1,2\}),
        \quad
        P_1=(\{1,3\},X_n\setminus\{1,3\}),
        \quad
        P_2=(\{1\},X_n\setminus\{1\}).
\]
The required compatibilities are
\[
Y_0\in\Tr(P_0)\cap\Tr(P_1),
        \quad
        Y_1\in\Tr(P_1)\cap\Tr(P_2),
        \quad
        Y_2\in\Tr(P_2)\cap\Tr(P_0).
\]
The corresponding signs are
\(s(P_0,Y_0)=+1, \quad s(P_1,Y_0)=-1,\)
\(s(P_1,Y_1)=s(P_2,Y_1)=+1, \quad s(P_2,Y_2)=s(P_0,Y_2)=+1.\)
Hence
\[
        (-1)^3
        \left(\frac{s(P_1,Y_0)}{s(P_0,Y_0)}\right)
        \left(\frac{s(P_2,Y_1)}{s(P_1,Y_1)}\right)
        \left(\frac{s(P_0,Y_2)}{s(P_2,Y_2)}\right)=1.
\]

Since $|\mathcal Y|$ is odd, the set
$\mathcal Y\setminus\{Y_0,Y_1,Y_2\}$ has even cardinality. Pair its elements
arbitrarily. If $Y$ and $Y'$ are paired, choose a canonical bipartition
$Q(Y,Y')$ separating both pairs. Such a bipartition exists because a graph
with two edges is bipartite. Define $\sigma\colon \mathcal Y\to\mathcal K$ by
\(Y_i\sigma=P_i\quad(i=0,1,2),\)
and by
\(Y\sigma=Y'\sigma=Q(Y,Y')\)
for each paired set $\{Y,Y'\}$. Define $\rho$ by the cycle
\(Y_0\rho=Y_1, \quad Y_1\rho=Y_2, \quad Y_2\rho=Y_0,\)
and by swapping the two elements in every remaining pair. The two-cycles have
trivial sign contribution, because both elements in such a pair have the same
selected kernel. The distinguished three-cycle satisfies the displayed sign
condition above. Thus Theorem~\ref{t:Rees0-general-anchored} applies.
\end{proof}

We can now state and prove the main theorem of this section.

\begin{theorem}\label{Thm:full-transformation-semigroup}
The full transformation monoid $T_n$ has a complete mapping if and only if
$n=1$ or $n\ge4$.
\end{theorem}

\begin{proof}
For $n=1$, the monoid $T_n$ is the trivial group and has a complete mapping.
For $n=2$ and $n=3$, the top principal factor is $S_n$ with zero adjoined. If
$S_n\cup\{0\}$ had a complete mapping, then that mapping would fix zero by
Proposition~\ref{Prp:zero}, and its restriction to $S_n$ would be a complete
mapping of the group $S_n$. But $S_2$ and $S_3$ have cyclic Sylow
$2$-subgroups, and therefore do not have a complete mapping by the Hall--Paige theorem. Hence $T_2$ and $T_3$ have no complete mappings by
Theorem~\ref{t:J-reduct}.

Assume now that $n\ge4$. The rank-1 principal factor is a band, so the
identity mapping is a complete mapping. The principal factors of ranks $k\ge4$
have complete mappings by Proposition~\ref{p:Tn-ranks-at-least-four}. The
rank-3 principal factor has a complete mapping by Theorem~\ref{t3:thm:main}.
The rank-2 principal factor has a complete mapping by Proposition~\ref{p:Tn-rank-two-indirect}.
Therefore every principal factor of $T_n$ has a complete mapping, and
Theorem~\ref{t:J-reduct} gives a complete mapping of $T_n$.
\end{proof}

\section{Inverse semigroups revisited}
\label{Sec:inv-revisited}

In this section we classify finite inverse semigroups with complete mappings.

\begin{theorem}\label{t:inverse}
Let $I$ be a finite inverse semigroup. Then $I$ has a complete mapping if and only if, for each $\mathcal{J}$-class $J$ of $I$, one of the following holds:
\begin{enumerate}
\item The maximal subgroups of $J$ have a complete mapping;
\item The number of $\mathcal L$-classes of $J$, which equals the number of $\mathcal R$-classes of $J$, is even.
\end{enumerate}
\end{theorem}
\begin{proof}
By Theorem~\ref{t:J-reduct}, $I$ has a complete mapping if and only if every
principal factor $J^0$ has a complete mapping. Let $J$ be a $\mathcal J$-class
and let
\(S=\mathcal M^0(G,A,A,P)\)
be the Rees $0$-matrix semigroup isomorphic to $J^0$. Since $I$ is an inverse
semigroup, the number of $\mathcal L$-classes and the number of $\mathcal R$-classes
in $J$ are equal, and, after applying the usual normalizations and permuting the
indices, we may assume that $P$ is the identity matrix.

If $G$ has a complete mapping, then $S$ has a complete mapping by
Theorem~\ref{c:Rees0}, since the identity pattern satisfies conditions (c) and
(d) of that theorem. If $G$ does not have a complete mapping and $|A|$ is even,
then $S$ has a complete mapping by Corollary~\ref{c:Rees0-even-transversal}.
This proves one implication.

Conversely, suppose that some principal factor $J^0$ has maximal subgroup $G$
with no complete mapping and has an odd number of $\mathcal L$-classes. Then
$J^0$ is isomorphic to $\mathcal M^0(G,A,A,P)$ with $|A|$ odd and $P$ the
identity matrix. By the Hall--Paige theorem, $G$ has non-trivial cyclic Sylow
$2$-subgroups. The matrix $P$ is obtained from a normalized matrix all of
whose entries are the identity, by replacing the off-diagonal entries by $0$.
Thus Theorem~\ref{t:converse0} implies that $J^0$ has no complete mapping.
By Theorem~\ref{t:J-reduct}, neither does $I$.
\end{proof}

As an application, we obtain the following result about symmetric inverse semigroups.

\begin{cor}
Let $\mathcal I_n$ be the symmetric inverse semigroup on $n$ points. Then $\mathcal I_n$ has a complete mapping if and only if $n\equiv0,1\pmod 4$.
\end{cor}
\begin{proof}
The $\mathcal J$-classes of $\mathcal I_n$ are indexed by $0,1,\ldots,n$. The class $J_i$ has maximal subgroup $S_i$ and $\binom ni$ $\mathcal L$-classes. Hence the first condition of Theorem~\ref{t:inverse} holds unless $i=2,3$, while the second holds for both $J_2$ and $J_3$ exactly when $n\equiv0,1\pmod 4$.
\end{proof}

\section{Complete regularity and involutive complete mappings}
\label{Sec:CR_involutions}

A \emph{completely regular} semigroup is a union of groups. For $x$ in a completely regular semigroup,
let $x\inv$ denote the inverse of $x$ in the maximal subgroup containing $x$. Then $x\inv$ is an inverse
in the semigroup sense: $xx\inv x=x$, $x\inv xx\inv =x\inv$, and it also commutes with $x$:
$xx\inv=x\inv x$. In addition, $x\inv$ is the unique inverse of $x$ commuting with $x$, hence is
called the commuting inverse of $x$.

An element $a'\in S$ is a weak inverse of $a$ if $a'=a'aa'$. The set of weak inverses of $a$ is denoted by $W(a)$.

If $S$ is a regular semigroup with inverse mapping $x\mapsto x'\in V(x)$, then for each $x\in S$, set
\(x\inv = x(x^3)'x.\)
Note that $x\inv$ is a weak inverse of $x$:
\begin{equation}\label{Eqn:CR-out}
x\inv xx\inv = x(x^3)'x^3(x^3)'x = x(x^3)'x = x\inv.
\end{equation}
{By Hall's Lemma~4.1.1 and Theorem~4.1.2 \cite[Lemma~4.1.1 and Theorem~4.1.2]{Hall89}, $S$ is completely regular if and only if $xx\inv x=x$. In this case Lemma~4.1.1 identifies $x\inv=x(x^3)'x$ with the inverse of $x$ in its maximal subgroup, and hence $x\inv$ is the commuting inverse of $x$.}

We will need the following slightly technical lemma.

\begin{lemma}\label{Lem:x'xx}
Let $S$ be a semigroup, and suppose that $x\mapsto x'\in W(x)$ is a weak inverse mapping of $S$ satisfying
\begin{equation}\label{Eqn:x'xx}
x'xx = x
\end{equation}
for all $x\in S$. Then the following hold.
\begin{enumerate}
  \item For all $x\in S$, $x'\in V(x)$;
  \item $S$ is completely regular with commuting inverse $x\inv = x(x^3)'x$.
\end{enumerate}
\end{lemma}
\begin{proof}
For (1), we compute
\begin{alignat*}{3}
\underbrace{x}x'x
    &\byeqn{Eqn:x'xx} \underbrace{x'}xxx'x &\byeqn{Eqn:x'xx} x''x'\underbrace{x'xx}x'x &\byeqn{Eqn:x'xx} x''\underbrace{x'xx'}x \\
    &= x''x'\underbrace{x} &\byeqn{Eqn:x'xx} \underbrace{x''x'x'}xx &\byeqn{Eqn:x'xx} x'xx \\
    &\byeqn{Eqn:x'xx} x. & &
\end{alignat*}

For (2), we have already noted that $x\inv$ is an outer inverse of $x$; see \eqref{Eqn:CR-out}.
Next, we verify the identity:
\begin{equation}\label{Eqn:CR-4}
xx'x' = x'
\end{equation}
as follows:
\begin{alignat*}{3}
\underbrace{x}x'x'
    &\byeqn{Eqn:x'xx} \underbrace{x'}xxx'x' &&\byeqn{Eqn:x'xx} x''x'\underbrace{x'xx}x'x' &\byeqn{Eqn:x'xx} x''\underbrace{x'xx'}x' \\
    &= x''x'x' &&\byeqn{Eqn:x'xx} x'. &
\end{alignat*}

Next we show
\begin{equation}\label{Eqn:CR-5}
x^i (x')^i = xx'
\end{equation}
for all $i\ge 2$. First, $x^2(x')^2 = x\underbrace{xx'x'} \byeqn{Eqn:CR-4} xx'$. Next, assuming
\eqref{Eqn:CR-5} for $i$, we have $x^{i+1} (x')^{i+1} = xx^i (x')^i x' = xxx'x' = xx'$, using the
induction hypothesis in the second equality and the case $i=2$ in the third.

Next we prove
\begin{equation}\label{Eqn:CR-6}
x\inv x = x'x
\end{equation}
for all $x\in S$ as follows:
\begin{alignat*}{3}
x\inv x
    &= x\inv\underbrace{xx'}x && \byeqn{Eqn:CR-5} x\inv x^5 (x')^5 x &&= x\underbrace{(x^3)'x^6}(x')^5 x \\
    &\byeqn{Eqn:x'xx} \underbrace{xx^3(x')^4} x'x && \byeqn{Eqn:CR-5} \underbrace{xx'x'}x &&\byeqn{Eqn:CR-4} x'x.
\end{alignat*}

From \eqref{Eqn:CR-6}, we conclude that $x\inv$ is an inverse:
\begin{equation}\label{Eqn:CR-inn}
xx\inv x = xx'x = x.
\end{equation}
{By Hall's Theorem~4.1.2  \cite{Hall89}, $S$ is completely regular, and Lemma~4.1.1 identifies $x\inv$ with the inverse of $x$ in its maximal subgroup \cite[Lemma~4.1.1 and Theorem~4.1.2]{Hall89}. Thus $x\inv$ is the commuting inverse of $x$.}
This proves (2) and completes the proof of the lemma.
\end{proof}

A mapping $\alpha\colon S\to S$ of a semigroup $S$ is \emph{involutive} if $\alpha^2 = \idmap_S$.

\begin{theorem}\label{Thm:involution}
A finite semigroup admitting an involutive complete mapping is completely regular.
\end{theorem}
\begin{proof}
Let $S$ be a finite semigroup with an involutive complete mapping $\alpha\colon S\to S$ and orthomorphism $\theta\colon S\to S,\quad x\mapsto x\cdot x\alpha$.
By Theorem~\ref{Thm:regular}, the semigroup $S$ is regular, and by Theorem~\ref{Thm:good_inverse} we may choose an inverse mapping $x\mapsto x'\in V(x)$ satisfying $x'\cdot x\theta=x\alpha$ for every $x\in S$. We shall prove that \eqref{Eqn:x'xx} holds with $x\inv$ in place of $x'$. The result will then follow from Lemma~\ref{Lem:x'xx}.

First, we prove, for all $x\in S$, the identities
\begin{align}
(x\theta\inv\alpha)'\cdot x\theta\inv\alpha\cdot x &= x, \label{Eqn:inv1} \\
x\cdot (x\theta\inv\alpha)'\cdot x\theta\inv\alpha &= x. \label{Eqn:inv2}
\end{align}

To prove \eqref{Eqn:inv1}, observe that, for every $x\in S$,
\begin{alignat*}{2}
(x\alpha)'\cdot x\alpha\cdot \underbrace{x\theta}
    &= (x\alpha)'\cdot x\alpha\cdot \underbrace{x}\cdot x\alpha &&= (x\alpha)'\cdot \underbrace{x\alpha\cdot x\alpha^2}\cdot x\alpha \\
    &= \underbrace{(x\alpha)'\cdot x\alpha\theta}\cdot x\alpha &&= \underbrace{x\alpha^2} \cdot x\alpha \\
    &= x\cdot x\alpha &&= x\theta.
\end{alignat*}
Replacing $x$ with $x\theta\inv$, we obtain \eqref{Eqn:inv1}.

To prove \eqref{Eqn:inv2}, observe that, for every $x\in S$,
\[
\underbrace{x\theta}\cdot (x\alpha)'\cdot x\alpha = x\cdot \underbrace{x\alpha\cdot (x\alpha)'\cdot x\alpha} = x\cdot x\alpha = x\theta.
\]
Replacing $x$ with $x\theta\inv$, we obtain \eqref{Eqn:inv2}.

Similarly to \eqref{Eqn:inv2}, we prove
\begin{equation}\label{Eqn:inv3}
x\theta\inv\alpha\cdot y(xy)'xy = x\theta\inv\alpha\cdot y
\end{equation}
for all $x,y\in S$ as follows.
\begin{align*}
(\underbrace{x\alpha}\cdot y)(x\theta\cdot y)'(x\theta\cdot y) &= x'\underbrace{(x\theta\cdot y)(x\theta\cdot y)'(x\theta\cdot y)} \\
&= x'(x\theta\cdot y) \\
&= x\alpha\cdot y.
\end{align*}
Replacing $x$ with $x\theta\inv$, we have \eqref{Eqn:inv3}.

Next we prove the identity
\begin{equation}\label{Eqn:inv4}
x(x^2)'x^2 = x.
\end{equation}
We compute
\begin{align*}
\underbrace{x}(x^2)'x^2 &\byeqn{Eqn:inv1} (x\theta\inv\alpha)'\cdot \underbrace{x\theta\inv\alpha\cdot x(x^2)'x^2} \\
&\byeqn{Eqn:inv3} (x\theta\inv\alpha)'\cdot x\theta\inv\alpha\cdot x \\
&\byeqn{Eqn:inv1} x,
\end{align*}
as claimed.

Next we prove, for all $x\in S$,
\begin{equation}\label{Eqn:inv5}
x\cdot ((x^2)\theta\inv\alpha)'((x^2)\theta\inv\alpha) = x.
\end{equation}
We have
\begin{align*}
\underbrace{x}\cdot ((x^2)\theta\inv\alpha)'((x^2)\theta\inv\alpha)
    &\byeqn{Eqn:inv4} x(x^2)' \underbrace{x^2((x^2)\theta\inv\alpha)'((x^2)\theta\inv\alpha)} \\
    &\byeqn{Eqn:inv2} x(x^2)'x^2 \\
    &\byeqn{Eqn:inv4} x.
\end{align*}

For the penultimate step, we prove
\begin{equation}\label{Eqn:inv6}
x x\inv x^2 = x^2
\end{equation}
as follows:
\begin{align*}
x x\inv x^2 &= \underbrace{x} x(x^3)' x^3 \\
&\byeqn{Eqn:inv5} x\cdot ((x^2)\theta\inv\alpha)'\underbrace{((x^2)\theta\inv\alpha)\cdot x(x^2x)'x^2x} \\
&\byeqn{Eqn:inv3} \underbrace{x\cdot ((x^2)\theta\inv\alpha)'(x^2)\theta\inv{\alpha}}\cdot x \\
&\byeqn{Eqn:inv5} x^2.
\end{align*}

Finally, we prove \eqref{Eqn:x'xx}:
\begin{align*}
x\inv xx &= \underbrace{x}(x^3)'x^3 \\
&\byeqn{Eqn:inv4} x(x^2)'x\underbrace{x(x^3)'x}x^2 \\
&= x(x^2)'\underbrace{xx\inv x^2} \\
&\byeqn{Eqn:inv6} x(x^2)'x^2 \\
&\byeqn{Eqn:inv4} x.
\end{align*}
Since we already know $x\inv$ is a weak inverse of $x$ by \eqref{Eqn:CR-out}, it follows from Lemma~\ref{Lem:x'xx} that $S$ is completely regular.
\end{proof}

\begin{cor}
Let $S$ be a finite semigroup in which the identity mapping is a complete mapping. Then $S$ is completely regular.
\end{cor}

\begin{proof}
The identity mapping is involutive, so the assertion follows immediately from
Theorem~\ref{Thm:involution}.
\end{proof}

\section{Bands}
\label{Sec:bands}

In this section, we examine complete mappings on several particular classes of semigroups. A \emph{band} is a semigroup in which every element is an idempotent, and a commutative band is a \emph{semilattice}. \emph{Left-zero semigroups}, defined by $xy=x$, and \emph{right-zero semigroups}, defined by $xy=y$, are bands. A band is \emph{rectangular} if it satisfies $xyx=x$. Equivalently, a semigroup is a rectangular band if and only if it satisfies the quasi-identity $xy=yx\Rightarrow x=y$. Every rectangular band is isomorphic to a direct product $L\times R$, where $L$ is a left-zero band and $R$ is a right-zero band~\cite[Theorem~1.1.3]{Howie}.

We now give a couple of characterizations of (finite) bands as semigroups with special complete mappings or special orthomorphisms.

\begin{prop}\label{Prp:id_ortho_band}
Let $S$ be a finite semigroup and let $\alpha$ be a complete
mapping satisfying $x\cdot x\alpha=x$ for all $x\in S$. Then $S$ is a
band.
\end{prop}
\begin{proof}
Since $S$ is finite, $\alpha$ has finite order, say, $k$. We claim that for all $i\ge 1$, $x\cdot x\alpha^i=x$. This holds for $i=1$ by assumption. Assuming the claim for some $i\ge 1$, we have
\(x\cdot x\alpha^{i+1} = x\cdot x\alpha^i\cdot x\alpha^{i+1} = x\cdot x\alpha^i = x,\)
which establishes the claim. Taking $i=k$, we have the desired result.
\end{proof}

The band case can be stated without reference to principal factors.

\begin{prop}\label{Prp:weakinv_band}
For a semigroup $S$, the following are equivalent:
\begin{enumerate}
    \item\label{band1} $S$ is a band;
    \item\label{band2} $S$ admits a complete mapping $\alpha$ satisfying $x\alpha\in V(x)$ for all $x\in S$;
    \item\label{band3} $S$ admits a complete mapping $\alpha$ satisfying $x\alpha\in W(x)$ for all $x\in S$.
\end{enumerate}
\end{prop}
\begin{proof}
    (\ref{band1}) $\Rightarrow$ (\ref{band2}). The identity permutation is a complete mapping satisfying the desired condition.

    \noindent (\ref{band2}) $\Rightarrow$ (\ref{band3}). This is trivial.

    \noindent (\ref{band3}) $\Rightarrow$ (\ref{band1}). Let $\theta\colon S\to S,\quad x\mapsto x\cdot x\alpha$ be the orthomorphism corresponding to $\alpha$. For $a\in S$, $a\alpha\in W(a)$
    implies $a\theta = a\cdot a\alpha = a\cdot a\alpha\cdot a\cdot a\alpha = a\theta\cdot a\theta$, that is, $a\theta$ is an idempotent. Since $\theta$ is a permutation, $S = \{a\theta\mid a\in S\}\subseteq E(S)$, that is, $S$ is a band.
\end{proof}

Next we consider the extreme case of semigroups in which every permutation of the underlying set is a complete mapping. These are easy to characterize.

\begin{prop}\label{Prp:every_perm}
Every permutation of a semigroup $S$ is a complete mapping if and only if $S$ is a right-zero or left-zero semigroup.
\end{prop}
\begin{proof}
If $S$ is a left-zero semigroup, then every permutation $\alpha$ of $S$ is a complete mapping with the identity mapping as orthomorphism: $x\cdot x\alpha = x$. A dual argument using Proposition~\ref{Prp:equivs} applies if $S$ is a right-zero semigroup.

Conversely, let $S$ be a semigroup in which every permutation is a complete mapping. If $|S| = 1$, there is nothing to prove, so assume $|S| > 1$. For distinct $a,b\in S$, let $\sigma$ denote the $2$-cycle exchanging $a$ and $b$. Then $\sigma$ is a complete mapping, so $a b = a\cdot a\sigma \ne b\cdot b\sigma = b a$. Thus $S$ satisfies the quasi-identity $xy = yx \Rightarrow x=y$. Thus $S$ is a rectangular band, and hence is isomorphic to a direct product $L\times R$ where $L$ is a left-zero band and $R$ is a right-zero band. We claim that either $|L| = 1$ or $|R| = 1$. By way of contradiction, assume $a,b\in L$ and $c,d\in R$ with $a\ne b$ and $c\ne d$. Let $\alpha$ denote the $2$-cycle on $L\times R$ exchanging the pairs $(a,c)$ and $(b,d)$. Then $(a,c)\cdot (a,c)\alpha = (a,c)(b,d) = (a,d) = (a,d)(a,d) = (a,d)\cdot (a,d)\alpha$. Therefore $\alpha$ is not a complete mapping, a contradiction. If $|L|=1$, then $S$ is a right-zero semigroup, while if $|R|=1$, then $S$ is a left-zero semigroup. This completes the proof.
\end{proof}

At the other extreme are semigroups admitting precisely one complete mapping. Bands with this property have a simple characterization.

\begin{lemma}\label{Lem:cfixed}
    Let $S$ be a finite band and let $\alpha\colon S\to S$ be a complete mapping with orthomorphism $\theta\colon S\to S,\quad x\mapsto x\cdot x\alpha$. If $c\in S$ satisfies $cx = xc$ for all $x\in S$, then $c\alpha = c$.
\end{lemma}
\begin{proof}
    Since $S$ is finite, $\alpha$ and $\theta$ have finite orders, say, $k$ and $m$, respectively. For all $x\in S$,
    \begin{equation}\label{Eqn:cfixed-1}
        x\cdot x\theta = x\cdot x\cdot x\alpha = x\cdot x\alpha = x\theta.
    \end{equation}
    Replacing $x$ with $x\theta\inv$, we have
    \begin{equation}\label{Eqn:cfixed-2}
        x\theta\inv\cdot x = x.
    \end{equation}
    We claim that $x\theta^{-j}\cdot x = x$ for all $j\ge 1$, with \eqref{Eqn:cfixed-2} being the case $j=1$. Assuming the claim to be true for some $j\ge 1$, we have
    \[
    x\theta^{-(j+1)}\cdot x =
    \underbrace{x\theta^{-(j+1)}\cdot x\theta^{-j}}\cdot x \byeqn{Eqn:cfixed-1}
    x\theta^{-j}\cdot x = x,
    \]
    where the last equality follows by the induction hypothesis.

    If $m=1$, then \eqref{Eqn:cfixed-3} follows since $\theta=\idmap_S$ and $S$ is a band. If $m>1$, taking $j = m-1$, we get
    \begin{equation}\label{Eqn:cfixed-3}
        x\theta\cdot x = x
    \end{equation}
    for all $x\in S$. This implies
    \begin{equation}\label{Eqn:cfixed-4}
        c\cdot c\alpha = c\theta \byeqn{Eqn:cfixed-1}
        c\cdot c\theta = c\theta\cdot c \byeqn{Eqn:cfixed-3} c.
    \end{equation}

    Now we claim that $c\cdot c\alpha^i = c$ for all $i\ge 1$, with
    \eqref{Eqn:cfixed-4} as the case $i=1$. Assuming the claim holds for some $i\ge 1$, we have
    \begin{alignat*}{3}
        c &= c\cdot c\alpha^i &&\byeqn{Eqn:cfixed-3} c\cdot c\alpha^i\theta\cdot c\alpha^i &&= c\alpha^i\theta\cdot c\cdot c\alpha^i \\
        &= c\alpha^i\theta\cdot c &&= c\cdot c\alpha^i\theta &&= c\cdot c\alpha^i\cdot c\alpha^{i+1}\\
        &= c\cdot c\alpha^{i+1}. && &&
    \end{alignat*}
    This establishes the claim. If $k=1$, then $\alpha=\idmap_S$ and the result is immediate. We may therefore assume that $k>1$. Taking $i=k-1$, we get
    \(c=c\cdot c\alpha^{k-1}=c\alpha^{k-1}\cdot c=c\alpha^{k-1}\cdot c\alpha^{k-1}\alpha=c\alpha^{k-1}\theta.\)
    Apply $\theta^{m-1}$ to both sides and use $c = c\theta$ to conclude
    $c = c\alpha^{k-1}$. Applying $\alpha$ to both sides, we get $c\alpha = c$, which completes the proof.
\end{proof}

The semilattice case is rigid.

%

\begin{prop}\label{p:unique}
    A finite band has exactly one complete mapping if and only if it is a
    semilattice.
\end{prop}

\begin{proof}
    Since $S$ is a band, 
    the identity is a complete mapping.

    We start by proving that if $S$ is not commutative, then it contains a two-element left-zero or right-zero
    subsemigroup.

    Consider the following two assertions:
    \[
        uv=u\ \text{and}\ vu=v
        \ \Rightarrow\ 
        u=v,
        \tag{30}
    \]
    and
    \[
        uv=v\ \text{and}\ vu=u
        \ \Rightarrow\ 
        u=v.
        \tag{31}
    \]
    Suppose, for a contradiction, that both assertions hold for every
    $u,v\in S$. We claim  that $S$ is commutative.

    Let $x,y\in S$. Apply (31) with
    \[
        u=xy
        \ \text{and}\ 
        v=xyx.
    \]
    Since $S$ is a band, both $xy$ and $xyx$ are
    idempotent. Moreover,
    \[
        uv=(xy)(xyx)=xyxyx=(xy)^2x=xyx=v,
    \]
    and
    \[
        vu=(xyx)(xy)=xyxxy=xyxy=(xy)^2=xy=u.
    \]
    Hence (31) implies that
    \[
        xy=xyx.
        \tag{32}
    \]

    Next apply (30) with
    \[
        u=yx
        \ \text{and}\ 
        v=xyx.
    \]
    We have
    \[
        uv=(yx)(xyx)=yxxyx=yxyx=(yx)^2=yx=u.
    \]
    Also,
    \[
        vu=(xyx)(yx)=xyxyx.
    \]
    Since
    \[
        (xyx)^2=xyxxyx=xyxyx
    \]
    and $xyx$ is idempotent, it follows that
    \[
        vu=xyxyx=(xyx)^2=xyx=v.
    \]
    Therefore (30) implies that
    \[
        yx=xyx.
        \tag{33}
    \]
    Combining (32) and (33), we obtain
    \[
        xy=xyx=yx.
    \]
    Since $x$ and $y$ were arbitrary, $S$ is commutative.

    It follows that if $S$ is non-commutative, then at least one of (30)
    and (31) fails. If (30) fails, then there exist distinct elements
    $a,b\in S$ such that
   $ab=a$
        and $ba=b$. 
    Together with $a^2=a$ and $b^2=b$, these equalities show that
   $uv=u$, 
    for all $u,v\in\{a,b\}$. Thus $\{a,b\}$ is a two-element left-zero
    semigroup.

    If (31) fails, it follows by symmetry that  there exist distinct elements
$a,b\in S$ such that $\{a,b\}$ is a two-element right-zero semigroup.

    In either case, the transposition $(a\ b)$ of $S$  is a non-identity complete mapping.
It is proved that if $S$ has exactly one complete mapping, then $S$ is
commutative, and hence it is a semilattice.

    Conversely, suppose that $S$ is a finite semilattice, and let $\phi$
    be a complete mapping of $S$. We prove that $\phi$ is the identity.

    Define the natural partial order on $S$ by
    \[
        x\leq y
        \ \Leftrightarrow\ 
        xy=x.
    \]

    Let
    \[
        \widehat{\phi}\colon S\rightarrow S,
        \ 
        x\widehat{\phi}=x(x\phi).
    \]
    Since $\phi$ is a complete mapping, $\widehat{\phi}$ is a permutation
    of $S$. For every $x\in S$, we have
    \[
        \bigl(x\widehat{\phi}\bigr)x
        =\bigl(x(x\phi)\bigr)x
        =x(x\phi)
        =x\widehat{\phi},
    \]
    using commutativity and idempotence. Hence
    \[x\widehat{\phi}\leq x
        \tag{34}
    \]
    for every $x\in S$.

    We now use the finiteness of $S$. Fix $x\in S$. Since
    $\widehat{\phi}$ is a permutation of the finite set $S$, the element
    $x$ lies in a finite cycle of $\widehat{\phi}$. Thus, for some
    positive integer $m$,
$x{\widehat{\phi}}^{^{^{m}}}=x$.
    Applying (34) successively gives
    \[
        x
        \geq x\widehat{\phi}
        \geq x\widehat{\phi}^{\,2}
        \geq\cdots
        \geq x\widehat{\phi}^{\,m}
        =x.
    \]
    By antisymmetry of the partial order, every inequality in this chain
    is an equality. In particular,
    \[
        x\widehat{\phi}=x.
    \]
    Since $x$ was arbitrary, $\widehat{\phi}$ is the identity
    permutation. Therefore
    \[
        x(x\phi)=x
    \]
    for every $x\in S$, which, by the definition of the natural order,
    means that
    \[
        x\leq x\phi
        \tag{35}
    \]
    for every $x\in S$.

    Finally, $\phi$ is itself a permutation of the finite set $S$. Fix
    $x\in S$, and choose a positive integer $n$ such that
    \[
        x\phi^n=x.
    \]
    Applying (35), we get
    \[
        x
        \leq x\phi
        \leq x\phi^2
        \leq\cdots
        \leq x\phi^n
        =x.
    \]
    Antisymmetry again shows that all these elements are equal. In
    particular,
    \[
        x\phi=x.
    \]
    
    Thus a finite semilattice has exactly one complete mapping, namely
    the identity.
\end{proof}
\section{Strong complete mappings}
\label{Sec:strong}

A bijection $\alpha\colon S\to S$ of a semigroup $S$ is a \emph{strong complete mapping} if $\alpha$ is both a complete mapping and an orthomorphism.
Evans \cite[Theorem~15]{evans3} showed that a finite abelian group admits a strong complete mapping if and only if both its Sylow $2$-subgroup and its Sylow $3$-subgroup are either trivial or non-cyclic. Akhtar and Gagola \cite[Abstract]{ag}
showed that every non-cyclic $3$-group admits a strong complete mapping, except possibly those in a certain infinite family. For a useful survey of strong complete mappings in groups, current up to 2013, see \cite{evans4}.

\begin{theorem}\label{Thm:strong-CR}
Let $S$ be a finite semigroup with a strong complete mapping $\alpha\colon S\to S$ and orthomorphism $\theta\colon S\to S,\quad x\mapsto x\cdot x\alpha$.
Then $S$ is completely regular. Further, if $x\mapsto x\inv\in V(x)$ is the commuting inverse, then $x\inv\cdot x\theta = x\alpha$
for all $x\in S$.
\end{theorem}
\begin{proof}
Since $S$ is finite, $\alpha$ has finite order, say, $k$.

By Theorem~\ref{Thm:regular}, the semigroup $S$ is regular, and by Theorem~\ref{Thm:good_inverse} we may choose an inverse mapping $x\mapsto x'\in V(x)$ satisfying $x'\cdot x\theta=x\alpha$ for every $x\in S$. Since $\alpha$ is a strong complete mapping, it is an orthomorphism $x\mapsto x\cdot x\beta$ for some complete mapping $\beta\colon S\to S$.

We claim that for all $j>0$,
\begin{equation}\label{Eqn:x'xxaj}
x'x\cdot x\alpha^j = x\alpha^j.
\end{equation}
Indeed, first we have $x'x\cdot x\alpha = x'xx'\cdot x\theta = x'\cdot x\theta = x\alpha$,
taking care of the case $j=1$. Assuming the claim for $j\ge 1$, we have $x'x\cdot x\alpha^{j+1} = x'x\cdot x\alpha^j \cdot x\alpha^j\beta =
x\alpha^j\cdot x\alpha^j\beta = x\alpha^{j+1}$, using the induction hypothesis in the second equality. This establishes the claim. Now in \eqref{Eqn:x'xxaj}, take $j=k$, the
order of $\alpha$. Then $x'xx = x$ for all $x\in S$. By Lemma~\ref{Lem:x'xx}, $S$ is completely regular.

For the remaining assertion, for all $x\in S$ we compute
\[
x\inv\cdot x\theta = x\inv x\cdot x\alpha = \underbrace{x\inv x}x\cdot x\beta = xx\inv x\cdot x\beta = x\cdot x\beta = x\alpha.
\]
This completes the proof.
\end{proof}

\section{Products of all elements}
\label{Sec:products}

The Hall--Paige conjecture (Theorem~\ref{Thm:HP}) is sometimes stated in an extended form:

\begin{theorem}\label{Thm:HPextended}
For a finite group $G$, the following are equivalent: 
\begin{enumerate}
\item[(A)] $G$ admits a complete mapping; 
\item[(B)] the Sylow $2$-subgroups of $G$ are trivial or non-cyclic;
\item[(C)] there exists an ordering $g_1,\ldots,g_n$ of all elements of $G$ such that $g_1 g_2 \cdots g_n = 1$.
\end{enumerate}
\end{theorem}

{\noindent Paige proved (A) $\Rightarrow$ (C) \cite[Theorem~1 and its corollary]{Paige}. Hall and Paige proved that a non-trivial cyclic Sylow $2$-subgroup obstructs (A), which is exactly (A) $\Rightarrow$ (B); see \cite[Theorem~5]{evans4}. Vaughan-Lee and Wanless found a short direct proof that (B) $\Rightarrow$ (C) \cite[Theorem~2.1]{VLW}. Thus a proof that (C) $\Rightarrow$ (A) would have resolved the conjecture, but the proof eventually found was of (B) $\Rightarrow$ (A).}

In this section, we examine analogues of condition (C) for semigroups.

\begin{lemma}\label{Lem:prod_idem}
Let $S$ be a finite semigroup with a complete mapping $\alpha\colon S\to S$ and corresponding orthomorphism
$\theta\colon S\to S,\quad x\mapsto x\cdot x\alpha$. For $c\in S$, let $\{c,c\theta,\ldots,c\theta^{m-1}\}$ be the orbit of $\theta$ through $c$. Then $c\alpha\cdot c\theta\alpha \cdots c\theta^{m-1}\alpha$ is an idempotent in the same $\mathcal{L}$-class as $c$.
\end{lemma}
\begin{proof}
By Theorem~\ref{Thm:regular}, $S$ is regular and by Theorem~\ref{Thm:good_inverse}, we may choose an inverse mapping $x\mapsto x'\in V(x)$ such that $x'\cdot x\theta = x\alpha$ for all $x\in S$. Then
\begin{align*}
c\alpha\cdot c\theta\alpha \cdots c\theta^{m-1}\alpha &= c'\cdot \underbrace{c\theta \cdot c\theta\alpha} \cdots c\theta^{m-1}\alpha \\
&= c'\cdot c\theta^2 \cdots c\theta^{m-1}\alpha \\
&= \cdots \\
&= c'\cdot c\theta^{m-1}\cdot c\theta^{m-1}\alpha \\
&= c'\cdot c\theta^m \\
&= c'c.
\end{align*}
This is an idempotent $\mathcal{L}$-related to $c$.
\end{proof}

A semigroup $S$ is said to be an $E$-\emph{semigroup} if the set $E(S)$ of idempotents is a subsemigroup of $S$. A regular $E$-semigroup is said to be \emph{orthodox}.

\begin{theorem}\label{Thm:E-order}
Let $S$ be a finite $E$-semigroup with a complete mapping. Then there exists an ordering $c_1,\ldots,c_n$ of all elements of $S$ such that $c_1\cdots c_n$ is an idempotent.
\end{theorem}
\begin{proof}
Let $\mathcal P=\{C_1,\ldots,C_{\ell}\}$ be the partition of $S$ given by the orbits of $\theta$. Since $\alpha$ is a permutation,
$\mathcal P\alpha=\{C_1\alpha,\ldots,C_{\ell}\alpha\}$ is also a partition. For each $i$, choose $a_i\in C_i$ and write
$C_i=\{a_i,a_i\theta,\ldots,a_i\theta^{m_i-1}\}$. By Lemma~\ref{Lem:prod_idem}, $a_i\alpha\cdot a_i\theta\alpha\cdots a_i\theta^{m_i-1}\alpha$ is an idempotent.
Thus \(a_1\alpha\cdot a_1\theta\alpha\cdots a_1\theta^{m_1-1}\alpha\cdots a_{\ell}\alpha\cdot a_{\ell}\theta\alpha\cdots a_{\ell}\theta^{m_{\ell}-1}\alpha\) is a product of idempotents, and hence an idempotent since $S$ is an $E$-semigroup.
\end{proof}

We remark that for semigroups with zero this conclusion can be uninformative: nothing rules out the possibility that the product of elements
given by Theorem~\ref{Thm:E-order} is zero, even if restricted to the non-zero elements. For instance, for the
$3$-element commutative idempotent semigroup $S=\{0,1,2\}$ with $1\cdot2=0$, let $\alpha=\theta=\idmap_S$ and note that
$\mathcal P\alpha = \{\{0\},\{1\},\{2\}\}$.

\section*{Acknowledgements}
We thank Persi Diaconis for calling our attention to \cite{Mittenthal} and to the enormous impact it had in industry.

The first author was partially supported by the Funda\c{c}\~ao para a
Ci\^encia e a Tecnologia (FCT) through projects UID/00297/2025 and
UID/PRR/00297/2025, and by CEMAT-Ci\^{e}ncias FCT through projects
UIDB/04621/2020 and UIDP/04621/2020.

The second author was partially supported by the Funda\c{c}\~ao para a
Ci\^encia e a Tecnologia (FCT) through projects UID/00297/2025 and
UID/PRR/00297/2025.

The fourth author was supported by the Australian Research Council.

\section{Problems}
\label{Sec:problems}

{
The structural proof of Theorem~\ref{t:linear-monoid} leads to the following
problem.

\begin{problem}\label{p:matrix-monoids}
Give explicit complete mappings for every proper principal factor of the full
linear monoid. When $M_n(\mathbb F_q)$ has a complete mapping, construct one
without using the general existence theorems for Rees $0$-matrix semigroups.
In particular, replace the anchored constructions for the rank-$1$ factors
over fields of odd order and the rank-$2$ factors over $\mathbb F_2$ by
explicit formulas or efficient algorithms.
\end{problem}
}

P. T. Bateman \cite[Theorem]{Bateman} proved that every infinite group has a complete mapping. For semigroups, the existence problem should depend on ideal structure and the Green classes, not only on maximal subgroups.

\begin{problem}\label{p:infinite-semigroups}
Develop a theory of complete mappings for infinite semigroups. In particular, determine necessary and sufficient conditions for infinite regular semigroups, inverse semigroups, Rees matrix semigroups and Rees $0$-matrix semigroups to have complete mappings. Is there any analogue of the reduction to principal factors in the infinite case?
\end{problem}

For finite groups, complete mappings which are automorphisms have been studied by Bors~\cite{bors}. Let $S$ be a semigroup admitting a complete mapping. This does not imply that
$\End(S)$ or $\Aut(S)$ admits a complete mapping. For example, if
$S=C_2\times C_2$, then $\Aut(S)\cong S_3$, and $S_3$ has no complete mapping.
Since $S_3$ is also the group of units of the endomorphism monoid of $S$, it
follows that $\End(S)$ cannot have a complete mapping either.

\begin{problem}
Characterize finite semigroups $S$ according to the four possible combinations
of existence and non-existence of complete mappings for $S$ and for $\End(S)$.
\end{problem}

\begin{problem}\label{p:resolvable-transversals}
{Determine which finite Rees matrix semigroups have a resolution
of their Cayley table into disjoint transversals. For group tables, the
transversals in such a resolution are the symbol classes of an orthogonal
mate. In particular, decide whether the existence of one complete mapping
imposes any structural condition that forces such a resolution in non-group
Rees matrix semigroups.}
\end{problem}

The converse of Theorem~\ref{Thm:E-order} leads to the following question.
An affirmative answer would generalize the Hall--Paige conjecture.

\begin{problem}
Let $S$ be an orthodox semigroup with an ordering $c_1,\ldots,c_n$ of all elements of $S$ such that $c_1\cdots c_n$ is an idempotent. Must $S$ have a complete mapping?
\end{problem}

\begin{problem}\label{p:permutation-structure}
Let $S$ be a finite semigroup admitting a complete mapping $f$, and let $\theta$ be the corresponding orthomorphism. Determine which pairs of permutations $(f,\theta)$ can occur. In particular, study the possible cycle structures of $f$, of $\theta$, and of $f^{-1}\theta$, and determine how these cycle structures reflect the Green--Rees structure of $S$.
\end{problem}

\begin{problem}\label{p:orthogonal-semigroup-cliques}
Let $\mathbb S_n$ be the set of semigroup operations on a fixed $n$-element set. Say that two operations $\cdot$ and $\ast$ are orthogonal if the map
\((x,y)\longmapsto (x\cdot y,x\ast y)\)
from $X^2$ to $X^2$ is a bijection. Determine the maximum size of a family of pairwise orthogonal semigroup operations on an $n$-element set, and identify the extremal families.
\end{problem}

It is well known that a finite group $G$ has odd order if and only if the identity is a complete mapping of $G$. This simple idea can be extended to the context of complete mappings as follows.

\begin{definition}
Let $S$ be a semigroup and let $f\colon S\to S$ be a bijection. We call $f$ a $k$-complete mapping if
\(x\mapsto x\cdot(xf)^k\)
is a bijection.
\end{definition}

In this terminology, a group $G$ has odd order if and only if the identity is a $1$-complete mapping. Similarly, in a group $G$, the power map
\(\pi_k\colon G\to G,\quad x\mapsto x^k,\)
is a bijection if and only if $\gcd(k,|G|)=1$ if and only if the identity mapping is a $(k-1)$-complete mapping of $G$.

\begin{problem}
Let $C$ be a class of semigroups, for instance groups, Clifford semigroups,
inverse semigroups, regular semigroups, Rees $0$-matrix semigroups, or natural
classes of transformation semigroups. Characterize the finite semigroups $S$ in
$C$ that have a $k$-complete mapping; where possible, determine the
corresponding infinite cases as well.
\end{problem}

\begin{problem}\label{p:Rees0-full-classification}
{Determine the exact obstruction obtained by combining the
Hall--Paige obstruction in the maximal subgroup with the pattern and the
sandwich entries of $P$, and thereby give necessary and sufficient conditions
for a finite Rees $0$-matrix semigroup $\mathcal M^0(G,I,\Lambda,P)$ to
have a complete mapping without assumptions on the complete mappings of $G$
or on the parities of $I$ and $\Lambda$.}
\end{problem}

\begin{problem}\label{p:complexity-complete-mapping}
Determine the computational complexity of deciding, from the Cayley table of a
finite semigroup $S$, whether $S$ has a complete mapping. Determine the
complexity of the same problem when $S$ is restricted to be regular, inverse,
completely regular, completely $0$-simple, or aperiodic.
\end{problem}

{
We assume the reader to be familiar with the Brauer monoid $\mathcal B_n$
and the partial Brauer monoid $\mathcal{PB}_n$; their definitions are given
in \cite[Section~2.2]{EastMitchellRuskucTorpey}. In
\cite[Proposition~2.1 and Remark~2.2(v)--(vii)]
{EastMitchellRuskucTorpey} it is proved that their $\mathcal J$-classes are
the rank classes and that a maximal subgroup in rank $r$ is isomorphic to
$S_r$.

\begin{problem}\label{p:diagram-monoids}
Classify complete mappings for $\mathcal B_n$ and $\mathcal{PB}_n$. In view
of the Hall--Paige theorem, the unresolved proper factors are those of ranks
$2$ and $3$ which occur in the corresponding monoid. Determine also for
which values of $n$ the whole monoid has a complete mapping.
\end{problem}
}

\begin{problem}\label{p:other-algebras}
Let $\mathcal V$ be a class, or a variety, of finite algebras equipped with a
distinguished finite family $\mathcal T$ of binary basic operations or binary
term operations. Call a permutation $\alpha$ of the underlying set of
$A\in\mathcal V$ simultaneously $\mathcal T$-complete if, for every
$t\in\mathcal T$, the map
\(x\mapsto t(x,x\alpha)\)
is a permutation of $A$. Develop structural, enumerative and algorithmic
criteria for the existence of such permutations in natural classes which do not
reduce to a single semigroup or quasigroup reduct, such as finite lattices,
finite rings considered with both addition and multiplication, Lie algebras,
Jordan algebras and other finite non-associative algebras. Determine which
parts of the Hall--Paige theory and which parts of the Hall matching theory
survive in this broader setting.
\end{problem}

Mittenthal's use of orthomorphic mappings of elementary abelian $2$-groups
as block substitutions suggests asking whether the additional freedom
provided by the ideal structure, principal factors, and Rees sandwich
matrices of finite semigroups can be exploited cryptographically
~\cite[pp.~59--60]{Mittenthal}.

\begin{problem}
Develop a theory of cryptographic constructions based on complete mappings
of finite semigroups.  In particular, determine whether there exist natural
infinite families of finite semigroups $S$ with efficiently computable
complete mappings $\alpha$ for which $\alpha$, the associated orthomorphism
\(\theta:S\to S,\quad x\mapsto x\cdot x\alpha,\)
or constructions derived from the pair $(\alpha,\theta)$ yield
cryptographic permutations, substitution boxes, or related primitives with
security or implementation properties unavailable, or not as efficiently
attainable, from complete mappings of groups of comparable order.  Which
semigroup theoretic features provide useful additional freedom, and which
instead produce invariant structures that can be exploited
cryptanalytically?
\end{problem}

\end{document}